\documentclass{amsart}
\usepackage{amsmath}
\usepackage{amsthm,amssymb,amsfonts,hyperref}
\usepackage{amsmath,ifthen,srcltx,amsthm,amsopn,amssymb,amsfonts,color}

\newcommand{\showcomments}{yes}

\newsavebox{\commentbox}
{\ifthenelse{\equal{\showcomments}{yes}}%
{\footnotemark
        \begin{lrbox}{\commentbox}
        \begin{minipage}[t]{1.25in}\raggedright\sffamily\tiny
        \footnotemark[\arabic{footnote}]}
{\begin{lrbox}{\commentbox}}}%
{\ifthenelse{\equal{\showcomments}{yes}}%
{\end{minipage}\end{lrbox}\marginpar{\usebox{\commentbox}}}
{\end{lrbox}}}

\title[Growth of $L^2$-invariants for sequences of lattices in Lie groups]{On the growth of $L^2$-invariants for sequences of lattices in Lie groups
}

\author[Abert, Bergeron, Biringer, Gelander, Nikolov, Raimbault, Samet]{Miklos Abert, Nicolas Bergeron, Ian Biringer, Tsachik Gelander, Nikolay Nikolov, Jean Raimbault and Iddo Samet}

\address {Renyi Institute of Mathematics \\
13-15 Realtanoda utca, 
1053 Budapest, Hungary\\}
\email{abert.miklos@renyi.mta.hu}

\address{Sorbonne Universit\'es, UPMC Univ Paris 06, Institut de Math\'ematiques de Jussieu--Paris Rive Gauche, UMR 7586, CNRS, Univ Paris Diderot, Sorbonne Paris Cit\'e, F-75005, Paris, France}
\email{nicolas.bergeron@imj-prg.fr}
\urladdr{http://people.math.jussieu.fr/~bergeron}

\address {Boston College \\
Mathematics Department \\
 140 Commonwealth Ave.
Chestnut Hill, MA  02467-3806\\}
\email{ianbiringer@gmail.com}

\address{Einstein Institute of Mathematics\\
Edmond J. Safra Campus, Givat Ram\\
The Hebrew University of Jerusalem\\
Jerusalem, 91904, Israel\\}
\email{tsachik.gelander@gmail.com}

\address{University College, \\ Oxford \\ OX1 4BH, UK. }
\email{zarkuon@gmail.com}

\address{Institut de Math\'ematiques de Toulouse ; UMR5219 \\ Universit\'e de Toulouse ; CNRS \\ UPS IMT, F-31062 Toulouse Cedex 9, France}
\email{Jean.Raimbault@math.univ-toulouse.fr}

\address{
University of Illinois at Chicago\\
Department of Mathematics, Statistics, and Computer Science\\
Chicago, IL 60607\\}
\email{samet@math.uic.edu}

 \DeclareFontFamily{OT1}{rsfs}{}
 
\DeclareFontShape{OT1}{rsfs}{n}{it}{<-> rsfs10}{}
\DeclareMathAlphabet{\mathscr}{OT1}{rsfs}{n}{it}

\newcommand{\BE}{{\mathbb{E}}}
\newcommand{\BN}{{\mathbb{N}}}
\newcommand{\BZ}{\mathbb{Z}}

\newcommand{\BC}{\mathbb{C}}
\newcommand{\BR}{\mathbb{R}}

\newcommand{\BH}{\mathbb{H}}

\newcommand{\Tr}{\mathrm{tr}\,}

\newcommand{\tx}{\tilde{x}}

\newcommand{\actson}{\curvearrowright}

\newcommand{\eps}{\varepsilon}

\newcommand{\C}{\mathbb{C}}

\newcommand{\Z}{\mathbb{Z}}

\newcommand{\p}{\mathfrak{p}}
\newcommand{\g}{\mathfrak{g}}

\renewcommand{\H}{\mathbf{H}}

\DeclareFontFamily{OT1}{rsfs}{}

\DeclareFontShape{OT1}{rsfs}{n}{it}{<-> rsfs10}{}
\DeclareMathAlphabet{\mathscr}{OT1}{rsfs}{n}{it}
\newcommand{\Q}{\mathbb{Q}}

\newcommand{\R}{\mathbb{R}}
\newcommand{\SO}{\mathrm{SO}}

\newcommand{\Sp}{\mathrm{Sp}}

\newcommand{\gD}{\Delta}
\newcommand{\gd}{\delta}
\newcommand{\gb}{\beta}

\newcommand{\gC}{\Gamma}
\newcommand{\gc}{\gamma}

\newcommand{\gS}{\Sigma}

\newcommand{\gep}{\epsilon}

\newcommand{\ga}{\alpha}
\newcommand{\gt}{\tau}

\newcommand{\ti}[1]{\tilde{#1}}
\newcommand{\vol}{\mathrm{vol}}

\newcommand\probability{\operatorname{Prob}}

\swapnumbers
\newtheorem{thm}[subsection]{Theorem}
\newtheorem{lem}[subsection]{Lemma}
\newtheorem*{lem*}{Lemma}
\newtheorem{prop}[subsection]{Proposition}
\newtheorem*{prop*}{Proposition}

\newtheorem{exa}[subsection]{Example}
\newtheorem{conj}[subsection]{Conjecture}

\newtheorem{cor}[subsection]{Corollary}

\theoremstyle{definition}
\newtheorem{defn}[subsection]{Definition}
\newtheorem*{rem}{Remark}

\numberwithin{equation}{subsection}

\renewcommand{\H}{\mathbb H}  %

\newcommand{\cal}{\mathcal}
\newcommand{\GL}{\mathrm{GL}}
\newcommand{\SL}{\mathrm{SL}}

\newcommand{\SU}{\mathrm{SU}}

\newcommand{\sub}{\mathrm{Sub}}

\newcommand {\comment} [1] {}

\begin{document}

\begin{abstract}
We study the asymptotic behaviour of Betti numbers, twisted torsion and other spectral invariants of sequences of locally symmetric spaces. 
Our main results are uniform versions of the DeGeorge--Wallach Theorem, of a theorem of Delorme and various other limit multiplicity theorems.

A basic idea is to adapt the notion of Benjamini--Schramm convergence (BS-convergence), originally introduced for sequences of finite graphs of 
bounded degree, to sequences of Riemannian manifolds, and analyze the possible limits. We show that BS-convergence of locally symmetric 
spaces $\gC\backslash G/K$ implies convergence, in an appropriate sense, of the normalized relative Plancherel measures associated 
to $L^2 (\gC\backslash G)$.
This then yields convergence of normalized multiplicities of unitary representations, Betti numbers and other spectral invariants.
On the other hand, when the corresponding Lie group $G$ is simple and of real rank at least two, we prove that there is only one 
possible BS-limit, i.e. when the volume tends to infinity, locally symmetric spaces always BS-converge to their universal cover 
$G/K$. This leads to various general uniform results.

When restricting to arbitrary sequences of congruence covers of a fixed arithmetic manifold we prove a strong quantitative version of 
BS-convergence which in turn implies upper estimates on the rate of convergence of normalized Betti numbers in the spirit of Sarnak--Xue. 

An important role in our approach is played by the notion of Invariant Random Subgroups. For higher rank simple Lie groups $G$, we exploit
rigidity theory, and in particular the Nevo--St\"{u}ck--Zimmer theorem and Kazhdan`s property (T), 
to obtain a complete understanding of the space of IRSs of $G$.

\end{abstract}
\maketitle
\tableofcontents

\section{Introduction and statement of the main results}

Let $G$ be a connected center-free semi-simple Lie group without compact factors, $K\le G$ a maximal compact subgroup and $X=G/K$ the associated Riemannian symmetric space. The main results of this paper concern the asymptotic of $L^2$-invariants of the spaces $\Gamma\backslash X$, where $\Gamma$ varies over the space of lattices of $G$.

Most of our results rely on the notion of \emph {Benjamini--Schramm convergence}, or \emph {BS-convergence}, for sequences of locally symmetric spaces $\Gamma_n\backslash X$. We start by introducing a particularly transparent case:  when $\Gamma_n \backslash X $ BS-converges to $X$.

\begin{defn}\label{defn:conv-to-G}
Let $(\Gamma_n)$ be a sequence of lattices in $G$. We say that the $X $-orbifolds $M_n=\Gamma_n\backslash X$ {\it BS-converge} to $X$  if for every $R>0$, the probability that the $R$-ball centered around a random point in $M_n$ is isometric to the $R$-ball in $X$ tends to $1$ when $n\to \infty$; i.e.\ for every $R>0$, we have 
$$\lim_{n \to +\infty} \frac{\vol ((M_n)_{<R})}{\vol (M_n)} = 0,$$
where $M_{<R} = \{ x \in M \; : \; \mathrm{InjRad}_M (x) < R \}$ is the $R$-thin part of $M$.
\end{defn}

A straightforward and well studied example is when $\Gamma \leq G $ is a uniform lattice and $\Gamma_n \leq \Gamma $ is a chain of normal 
subgroups with trivial intersection; in this case, the $R $-thin part of $\Gamma_n \backslash X $ is empty for large enough $n $.

\medskip
\noindent \textbf{General BS-convergence.} The definition above fits into a more general notion of convergence, adapted from that introduced by Benjamini and Schramm \cite {localconvergence} for sequences of bounded degree graphs. 

Consider the space $\mathcal M $ of  pointed, proper metric spaces, endowed with the  pointed Gromov--Hausdorff topology. Each $\Gamma_n \backslash X$ can be turned into a probability measure on $\mathcal M $ by choosing the  basepoint at random with respect to volume; this measure is supported on  pointed spaces isometric to $\Gamma_n \backslash X$.  We say that $\Gamma_n\backslash X $ \emph {BS-converges} if these measures weakly converge.  The limit object is then a probability measure on $\mathcal M $. This  perspective is elaborated on in Section \ref {sec:3}. 

Most of the results of this paper assume (or prove) BS-convergence to $X$. These results can often be extended to general BS-convergent sequences, but they tend to get more technical and sometimes further assumptions are needed, they will appear in a sequel of this paper to be extracted from our original arXiv paper \cite{7s}. 

This definition of BS-convergence is very broad and works just as well for sequences of finite volume Riemannian manifolds. In our situation, the common ambient group $G$ allows a useful algebraic reformulation of BS-convergence where probability measures on $\mathcal M$ are replaced by \emph {invariant random subgroups} of $G $, i.e.\ $G$-invariant measures on the space of closed subgroups of $G$.  This reformulation is what we use in most of the paper. This will be discussed at the end of the Introduction and in Sections \ref {sec:2} and \ref {sec:3}.

\medskip
\noindent \textbf{Uniform discreteness.} A family of lattices (resp.\ the associated $X$-orbifolds) is {\it uniformly discrete} if there is an identity neighborhood in $ G$ that intersects trivially all of their conjugates. For torsion-free lattices $\Gamma_n$, this is equivalent to saying that there is a uniform lower bound for the injectivity radius of the manifolds $M_n=\Gamma_n\backslash X$. (So in particular, a uniformly discrete family of  lattices consists only of uniform lattices.) 

Any family $(M_n)$ of covers of a fixed compact orbifold is uniformly discrete.
Margulis has conjectured \cite[page 322]{margulis:book}
(see also \cite[Section 10]{Gel:HV}) that the family of all cocompact torsion-free arithmetic lattices in $G$ is uniformly discrete.  This is a weak form of the famous Lehmer conjecture on monic integral polynomials.

\medskip
\noindent \textbf{BS-convergence and Plancherel measure.} Our first result says that BS-convergence to $X$ implies a spectral convergence: namely, when $(\Gamma_n)$ is uniformly discrete, the relative Plancherel measure of $\Gamma_n \backslash G$ will converge to the Plancherel measure of $G$ in a strong sense.  

For an irreducible unitary representation $\pi\in \widehat G$ and a uniform lattice $\Gamma$ in $G$ 
let $\mathrm{m}(\pi,\Gamma)$ be the multiplicity of $\pi$ in the right regular representation $L^2(\Gamma \backslash G)$. 
Define the relative Plancherel measure of $\Gamma \backslash G$ as the measure
$$\nu_{\Gamma} = \frac{1}{\vol (\Gamma \backslash G)} \sum_{\pi \in \widehat{G}} \mathrm{m}(\pi,\Gamma) \delta_{\pi}$$
on $\widehat{G}$. Finally denote by $\nu^G$ the Plancherel measure of the right regular representation $L^2 (G)$. Recall that the support of $\nu^G$ is $\widehat{G}_{\rm temp}$ --- the subset of the unitary dual $\widehat{G}$ which consists
of tempered representations.

\begin{thm}[Theorem \ref{T1}] \label{thm1}
Let $(\Gamma_n)$ be a uniformly discrete sequence of lattices in $G$ such that the spaces $\Gamma_n \backslash X$ BS-converge to $X$. Then for every quasi-compact $\nu^G$-regular open subset $S \subset \widehat{G}$ or $S \subset \widehat{G}_{\rm temp}$, we have: 
$$
\nu_{\Gamma_n} (S) \to \nu^G (S).
$$
\end{thm}
Note that the Plancherel measure of $G$ depends on a choice of a Haar measure on $G$ as does $\vol(\Gamma \backslash G)$. 
We recall basic facts on the topology of $\widehat{G}$ in Section \ref{sec:7}.

Let $d(\pi)$ be the `multiplicity' --- or rather the formal degree --- of $\pi$ in the regular representation $L^2(G)$ with respect to the Plancherel measure of $G$. Thus, $d(\pi)=0$ unless $\pi$ is a discrete series representation. 
Theorem \ref{thm1} implies the following:

\begin{cor}\label{thm1bis}
Let $(\Gamma_n) $ be a uniformly discrete sequence of lattices in $G$ such that the spaces $\Gamma_n \backslash X $ BS-converge to $X$. Then for all $\pi \in \widehat G$, we have 
$$
 \frac{m(\pi,\Gamma_n)}{\vol(\Gamma_n \backslash G)}\to d(\pi).
$$
\end{cor}

In the special situation when $(\Gamma_n)$ is a chain of normal subgroups with trivial intersection in some fixed cocompact lattice $\Gamma\le G$, Corollary \ref{thm1bis} is the classical theorem of DeGeorge and Wallach \cite{DeGeorgeWallach}. In that very same situation Theorem \ref{thm1} is due to Delorme 
\cite{Delorme}. Since the pioneering work of DeGeorge and Wallach, `limit formulas' have been the subject of extensive studies. Two main directions of improvement have been considered. 

The first direction is concerned with the extension of the theorems of DeGeorge--Wallach and Delorme to 
non-uniform lattices. In the case of the DeGeorge--Wallach theorem we refer to \cite{DeGeorge,BarbaschMoscovici,Clozel,RohlfsSpeh,Savin}. 
Note that these works were partially motivated by a question of Kazhdan \cite{Kazhdan} pertaining to his work on the field of definition of arithmetic varieties. 
The limit multiplicity problem for the entire unitary dual has been solved for the standard congruence subgroups of $\SL_2 (\Z)$ by Sarnak in \cite{Sarnak} (see also \cite{Iwaniec1,DeitmarHoffman}) but is still open in general. A partial result for certain normal towers of congruence arithmetic lattices defined by groups
of $\Q$-rank one has been shown in \cite{DeitmarHoffman}. Very recently important progress have been made by 
Finis, Lapid and M\"uller \cite{FLM} who can deal with groups of arbitrary rank. In these works the authors usually deal with towers of normal subgroups. 

A second direction is to extend the theorems of DeGeorge--Wallach and Delorme to more general sequences of (uniform) lattices. This has been addressed in some of the above mentioned works for certain (non-principal) congruence subgroups of a fixed lattice, such as $\Gamma_0 (N)$, see also \cite{Iwaniec2} for another example. Theorem \ref{thm1} is the first example where one can deal with sequences of non-commensurable lattices. 

The classical theorem of DeGeorge and Wallach implies a corresponding statement on the approximation of 
$L^2$-Betti numbers by normalized Betti numbers of finite covers (see also Donnelly \cite{Donnellytowers}). 
Theorem \ref{thm1} implies the following uniform version of it. 

\begin{cor}\label{cor2}
Let $(\Gamma_n)_{n\geq 1}$ be a uniformly discrete sequence of uniform lattices in $G$ such that $\Gamma_n \backslash X$ BS-converges to $X$. Then for every $k\le \dim(X)$ we have
$$\frac{b_{k}(\Gamma_{n})}{\mathrm{vol}(\Gamma_n \backslash X)} \to \beta
_{k}^{(2)}(X).
$$
\end{cor}

In the corollary, $b_k (\Gamma_n) $ is the $k ^ {\text {th}} $ Betti number of the (virtually torsion-free) group $\Gamma_n $,\footnote{The group $\Gamma_n $ being virtually torsion-free, the orbifold $\Gamma_n \backslash X$ is finitely covered by a manifold whose $\Gamma_n$-invariant rational $k$-th cohomology group coincides with the rational $k$-th orbifold cohomology of $\Gamma_n \backslash X$ and is of finite rank $b_k (\Gamma_n )$; in particular if $\Gamma_n$ is torsion-free $b_k (\Gamma_n)$ is the $k$-th Betti number of $\Gamma_n \backslash X$.} and
$$
 \beta_k ^ {(2)}(X) = \begin {cases} \frac{\chi (X^d)}{\vol (X^d)} & k = \frac12 \dim X \\ 0 & \text {otherwise},\end {cases}
 $$
is the $k ^ {\text {th}} $ {\it $L^2$-Betti number of $X$,}
where $X ^d $ is the compact dual of $X$ equipped with the Riemannian metric induced by the Killing form on $\mathrm{Lie}(G)$.    We refer the reader to \S\ref{8.2} for an analytic definition of $\beta_k ^ {(2)}(X)$. By \cite {Allday} and \cite {Papadima}, the Euler characteristic $\chi (X ^ d) $ is nonzero exactly when the \emph {fundamental rank}
$$
 \delta (G) = \C\mbox{-rank} (G) - \C\mbox{-rank} (K)
$$
of $G $ is zero.  Alternatively, it follows from the equality of the Euler characteristic and its $L ^ 2 $-analogue that in the middle dimension, $\beta^ {(2)}_k (X) \neq 0$ if and only if the Euler characteristic of some (or, equivalently, every) closed $X$-manifold is nonzero.

\medskip
\noindent \textbf{Uniform BS-convergence in higher rank.} In the higher rank case we have the following remarkable phenomenon, that gives a surprisingly strong result when combined with Theorem \ref{thm1}. Note that in the following result we do not restrict to the case where the $\Gamma_n$ are cocompact and in particular, we do not assume uniform discreteness.

\begin{thm}[Corollary \ref{cor:4.15}] \label{thm3}
Suppose that $G$ has property $(T)$ and real rank at least $2$. Let $\Gamma_n\le G$ be any sequence of pairwise non-conjugate irreducible lattices in $G$. Then $\Gamma_n\backslash X$ BS-converges to $X$.
\end{thm}

\begin{cor}\label{cor4}
If in addition to the conditions of Theorem \ref {thm3} we have that $(\Gamma_n)$ is uniformly discrete (in particular, cocompact), then for every quasi-compact $\nu^G$-regular subset $S \subset \widehat G$, we have: 
$$
 \nu_{\Gamma_n} (S) \to \nu^G (S),
$$
and in particular,
$$
 \frac{m(\pi,\Gamma_n)}{\vol(\Gamma_n \backslash X)}\to d(\pi)
$$
for any $\pi\in\widehat G$. And even more particularly, we have: 
$$
 \frac{b_k(\Gamma_n)}{\vol(\Gamma_n \backslash X)}\to \beta^ {(2)}_k (X)
$$
for every $k\le \dim(X)$.
\end{cor}

Here is a particular example to illustrate the strength of Corollary \ref{cor4}:

\begin{exa}
Let $n\ge 3$, let $\Gamma$ be a cocompact lattice in $\SL_n(\R)$ and let $\Gamma_m\le \Gamma$ be a sequence of distinct, finite index subgroups of $\Gamma$. Then for all $k$,
$$
 \frac{b_k(\Gamma_m)}{[\Gamma:\Gamma_m]}\to 0.
$$
\end{exa}

Even in this example, where all the lattices fall in one commensurability class, we do not see a proof that avoids using Theorem \ref{thm3}. 

\medskip

 It is easy to see that the analogue of Corollary \ref{cor4} --- and therefore of Theorem \ref{thm3} --- is false for some rank one symmetric spaces. For instance, suppose $M$ is a closed hyperbolic $d$-manifold and $\pi _{1}(M)$ surjects onto the free group of rank $2$.
Then finite covers of $M$ corresponding to subgroups of $\Z * \Z$ have first
Betti numbers that grow linearly with the volume. However, for $d\ne 2$, there will be
sublinear growth of the first Betti number in any sequence of covers corresponding to
a chain of finite index normal subgroups of $\pi _{1}(M)$ with trivial
intersection, e.g. by the DeGeorge--Wallach theorem.

\medskip
\noindent \textbf{Removing the injectivity radius condition for hyperbolic manifolds.}  If $\text{rank}(X)\ge 2$ or if $X$ is the symmetric space corresponding to $\Sp(d,1)$ or $F_4^{-20}$, then all irreducible $X$-manifolds are arithmetic, by Margulis's Arithmeticity \cite[Theorem 1.10, p. 298]{margulis:book} and the Corlette--Gromov--Schoen Theorem~\cite{Corlette,Gr-Sc}, respectively. For $\SU(d,1)$ there are few known examples of non-arithmetic manifolds for $d=2,3$, and it is likely that most manifolds are arithmetic. According to Margulis' conjecture it is therefore natural to expect that if $X$ is not isometric to some real hyperbolic space $\BH^d$ ($d\geq 2$), then the family of all irreducible compact $X$-manifolds is uniformly discrete. On the other hand, it is shown in \cite{Agol,BHW,BT} that for every $d\geq2$ there are compact hyperbolic manifolds of dimension $d$ with arbitrarily small closed geodesics. Still, a careful estimate of the norm of the heat kernel in the thin part of rank one manifolds (see Section \ref{sec:thin}) allows us to prove the following. 

\begin{thm}[Theorem \ref{thm:rank-one}] \label{thm5}
Let $M_n=\Gamma_n\backslash \BH^d$ be a sequence of compact hyperbolic $d$-manifolds that BS-converges to $\BH^d$. Then for every $k\le d$,
$$
 \lim_{n \to +\infty} \frac{b_k(M_n)}{\vol(M_n)}=\beta^{(2)}_k(\BH^d).
$$
\end{thm}

Note that for $X=\BH^2$, the hyperbolic plane, Theorem \ref{thm:rank-one} is a consequence of the Gauss--Bonnet theorem, even under the weak assumption that only $\vol(M_n)\to\infty$, without requiring BS-convergence. In general there are many sequences of hyperbolic manifolds that BS-converge to $\BH^d$, but where the global injectivity radius is not bounded below. A typical example is given by Brock--Dunfield \cite{BrockDunfield}, and while these are (intentionally) integer homology spheres, similar examples can be constructed where the only control on the first Betti numbers is through Theorem~\ref{thm5}.

The idea of our argument for Theorem \ref{thm5} also gives an alternative proof, in the real hyperbolic case, of the classical theorem of Gromov that Betti numbers are linearly bounded by volume \cite[Theorem 2]{BGS}. We were not able to perform the same analysis in the higher rank case.  However, assuming the Margulis conjecture, our result for higher rank symmetric spaces (Corollary \ref{cor4}) is much stronger than Gromov's linear bound.\footnote{Recall however that Gromov's theorem applies in the much broader setup of Hadamard spaces with bounded curvature and no Euclidian factors, that we do not consider in this paper.}

\medskip
\noindent \textbf{Explicit estimates for congruence covers.} When restricted to congruence covers of a given arithmetic {\it hyperbolic} manifold, Gromov conjectured that the $k$'th Betti number should be bounded above by a constant times $n^{\alpha}$ where $n$ is the index of the cover and
$$
 \alpha=\frac{2k}{d-1}, \quad 0 \leq k \leq [ (d-1)/2 ],
$$
see Sarnak and Xue \cite{SarnakXue}. Cossutta and Marshall \cite{CossuttaMarshall} and Bergeron, Millson and Moeglin \cite{BMM} proved an even better (and sharp) bound for principal congruence covers of level a power of a prime and small degree $k < d/3$. Our next result is a weak form of Gromov's conjecture. While we cannot approach the precise constant suggested by Gromov, we do obtain a very general result that applies to all semi-simple Lie groups and general congruence (not just principal) subgroups.

\begin{thm}[Theorem \ref{Tcong}] \label{thm6}
Let $G$ be a semi-simple Lie group and let $\Gamma\le G$ be a uniform arithmetic subgroup.
Let $\pi\in\widehat G$ be a non-tempered irreducible representation. Then there are constants $\alpha>0$ and $C<\infty$ such that for every congruence subgroup $\Delta\le\Gamma$, 
$$
 m(\pi,\Delta)\le C\cdot [\Gamma:\Delta]^{1-\alpha}.
$$
\end{thm}

As a consequence we obtain the following:

\begin{cor}\label{cor7}
Let $G$ and $\Gamma$ be as in Theorem \ref{thm6}. Suppose that
$$
 |k-\frac{1}{2}\dim X|>\delta(G).
$$
Then there exist constants $\alpha>0$ and $C$ such that for every congruence subgroup $\Delta\le \Gamma$, we have
$$
 b_k(\Delta)\le C\cdot\vol (\Delta \backslash X)^{1-\alpha}.
$$
\end{cor}

Theorem \ref{thm6} is a consequence of the following result, which is of independent interest.

\begin{thm}[Theorem \ref{ag}]  \label{thm8}
Let $\mathbf{G}$ be a $k$-simple simply connected algebraic group defined over a number field $k$. Let $\mathcal{O}$ be the ring of integers in $k$. There exists a finite index center-free subgroup $\Gamma \subset \mathbf{G} (\mathcal{O})$  and a positive constants $\epsilon$ and $C$ (depending only on $\Gamma$ and some fixed word metric on it) with the following property: 

Let $g \in \Gamma - \{1\}$ and let $H$ be a congruence subgroup of index $N$ in $\Gamma$. Then $g$ fixes at most $e^{C  l(g)} N^{1-\epsilon}$
points in the action of $\Gamma$ on the right cosets $H \backslash \Gamma$ by multiplication. Here $l(g)$ is the length of $g$ with respect to the fixed word metric of $\Gamma$.
\end{thm}

Theorem \ref{thm8} leads to the following effective version (for subgroups of a fixed lattice) of Theorem \ref{thm3}; this allows us to prove Theorem~\ref{thm6}.  Implicit here is an effective proof (again, for subgroups of a fixed lattice) of the second part of Corollary~\ref{cor4}.

\begin{thm}[Theorem \ref{nik}] \label{thm9}
Let $\Gamma_0 \subset G$ be a cocompact arithmetic lattice. Then there exist positive constants $c$ and $\mu$ depending only on $\Gamma_0$, such that for any congruence subgroup $\Gamma \subset \Gamma_0$ and any $R>1$ we have:
$$\mathrm{vol} ((\Gamma \backslash X)_{<R}) \leq e^{cR} \mathrm{vol} (\Gamma \backslash X)^{1-\mu}.$$
\end{thm}

\medskip
\noindent \textbf{Growth of Reidemeister torsion.} When the fundamental rank $\delta (G)$ is positive, the symmetric space $X$ is $L^2$-acyclic. It is then natural to investigate a secondary invariant such as the $L^2$-torsion of $X$, see \cite{LuckBook,BV}. This is known to be non-vanishing if and only if $\delta (G)=1$, e.g. in the case $G= \SL_2 (\C)$. We study $L^2$-torsion for BS-convergent sequences in Section \ref{sec:torsion}; see in particular Theorem \ref{approxthm}. 

In this Introduction, we stress the particular case of compact \emph {orientable} hyperbolic $3$-manifolds. Given such an $M$ we denote by $\alpha_{\rm can}$ the discrete faithful $\SL_2 (\C)$-representation of $\pi_1 M$.
The corresponding twisted chain complex
$$C_* (\widetilde{M}) \otimes_{\Z [\pi_1 M ]} \C^2 $$
is acyclic \cite{Porti} and it follows that the corresponding Reidemeister torsion $$\tau (M, \alpha_{\rm can}) \in \R^*$$ is well defined. The following result is a consequence of Theorem \ref{approxthm}.

\begin{thm} \label{thm7}
Let $(M_n)_n$ be a uniformly discrete sequence of orientable compact hyperbolic $3$-manifolds
which BS-converges toward $\H^3$, then:
$$\lim_{n \rightarrow +\infty} \frac{1}{{\rm vol} (M_n )} \log |\tau (M_n , \alpha_{\rm can})| = - \frac{11}{12\pi}.$$
\end{thm}

\medskip
\noindent \textbf{The role of IRS.} An important tool in our project is the notion of an {\it invariant random subgroup} (IRS). An IRS is a 
conjugacy invariant probability measure on the space $\text{Sub}_G$ of closed subgroups of $G$. We refer the reader to \cite{Miklos2,Bowen1,Vershik,WUD,OW-Lecture} for other recent works that make use of this notion. 

Any lattice $\Gamma\le G$  defines an IRS $\mu_\Gamma$ supported on the conjugacy class $\Gamma^G$. It turns out (see Theorem \ref{discIRS}) that if $G$ is a connected simple Lie group then any non-atomic IRS is supported on discrete subgroups (hereafter called a discrete IRS). Every discrete IRS gives rise to a probability measure on the space of rooted metric spaces  $\mathcal M$  mentioned above, and one can relate weak$^*$ convergence of IRSs to weak$^*$  convergence of measures on $\mathcal M$.   See Section \ref {sec:3} for details.

Denote by $\mu_G$ and $\mu_\text{Id}$ the atomic measures supported on $\{ G\}$ and $\{ \text{Id}_G\}$ respectively.
The following is a variant of Theorem \ref{thm3} stated in the language of IRSs:

\begin{thm}[Theorems \ref{weak} and \ref{weak2}] \label{thm10}
Let $G$ be a connected, center-free higher rank simple Lie group. Then: 
\begin{itemize}
\item The ergodic IRSs are exactly $\mu_G,\mu_\text{Id}$ and $\mu_\Gamma$ where $\Gamma$ is a lattice in $G$.
\item The set of ergodic IRSs is compact and its only accumulation point is $\mu_\text{Id}$.
\end{itemize}
\end{thm}

The first part of Theorem \ref{thm10} is a consequence of the Nevo--St\"{u}ck--Zimmer rigidity theorem \cite{stuckzimmer,Nevogeneralization}. 

The picture is much wilder in rank one. For example, starting with a lattice $\Gamma\le G$ and an infinite index normal subgroup $\Delta\lhd \Gamma$, one can induce the measure on $\Gamma \backslash G$ to an ergodic IRS supported on the conjugacy class $\Delta^G$. More generally, any IRS in $\Gamma $ can be induced to an IRS in $G$. We investigate these constructions and more exotic ones in a sequel of this paper to be extracted from our original arXiv paper \cite{7s}. In particular we define their spectral measure and their $L^2$-Betti numbers and we prove spectral convergence along sequences that BS-converge toward a non-trivial IRS.

\medskip
\noindent \textbf{Acknowledgments.} This research was supported by the MTA Renyi ``Lendulet" Groups and Graphs Research Group, the NSF Postdoctoral Fellowship, the Institut Universitaire de France, the ERC Consolidator Grant 648017, the EPSRC, the ISF grant 1003/11 and the ISF-Moked grant 2095/15.

\section{Invariant Random Subgroups} \label{sec:2}

Let $G$ be a locally compact second countable group.
We denote by $\sub_{G}$ the set of closed
subgroups of $G$. There exists a natural topology on $\sub_{G}$, the
\it Chabauty topology \rm \cite {Chabauty}, which is generated by open sets of the form
\begin{enumerate}
  \item $\mathcal{O}_1(K) = \{ H \in \sub_{G} : H \cap K = \emptyset \}$, for $K \subset G$ compact, and
  \item $\mathcal{O}_2(U) = \{ H \in \sub_{G} : H \cap U \neq \emptyset \}$, for $U \subset G$ open.
\end{enumerate}
 Alternatively, a sequence $(H_n )_{n\geq 0} $ in $ \sub_{G}$ converges to $H \in \sub_{G}$ if and only if:
\begin{enumerate}
\item  For every $x\in H$, there exists a sequence $(x_n) \in G^{\mathbb{N}}$ such that $x_n \in H_n$ and $x_n
\rightarrow x$ in $G$.
\item For every strictly increasing sequence of integers $(n_k)_{k\geq 0}$ and for any converging sequence $x_{n_k} \rightarrow x$ such that $x_{n_k} \in H_{n_k}$, we have $x \in H$.
\end{enumerate}

The Chabauty topology is compact, separable and metrizable \cite[Lemma E.1.1]{Benedetti-Petronio}. While the proof  of metrizability referenced uses Urysohn's theorem,  one can also write down an explicit metric.  For instance, when $G$ is compact, the Chabauty topology is induced by the {Hausdorff metric} on $C(G)$.  In the noncompact case, one can metrize it by integrating up the Hausdorff metrics on all $R$-balls around a fixed base point, see \cite{Abertinvariant}. We refer the reader to \cite{Harpenotes} (and also to \cite{OW-Lecture}) for more information on the topology of Chabauty spaces.

\smallskip

\noindent \it Note: \rm we will  not always require $G$  to be locally compact.  In this case, the Chabauty topology is defined as above, but will not always be compact.
\medskip

 Here is an easy exercise in the definitions that we will use in Section \ref{sec:3}.

\begin {lem}\label {uniformdiscretelemma}
Let $G$ be a connected Lie group and suppose that $ (\Gamma_n)$ is a  sequence  in $\sub_G$ with $\Gamma_n \cap U=\{\mathrm{id}\}$ for some fixed neighborhood $U$of the identity $\mathrm{id}$  in $G$.   If $(\Gamma_n)$ converges toward a group $H$ in $\sub_G$, then $H$ is discrete.   Moreover, if all the $\Gamma_n$'s are torsion free, so is $H$.	
\end {lem}
\begin {proof} If $H$ is not discrete, then it intersects $U$.  So, $\mathcal{O}_2(U\smallsetminus \{\mathrm{id}\})$  is a neighborhood of $H $   that does not contain any $\Gamma_n$.  
 If $g$ is a non trivial element in $H$ with $g^k=\mathrm{id}$, then there is a sequence $\gamma_n\in \Gamma_n$ with $\gamma_n \to g$.  Therefore, $\gamma_n^k \to \mathrm{id}$,  so by uniform discreteness we have $\gamma_n^k=\mathrm{id}$ for large $n$, implying that $\Gamma_n$ has torsion.
\end {proof}

\begin{prop}\label{prop:isolated}
 Let $G$ be a connected Lie group. Then $G$ is an isolated point in $\sub_G$ if and only if $G$ is topologically perfect.
\end{prop}
Recall that a topological group is {\it topologically perfect} if its commutator subgroup is dense.  So in particular, the proposition implies that if $G$ is connected and semi-simple then $ G$ is an isolated point in $\sub_G$.

\begin{proof}
Suppose that $G$ is topologically perfect. Then
by \cite[Theorem 2.1]{BGdense} there exist $d=\dim(G)$ open sets $\Omega_1,\ldots,\Omega_d \subset G$ such that for every choice of $d$ elements $g_i \in \Omega_i,~i=1,\ldots,d$ the subgroup $\langle g_1,\ldots,g_d\rangle$ is dense in $G$. Therefore, if $H \in \sub_G$ intersects each of the $\Omega_i$, then $H=G$, and thus $\cap_{i=1}^d \mathcal{O}_2(\Omega_i) = \{ G \}$ is open.

Conversely, if $G$ is not topologically perfect then it surjects on the circle $S^1$. Let $H_n$ be the pre-image in $G$ of the cyclic group of order $n$ in $S^1$. Then clearly we have that $H_n$ converges to $G$.
\end{proof}

We also recall the following well-known fact, which follows from classical work of Kuranishi \cite{Kuranishi} and Toyama \cite{Toyama}. See also Theorem 4.1.7 in \cite{Th}.

\begin{prop}\label{conv-to-nilpotent}
  Let $G$ be a Lie group and
  let $(\Gamma_n)_{n \geq 1}$ be a sequence of discrete subgroups in $G$ that converges to a subgroup $H$. Then the connected component $H^\circ$ of $H$ is nilpotent.
\end{prop}

We now come to the central definition of this section.  

\begin{defn}
Let $G$ be a topological group. An \emph{invariant random subgroup} (IRS) of $G$ is a $G $-invariant Borel probability
measure on $\sub_{G}$.
\end{defn}

Here, $G $ acts on $\sub_G $ by conjugation.  The name IRS has been coined in \cite{Miklos2}. We consider the set
$$
 \text{IRS}(G)=\text{Prob(Sub}_G)^\text{inv}
$$
of invariant random subgroups of $G$ endowed with the weak-$*$ (or, vaguely speaking, the weak) topology. When $G$ is locally compact, as $\sub_{G}$ is compact and the $G$-action is continuous, it follows from Riesz' representation theorem and Alaoglu's theorem that the space of invariant random subgroups of $G$ is also compact.

IRSs arise naturally when dealing with non-free actions, as the stabilizer of a random point in a probability measure preserving action is an IRS. More precisely, when $G$ acts by measure preserving transformations on a countably separated probability space $(\Omega,\nu)$, the push forward of $\nu$ under the stabilizer map\footnote{It is a result of Varadarajan that the stabilizers are closed subgroups, see \cite[Corollary 2.1.20]{Zimmer}.}
$$\text {stab} : \Omega \longrightarrow \text{Sub}_G,\ \  \text {stab}(x)= G_x$$ is an IRS in $G$. We say that the IRS is \emph {induced} from the p.m.p action.

As an example, suppose that $H$ is a closed subgroup of $G$ such that $G / H$ admits a finite $G$-invariant measure; for instance, $H$ could be a lattice in $G $.  In this case, we scale the measure on $G/H $ to a probability measure and denote by $\mu_H$ the invariant random subgroup induced by the left action of $G$ on $G / H$.

This construction can be further generalized. Let $H$ be a closed subgroup in $G$, and let $N=N_G(H)$ be its normalizer in $G$. Suppose that $G / N$ admits a left $G$-invariant probability measure. Consider the map $G \to \sub_G$ given by $g \mapsto gHg^{-1}$. The push-forward measure on $\sub_G$ is a $G$-invariant measure supported on the conjugates of $H$, which we denote $\mu_H$. This notation conflicts with that in the previous paragraph; indeed, if both $G/H$ and $G/N$ admit an invariant probability measure, then we have two definitions of $\mu_H$. However, both these constructions give rise to the same measure.

 Invariant random subgroups can also be constructed as products. Let $H_1,H_2$ be commuting subgroups of $G$, and assume that $H_1,H_2$ and  the product $H_1H_2$ are all closed in $G $.  If $\mu_1,\mu_2$ are invariant random subgroups of $G$ that are supported on $\sub_{H_1},\sub_{H_2}$,  we can push forward the product measure on $\sub_{H_1} \times \sub_{H_2}$ using $$\sub_{H_1} \times \sub_{H_2} \to \sub_G, \ \ (K_1,K_2) \mapsto K_1 \times K_2$$ to a measure on $\sub_G$ which we denote $\mu_1 \otimes \mu_2$. It is easy to check that this measure is $G$-invariant.

Finally, suppose $\Gamma <G$ is a lattice and $\mu $ is an IRS of the discrete group $\Gamma $. The IRS of $G $ \it induced \rm from $\mu$ is the
random subgroup of $G$ obtained by taking a random conjugate of $\Gamma$ and then
a $\mu$-random subgroup in this conjugate (which is well-defined because
of the invariance of $\mu$). Formally, the natural
map $$G\times\sub_{\Gamma} \ni (g,\Lambda) \longmapsto g \Lambda g^{-1} \in\sub_G$$
factors through the quotient of $G \times \sub_\Gamma$ by the $\Gamma $-action
$(g,\Lambda)\gamma=(g\gamma, \gamma^{-1}\Lambda\gamma)$.  This quotient has a natural $G$-invariant probability measure, and we define our IRS to be the push forward of this measure by the factored map
$(G\times\sub_{\Gamma})/\Gamma\rightarrow\sub_G$.  This is a particularly important construction when $G=\mathrm{SO}(n,1)$, in which case lattices have many IRSs, see \cite{7s}.

\subsection{IRSs via stabilizers}

The following theorem shows that, when $G$ is locally compact, every IRS is induced from a p.m.p action.

\begin{thm}\label{pmpthm}
Let $G$ be a locally compact second countable group, and let $\mu\in\text{IRS}(G)$. Then $\mu$ is induced from some p.m.p action of $G$.
\end{thm}

When $G$ is countable and discrete, this was proven in \cite[Proposition 13]{Miklos2}.  The reader should note that we do not use this result anywhere in this paper, although it could be used to give a slightly shorter proof of Theorems \ref{thm:SZ-rk1} and \ref{prop:SZ}.  However, we consider it of  independent interest.

\begin{proof}
The coset space of $G$, written $\text{Cos}_G$, is the set
\begin{align*}
  \text{Cos}_G = \{Hg \,:\, H \in \sub_G, \, g\in G\},
\end{align*}
equipped with the Fell topology of closed subsets of $G$.   Then $G $ acts on $  \text{Cos}_G $, both from the left via $Hg \overset {k}{\mapsto} kHg$ and from the right via $Hg \overset {k}{\mapsto} Hgk$.  With respect to the left action, the stabilizer map is the natural projection
$$ \text {stab} : \text{Cos}_G \longrightarrow\sub_G, \ \ Hg \longmapsto H,$$
where the fiber above $H \in \sub_G$  is the coset space $H \backslash G$.  Note that as usual for a stabilizer map, the $G$-action permutes the fibers, and descends to the conjugation action of $G$ on $\sub_G$. 

By \cite[Theorem 3.1]{ Biringerunimodularity}, for almost every subgroup $H$ in the support of $\mu$, there is a right $G$-invariant measure $\nu_H$ on $H \backslash G$, and one can choose the map $H \mapsto \nu_H$ to be Borel\footnote {Here, we regard the $\nu_H $ as measures on $ \text{Cos}_G$, so the parameterization is Borel when for every Borel $ B\subset \text{Cos}_G$  the map $H\longmapsto \nu_H (B) $ is Borel.}.  So,  integrating against $\mu$  creates a measure $\nu$ on $ \text{Cos}_G$:
$$\nu=\int_{\sub_G} \nu_H \, d\mu.$$
The left and right $G$-actions on $\text{Cos}_G$ commute, so the left action of $g$ pushes forward $\nu_H$  to a measure on $gHg^ {-1} \backslash G$  that is again right $G $-invariant.  By uniqueness,  we have $g_*\nu_H=\lambda \nu_{gHg^ {-1} }$  for some $\lambda\in \BR $.  Combining this with the  conjugation invariance  of $\mu $, the left $G $-action preserves the \emph {measure class} of $\nu$.

Suppose for a moment that each $\nu_H $ is finite,  i.e.\ that $\mu $-a.e.\ $H\in \sub_G$ has co-finite volume.  Then  after scaling, we can  take each $\nu_H $ to be a probability measure.  The $\nu_H$ are then permuted by the left $G $-action,  which implies that $\nu$ (and not just its measure class) is left $G$-invariant.  Furthermore, $\nu$  is  then a probability  measure that pushes forward to $\mu$ under $\text {stab}$. So, the theorem follows.

 In general, it is enough to prove the theorem when $\mu$ is ergodic,  so we can break into cases depending on whether  for $\mu$-a.e.\ $H\in \sub_G$ we have
 \begin {enumerate}
 \item $\nu_H $ is finite (in which case we are done),
 \item $\nu_H$ is infinite and $H \backslash G$  is discrete,
 \item $\nu_H$ is infinite and $H \backslash G$  is non-discrete.
 \end {enumerate}
 Of the latter two cases, (3)  is the more difficult, so we will focus on that and  mention (2) again at the end.   Assuming (3), the two problems are that only the measure class of $\nu$ is $G $-invariant, and  that $\nu$ is not a probability measure.   To deal with the first issue, we replace  the action $ G \actson \text{Cos}_G $  with its \emph {Maharam  extension}
$$ G \actson \text{Cos}_G \times \BR, \ \ (Hg,t) \overset {k}{ \longmapsto} \left (kHg, \frac {d\nu}{d k_*(\nu)}(Hg) \cdot t\right ),$$
 which preserves the measure $\nu\times \ell $ on $\text{Cos}_G \times \BR$, where $\ell$ is Lebesgue measure.  Note that the stabilizer map for the Maharam  extension is just the projection
 $$ \text{Cos}_G \times \BR \longrightarrow \sub_G, \ \ (Hg,t) \longmapsto H,$$
 for since  $H$ acts trivially on $H \backslash G$ it preserves $\nu_H$, and combining this with the conjugation invariance of $\mu$ gives that $ \frac {d\nu}{d h_*(\nu)}(Hg)=1$  whenever $h\in H$.
 
We now have a $ G $-invariant measure $\nu\times \ell $,  and a disintegration
$$\nu \times \ell =\int_{\sub_G} \nu_H \times \ell \ d\mu$$
with respect to the stabilizer map.  If the fiber measures $\nu_H \times \ell$ were probability measures, we would be done as before, but they are not. So, the idea is to replace each fiber $H \backslash G \times\BR $  with an associated Poisson process.

 One way to make this rigorous is as follows.  Both $\text{Cos}_G \times \BR$ and $\sub_G\times\BR_+$  are Polish spaces  that map onto $\sub_G$, such that the fiber measures 
$$\nu \times \ell \text{ on } H \backslash G \times\BR, \ \text { and } \ \ell \text { on } \{H\} \times\BR_+$$
 have no atoms, so a result of Rokhlin \cite[pg 41]{Rohlin} gives a measure isomorphism $$\phi : \text{Cos}_G \times \BR \longrightarrow \sub_G\times\BR_+$$
that commutes with the projections to $\sub_G$.\footnote{In the cited reference, Rokhlin's theorem is only for probability measures, not arbitrary $\sigma$-finite measures.  However, one can always scale a   $\sigma$-finite measure with a Borel function to become a probability measure,  and it is not hard to then see that his theorem extends to the $\sigma$-finite case.}
Conjugating  the $G $-action on $\text{Cos}_G \times \BR $ by  $\phi$, we have a  measurable $G$-action on $\sub_G\times\BR_+$  such that
\begin {enumerate}
\renewcommand{\theenumi}{\alph{enumi}}
\item the fibers  $\{H\} \times\BR_+ $  are permuted by conjugating $H$,
\item $g\in G$  pushes forward the Lebesgue measure on $\{H\} \times \BR_+$ to the Lebesgue measure on $\{gHg^{-1}\} \times\BR_+$,
\item the stabilizer of $(H,t)\in \sub_G\times\BR_+$ is $H$.
\end {enumerate}

Let $\mathcal S(\BR_+)$ be the set of all countable subsets of $\BR_+ $, and let $\pi$ be the Poisson process on $\BR_+$, which we regard as a probability measure on $\mathcal S(\BR_+)$. (We refer the reader to \cite{Coxpoint,Daleyintroduction}  for details.)   There is an induced action $G \actson \sub_G\times \mathcal S(\BR_+)$, which (using (b) above) preserves  the \emph {probability measure} $\mu \times \pi$. 

 By (b) and (c),  the quotient group $N_G(H)/H$  acts freely on  the fiber $\{H\} \times\BR_+ \subset\sub_G \times\BR_+$,   preserving the Lebesgue measure, so by Lemma \ref{pois} below the induced action on  $\{H\} \times \mathcal S(\BR_+) $ is essentially free.  In other words,  the $G$-stabilizer map   
$$\text {stab} : \sub_G\times \mathcal S(\BR_+) \longrightarrow \sub_G$$  is the projection onto the first coordinate. Hence, $\text {stab}_*(\mu \times\ell)=\mu$, and we're done.

\medskip 

Finally, a quick remark about the proof of (2).  Here, there is no need for the Maharam  extension since if one defines the $\nu_H$ to be  the appropriate counting measures, then they will be permuted by the $G $-action.  The rest of the proof proceeds in the same way,  with  all references to $\BR_+ $ replaced by $\BZ$  and the Poisson  process replaced by an i.i.d.\end {proof}

 As promised, here is the lemma  we used in the proof above.

\begin {lem}\label {pois}
Suppose that a group $G $ acts  measurably and freely on $\BR_+$,  preserving Lebesgue measure $\ell$.   Then $G $ acts  essentially freely on  the space  $\mathcal S(\BR_+) $ of countable subsets of $\BR_+ $, with respect to the Poisson process  $\pi$.
\end {lem}

This is folklore, but we include a brief proof for completeness, since we are not aware of a reference.

\begin {proof}
On the contrary, suppose that  there is a positive $\pi$-probability  that an element $D\in\mathcal S(\BR_+) $  has nontrivial stabilizer.  Then there is some interval $(0,n)$ such that  there is a positive $\pi$-probability  that for $D \in S(\BR_+) $,  the intersection $D\cap (0,n)$ contains elements $x_1,\ldots,x_4$  with 
\begin {equation}g(x_1)=x_2, \ g (x_3) = x_4, \ \ \text { for some } g \in G. \label {gxeq}\end {equation}

The intersection $D\cap (0,n)$ is almost surely finite, and after conditioning on  cardinality $k$,  the points of $D $ are distributed within $(0,n)$  according to the Lebesgue  measure on $(0,n)^k$, c.f.\ \cite[Example 7.1(a)]{Daleyintroduction}:
\begin {equation*}\label {janossy}\mathrm {Prob}\left (\substack{\text {for } (x_1,\ldots,x_k)\in (0,n)^k, \text{ we have } D\cap (0,n) =\{x_1,\ldots,x_k\}, \\ \text { given that } D\cap (0,n) \text { has } k \text { elements.}}\right)=d\ell^k(x_1,\ldots,x_k).\end {equation*}
So  in particular, there is a positive probability that $ (x_1,\ldots,x_4) \in (0,n)^4$ satisfies \eqref{gxeq}. But by freeness of the action, for such  $ (x_1,\ldots,x_4)$, the last coordinate is determined by the first three. So,  applying Fubini's theorem on $(0,n)^4=(0,n)^3 \times (0,n)$, we have a contradiction.
 \end{proof}

\subsection {IRSs in Lie groups}  From now on, unless explicitly mentioned otherwise, we will assume that $G $ is a connected Lie group.  

The following is a variant of the classical Borel density theorem:

\begin{thm}[Borel's density theorem] \label{discIRS}
If $G$ is simple (with trivial center) then every IRS with no atoms is supported on discrete Zariski dense subgroups of $G$.
\end{thm}

A subgroup $\Gamma $ of $G $ is \emph {Zariski dense} if the only closed subgroup $H < G $ that contains $\Gamma $ and has finitely many connected components is $G $ itself.  This coincides with the algebraic definition of Zariski density when $G $ has the structure of a real algebraic group.  One recovers the classical theorem of Borel \cite {borel} when $\mu $ is the IRS $\mu_H$ associated with a closed subgroup of finite co-volume $H\le G$. 

 Although the proof could be rearranged in a way that avoids and hence reproves the classical Borel-Density-Theorem, we will make use of Borel's result in the proof of the following:

\begin{lem}\label{lem:finiteIRS}
The only IRS supported on the set of finite subgroups of $G$ is the Dirac measure $\delta_{\{ {\rm id} \}}$.
\end{lem}

\begin{proof}
Let $\mu$ be an ergodic IRS supported on finite subgroups of $G$. Since $G$ has only countably many conjugacy classes of finite subgroups, $\mu$ is supported on a single conjugacy class, say $F^G$ for some appropriate finite subgroup $F\le G$. Thus $\mu$ induces a finite $G$-invariant probability measure on the homogenous space $G/N_G(F)$. Thus $N_G(F)$ is of finite co-volume in $G$.  By the classical Borel density theorem, $N_G(F)$ is Zariski dense. Since $F$ is finite, $N_G (F) $ is algebraic and hence $N_G(F)=G$. As $G$ is connected and $F$ is discrete, we deduce that $F$ is central in $G$. Finally, since $G$ has a trivial center $F=\{ {\rm id} \}$.
\end{proof}

\begin{proof}[Proof of Theorem \ref{discIRS}]
Associating to a closed subgroup $H < G $ either
\begin{enumerate}
\item the Lie algebra of the identity component of $H $, or
\item the Lie algebra of the identity component of the Zariski closure of $H, $
\end{enumerate}
an IRS induces two
$\mathrm{Ad} (G)$-invariant measures, $\mu_1 $ and $\mu_2 $, on the Grasmannian manifold of the Lie algebra $\mathfrak{g}$ of $G$. Note that both (1) and (2) are measurable as maps from $\text{Sub}_G$ to the Grassmannian (see \cite{Gelander-Levit} for details).
As follows from Furstenberg's proof of the Borel density theorem (see \cite{Furstenberg1}) every such measure is supported on $\{\{0\},\mathfrak{g}\}$.

The $\mu_1 $-mass of $\mathfrak{g}$ is exactly the mass that the given IRS gives to the atom $G$, which is by assumption $0$. Thus $\mu_1 $ is the Dirac measure supported on $\{0\}$, which is equivalent to the statement that our IRS is supported on discrete subgroups.

On the other hand, any Zariski closed discrete subgroup of $G $ is finite.   Therefore, $\mu_2 (\{0\}) $ is the amount of mass that our IRS gives to finite subgroups of $G $.  By Lemma \ref{lem:finiteIRS}, this must be $0 $.  Therefore, $\mu_2 $ is the Dirac measure supported on $\mathfrak{g}$, which implies that our IRS is supported on Zariski dense subgroups.
\end{proof}

\begin{rem}
For a connected semisimple group Lie $G$ (which is neither necessarily simple nor center-free) an elaboration of the argument above provides the following. Given an ergodic IRS $\mu$ in $G$, there are normal subgroups $N_1, N_2\lhd G$, with $N_1\le N_2$ such that for $\mu$-almost every $H\in\sub_G$, the identity connected component $H^\circ=N_1$ and the Zariski closure $\overline{H}^Z=N_2$.
For details see \cite{Gelander-Levit} where the analog result is established for groups over general local fields.
\end{rem}

\section{Large injectivity radius and BS-convergence} \label{sec:3} \label {IRSs}

Let $G$ be a semi-simple Lie group, and let $X=G/K$ be an associated Riemannian symmetric space.  An  \emph {$X$-orbifold} is a  Riemannian orbifold obtained as a quotient $\Gamma \backslash X$, for some discrete subgroup $\Gamma < G$. 
Our goal is to understand the topology of $\sub_G$  geometrically through these quotient orbifolds, and then to promote this to an understanding of discrete IRSs as  random pointed $X$-orbifolds.

To begin with, let's understand the geometric meaning of Chabauty convergence to the identity.  The \emph {injectivity radius} of an $X$-orbifold $M=\Gamma \backslash X$ at $x\in M$  is $$\mathrm{InjRad}_{\Gamma_n \backslash X} (x) = \frac12 \min \{ d(\tx , \gamma \tx ) \; | \; \gamma \in \Gamma_n - \{ \mathrm{id} \} \},$$
 where $\tilde x\in X$ is any lift of $x$.   We then have:

\begin {lem}\label{convergeid}
A  sequence  of subgroups $\Gamma_n<G$  converges to $\{\mathrm{id}\}$ in  the Chabauty topology if and only if $\mathrm{InjRad}_{\Gamma_n \backslash X} ([\mathrm{id}]) \to \infty$. 
\end {lem}

Here, $[\mathrm{id}]$  is the projection of the identity in $G $.

\begin {proof}
It suffices to show that the subsets 
$$ 
 U_R = \{ H \in \sub_G \ | \  \nexists \gamma \in H \smallsetminus \{\mathrm{id}\} \text{ with } d_X([\mathrm{id}],\gamma[\mathrm{id}]) \leq R \}\subset \sub_G,
$$ 
with $R\in(0,\infty)$, form a basis of open sets around $\{\mathrm{id}\}\in\sub_G $.  

To show that $U_R$ is open, consider a sequence $(H_n)$ of subgroups in $\sub_G$ that do not belong to $U_R $, i.e.\ there are elements $\gamma_n$ in $H_n \setminus \{\mathrm{id}\}$ with $d_X([\mathrm{id}],\gamma[\mathrm{id}]) \leq R$.  Passing to a subsequence, we may suppose that the sequence $(\gamma_n)$ converges towards some element $\gamma \in G$.   Using the exponential map and replacing each $\gamma_n $ by an appropriate power, we may furthermore assume that $\gamma \neq \mathrm{id}$. Then $\gamma$ translates $[\mathrm{id}]$ by at most $R$,  and appears in any Chabauty limit $H$ of $(H_n)$,  so any such limit is also outside $U_R $.

We can prove that the $U_R$  form a basis by comparing with the basic open sets $O_1(K), O_2(U)$ defined in \S \ref{sec:2}.  First, suppose $K\subset G$ is compact and $\mathrm{id} \notin K$.  Choosing $R$  larger than the $[\mathrm{id}]$-translation distance of all $\gamma \in K$,  we have $U_R \subset \mathcal O_1(K)$.  And if $U\subset G$ is open and $\mathrm{id} \in U$, then $\mathcal O_2(U)=\sub_G$, hence contains $U_R$ for all $R$.\end {proof}

On the level of IRSs,  we have:

\begin{prop}\label {convergenceprop123}
 A sequence of IRSs $(\mu_n)$ of $G$  converges weakly to $\mu_{\mathrm{id}}$ if and only if for every $R>0 $, the $\mu_n$-probability that for a subgroup $\Gamma \in\sub_G $  we have $\mathrm{InjRad}_{\Gamma \backslash X} ([\mathrm{id}]) \leq R$ tends to zero  as $n\to \infty $.	
\end{prop}
\begin {proof}
 With the notation of the proof of Lemma \ref{convergeid}, 
$\mu_{\Gamma_n} \rightarrow \mu_{\mathrm{id}}$ if and only if
$$
\lim_{n \rightarrow \infty }\mu _{\Gamma _{n}}(U_{R})=1\ \text{ for all }\
R>0.\qedhere
$$
\end {proof}

 And for lattice IRSs, this can be rewritten as follows.

\begin{cor}\label{lem:loccv}
 Suppose that $ (\Gamma_n)$ is a sequence of lattices in $G$.  Then $ ( \mu_{\Gamma_n} )$ converges weakly to $\mu _{\mathrm{id}}$  if for every $R>0$,  we have
\begin{equation}
\lim_{n  \to \infty} \cal{P}_n \{ x\in \Gamma_n \backslash X \ |\ \mathrm{InjRad}_{\Gamma_n \backslash X} (x) \leq R \}  = 0,\label {BSEQ}
\end {equation}
where $\cal {P}_n $ is the normalized Riemannian measure on $\Gamma_n\backslash X $.
\end{cor}

\begin{proof}
Consider the $G $-invariant probability measure $\hat \mu$ on $\Gamma_n \backslash G $.  Pushing forward this measure to $\mathrm {Sub}_G $ by the  stabilizer map $\Gamma_ng \mapsto g^{-1}\Gamma_n g$ gives $\mu_{\Gamma_n} $, while pushing it forward under the projection $\Gamma_n \backslash G \to \Gamma_n \backslash X$ gives $\cal{P}_n$.
Therefore,
\begin{multline*}
\mathcal{P}_n \left \{ x \in \Gamma_n \backslash X \mid \mathrm{InjRad}_{\Gamma_n \backslash X} (x) \leq R \right\}  \\ \begin{split}
& =  \hat \mu\left \{\Gamma_ng \in \Gamma_n \backslash G \ | \ \mathrm{InjRad}_{\Gamma_n \backslash X} ([g]) \leq R \right\} \\
&  = \mu\left \{g^{-1}\Gamma_ng \in \sub_G \mid  \mathrm{InjRad}_{g^{-1}\Gamma_ng \backslash X} ([\mathrm{id}]) \leq R \right \}, 
\end{split}
\end{multline*}  
where the last line uses that $d_X([g],\gamma[g])=d_X([\mathrm{id}],g^{-1}\gamma g[\mathrm{id}])$.  So, the corollary follows from Proposition \ref {convergenceprop123}.\end{proof}

In \S \ref{sec:4}, we will show that any sequence of irreducible lattice IRSs in a center-free higher rank semi-simple Lie group with property (T) weakly converges to $\mu_\mathrm{id}$.   Using Corollary \ref{lem:loccv},  we will see in Section \ref{sec:7}  how to deduce asymptotics for Betti numbers and representation multiplicities, as discussed in the introduction.
\medskip

 As in the Introduction, we say a sequence of $X $-orbifolds \emph {Benjamini--Schramm converges}  to $X $ when \eqref{BSEQ} holds for all $R $.   More generally,  the theory of invariant random subgroups of Lie groups can be  recast in a geometric context, where weak convergence  is replaced by a suitable generalized BS-convergence.
 This interpretation is inspired by a program in graph theory, e.g.\ \cite{Abertmatchings,Abert1,Abertbenjamini,Aldousprocesses,Benjaminigroup}, that was popularized by Benjamini--Schramm~\cite {localconvergence}.  We will briefly discuss the graph theory,  and then explain how to translate to the continuous setting.

\subsection{Graphs and IRSs of discrete groups}  \label{discgrpsec} All the material here is well-known; for more information we refer the reader to \cite{Miklos1, Aldousprocesses, Biringerunimodularity}.

Let $\mathcal G$ be the space of all isomorphism types of rooted graphs $(X,p)$, where
$$ 
 d\big ((X,p),(Y,q)\big )=\inf \left \{ \frac 1 R \ \Big | \ B_X(p,R)\cong B_Y(q,R)\right \}, 
$$ 
so two rooted graphs are close if balls of large radius around their base points are isomorphic.  We consider the set $\text {Prob}(\mathcal G) $ of all Borel probability measures $\mu$ on $\mathcal G$  with the topology of weak$^*$  convergence. 

One way to understand weak$^*$ convergence is as follows. For each $R>0$ and each finite rooted graph $B=(B,p)$,  let $P_{R,B}(\mu)$ be the probability that the $R$-ball around the root of a $\mu$-randomly chosen $(X,p)\in \mathcal G$ is isomorphic to $(B,p)$. Then \begin {equation} \label{Gconv} \mu_i \to \mu \text{ weakly} \ \ \iff \ \ P_{R,B}(\mu_i) \to P_{R,B}(\mu) \ \ \forall R>0, B=(B,p).\end {equation}
Here,  the condition $B_X(p,R)\cong (B,p)$ determines a basic open set for the topology of $\mathcal G$, whose $\mu$-measure is $P_{R,B}(\mu) $.  Equation \eqref{Gconv} follows since these sets are also closed.

Any finite graph $X$ determines an element $\mu_X \in \text{Prob}(\mathcal G)$, by choosing the root uniformly from the vertices of $X$. 
One says that a sequence $(X_i)$ of finite graphs \emph {Benjamini--Schramm converges} to  a measure $\mu \in \text{Prob}(\mathcal G)$ if $\mu_{X_i} \to \mu$ weakly.  In light of \eqref{Gconv}, a Benjamini--Schramm limit captures, for large $i$, the limiting statistics of the isomorphism types of all $R$-balls in $X_i$.

Now let $G=\langle S \rangle$ be a finitely generated group. A subgroup $H < G$ determines a rooted \emph {Schreier graph}, written $\mathrm{Sch}_S(H\backslash G)$, where vertices are cosets $Hg$, the root is $H$, and where an edge labeled $s \in S$ joins $Hg$ to $Hgs$.  
 Adding edge labels, we have a space $\mathcal G_S$ of isomorphism types of rooted $S$-labeled graphs, and the map
$$\sub_G \longrightarrow \mathcal G_S, \ \ H \longmapsto \mathrm{Sch}_S(H\backslash G)$$
is a homeomorphism onto its image, see\ \cite{Miklos1}.  So, an IRS $\mu$ of $ G$ determines a measure $\mathrm{Sch}_S(\mu)$ on $\mathcal G_S$, and the induced map $$\{\text {IRSs of } G\} \longrightarrow  \text {Prob}(\mathcal G) , \ \ \mu \longmapsto \mathrm{Sch}_S(\mu)$$ is a weak$^*$  homeomorphism onto its image.  Therefore, the study of IRSs of $G$ is equivalent to the study of (certain) random $S$-labeled rooted graphs.

\medskip

Passing to the continuous setting, we'd like to study discrete IRSs $\Gamma$ of a Lie group $G$ as random pointed quotients $\Gamma \backslash X$. There are a number of settings in which one can develop such a theory, but the following is the most general.

\subsection{The Gromov--Hausdorff space}  
Consider the set $$\mathcal M =\big \{ \text {proper, pointed length spaces } (X,x_0) \big \}/\text {pointed isometry}.$$  In \cite{Gromovgroups}, Gromov defined a notion of convergence of pointed metric spaces using a generalization of the Hausdorff metric. Following a variant given in \cite[\S 3.2]{ECG}, two pointed metric spaces $(X , x_0)$ and $(X',x_0')$ are
{\it $(\varepsilon, R)$-related}, written $$(X, x_0) \sim_{\varepsilon, R} (X', x_0') , $$ if there are compact subspaces $K\subset X$ and $K' \subset X' $ containing the basepoints and a relation $\sim$ between $K $ and $K' $ that satisfies the following properties:
\begin{enumerate}
\item $B_X (x_0 , R) \subset K$ and $B_{X' }(x_0', R) \subset K'$,
\item $x_0 \sim x_0'$,
\item for each $x \in K$, there exists $x'\in K'$ such that $x \sim x'$,
\item for each $x' \in K'$, there exists  $x \in K$ such that $x \sim x'$, and
\item if $x \sim x'$ and $y \sim y'$, then $|d_X (x,y) - d_{X'} (x',y') | \leq \varepsilon$.
\end{enumerate}
This defined a \emph {(pointed) Gromov--Hausdorff topology} on $\mathcal M $: a basis of neighborhoods of $(X,x_0)$ is defined by considering for each $\varepsilon >0$ and $R>0$ the set of proper, pointed length spaces that are $(\varepsilon , R)$-related to $(X, x_0)$. 

 It is well-known that this topology is separable and completely metrizable, i.e.\  {Polish}; a distance between $(X , x_0)$ and $(X',x_0')$ can be defined by taking the infimal $\epsilon=\epsilon_R$ such that $(X , x_0)$ and $(X',x_0')$ are $(\epsilon,R)$-related, and then setting
$$d\big ((X , x_0),(X',y_0)\big)=\sum_{R=1} ^\infty \frac { \min\{\epsilon_R,1\}}{2^R}.$$

\noindent Note that the space of rooted graphs $\mathcal G$  from the previous section embeds in $\mathcal M$, once we declare all edges in a graph to have unit length.

\medskip

Suppose that $G$ is a semi-simple Lie group, and let $X =G/K$ be an associated Riemannian symmetric space.  For simplicity, we will deal with \emph {$X$-manifolds} $\Gamma \backslash  X $ rather than $X $-orbifolds in the rest of the section, i.e.\  we will assume that our discrete  $\Gamma $ is torsion free.  As a geometric analogue of Lemma \ref {convergeid},  we have:

\begin {prop}\label {GHX}
A sequence of pointed	 $X$-manifolds $(M_i,p_i)$ converges in the Gromov--Hausdorff topology to $X$ if and only if $\mathrm{InjRad}_{M_i}(p_i) \to \infty $.
\end {prop}

 Here, any base point for a space with a transitive isometry group, like $X$,  gives the same element of $\mathcal M $,  so we drop the base point from the notation. The difficulty in proving Proposition \ref{GHX} is that Gromov--Hausdorff convergence is metric, not topological, and the homeomorphism type may change drastically in a Gromov--Hausdorff limit.  However, the following lemma shows that when the curvature and its derivatives are bounded,  this is not the case.

\begin{lem} \label{collapse} Suppose that $(M_i,p_i)$ is a sequence of complete Riemannian $d$-manifolds
and that  the covariant derivatives of the curvature tensors $R_i$ satisfy $$|\nabla^k R_i| < C_k < \infty$$ for some fixed sequence $(C_k)$ independent of $i$ and of the point in $M_i$.  If $(M_i,p_i)$  converges to a Riemannian $d$-manifold $(M,p)$ in the Gromov--Hausdorff topology, then the convergence is smooth and $$\mathrm{InjRad}_{M_i}(p_i)\longrightarrow \mathrm{InjRad}_{M}(p).$$
\end{lem}

Here, tensor norms are induced by the associated Riemannian metrics.  We say that $(M_i,p_i)\longrightarrow (M,p)$ \emph {smoothly} if there is a sequence of embeddings
$$\phi_i:B(p,R_i)\longrightarrow M_i$$with $R_i \rightarrow\infty$ and $\phi_i (p) = p_i $,  such that $\phi_i ^*g_i\rightarrow g $ in the $C ^\infty $-topology, where $g_i,g $ are the Riemannian metrics on $M_i,M $.  

The authors would like to thank Igor Belegradek for a very helpful conversation related to the proof below.

\begin{proof}[Proof of Lemma \ref {collapse}]
First, suppose that $\mathrm{InjRad}_{M_i}(p_i)\to 0$. As the sectional curvatures of $M_i$ are bounded, a result of Cheeger--Gromov--Taylor \cite[Theorem 4.7]{Cheegerfinite} implies that $\vol_{M_i}(p_i,1)\to 0 $ as well.   The Gromov--Hausdorff limit then has Hausdorff dimension at most $d-1$; this dates back to work of Gromov in the 1970s, but a citation for a more general result is \cite[Theorem 3.1]{Cheegerstructure}.   In our case, though, the limit is a Riemannian $d$-manifold, so we have a contradiction.

So, there is a lower bound on the injectivity radii $\mathrm{InjRad}_{M_i}(p_i)$.  From this and the bounds on the derivatives of curvature, it is well-known that the convergence is smooth, see \cite[Theorem 4.1]{Lessareeb}. The continuity of injectivity radius is then a result of Erlich \cite{Erlichcontinuity}.
\end{proof}

\begin {proof}[Proof of Proposition \ref {GHX}]
Pick a base point $p\in X$. If the injectivity radius at $p_i$ goes to infinity, then there are radii $R_i\to \infty$ such that the ball $B_{M_i}(p_i,R_i)$  is isometric to $B_X (p,R_i)$. This isometry gives a $ (0,R_i)$-relation.

 Conversely, suppose  $(M_i,p_i)$ converges in the Gromov--Hausdorff topology to $X$.  The hypotheses of Lemma \ref{collapse} are satisfied, since for every $i $ and at every point in $M_i $ we have $|\nabla^k R_i|=|\nabla^k R_X| $,  which we can take as our $C_k$.  So, the injectivity radius at $p_i$ goes to infinity.
\end {proof}

Benjamini--Schramm convergence of $X$-manifolds to $X$ can now be reinterpreted using weak convergence of measures on $\mathcal M$.  Note that every complete finite volume Riemannian orbifold $M $ produces a probability measure $\mu_M $  on $\mathcal M$: one pushes forward the normalized Riemannian measure of $M $ under the natural map $M \longrightarrow \mathcal M $, where $x \longmapsto (M,x)$.  Also, we denote the atomic measure on $X \in \mathcal M$ by $\mu_X$.
 
 As an immediate consequence of Corollary \ref{lem:loccv} and Proposition \ref {GHX}, we have:
 
\begin {cor}
 For $X $-manifolds $ M_i =\Gamma_i\backslash X $, the following are equivalent:
 \begin {enumerate}
 	\item  the IRSs $\mu_{\Gamma_i}$  weakly converge  to $\mu_{\mathrm{id}}$,
 	\item for every $R>0$, $
\lim_{i  \to \infty} \cal{P}_i \{ x\in \Gamma_i \backslash X \ |\ \mathrm{InjRad}_{\Gamma_i \backslash X} (x) \leq R \}  = 0,$
\item the measures $\mu_{M_i}$  on $\mathcal M $ weakly converge to $\mu_X$.
 \end {enumerate}
\end {cor}

\subsection {General Benjamini--Schramm convergence}

 To reinterpret weak convergence of more general IRSs geometrically, we need to add frames to our space  $\mathcal M$, similarly to how we added $S $-labels to rooted graphs in \S \ref {discgrpsec}.

A \emph {frame} for a Riemannian manifold $M $ is an orthonormal basis $f$ for some tangent space $T_{\pi (f)}M $, where $\pi (f)\in M $.  We let 
$$\mathcal {M F}^d =\big\{ \text {framed Riemannian } d\text{-manifolds } (M,f)\big\}/\text{framed isometry}.$$  A \emph {framed $(\epsilon, R)$-relation} between $(M,f) $ and $ (N,f') $ is an $(\epsilon, R) $-relation between   the pointed manifolds  $(M, \pi (f)) $ and $(N,\pi (f'))$ with the additional assumption
\begin {equation} \exp_{f}(v) \sim\exp_{f'}(v) \text{ when } v\in B_{\BR^d}(0,R).\label {exper}\end {equation}
 Here, $\exp_{f} : \BR^d \longrightarrow M$ is the Riemannian exponential map associated to $f$. If a framed  $(\epsilon,R)$-relation exists, we write $(M, f) \sim_{\epsilon, R} (N, f') $.  As in the pointed case, framed $(\epsilon, R) $-relations induce a  \emph {(framed) Gromov--Hausdorff topology} on the set $\mathcal {M F}^d $, and this topology is again Polish.
 \medskip
 
If $G $ is a semi-simple Lie group and $X = G/K $ is a symmetric space, let $$\sub_G^{\text{dtf}}=\{ \text {discrete, torsion free } \Gamma < G\} \subset \sub_G,$$
 Fixing an orthonormal frame $f$ for $X $ and setting $d=\text {dim} X $, we have a map
 $$\Phi : \sub_G^{\text {dtf}} \longrightarrow \mathcal{MF}^d, \ \  \Gamma \overset{\Phi}{\longmapsto} (\Gamma \backslash X, [f]).$$ 
   
\begin{prop} \label{p:chabvsGH}  The map $\Phi$ is a homeomorphism  onto its image. \end{prop}
\begin {proof}
 As $\mathcal {MF}^d$  is Hausdorff, it suffices to show that $\Phi$  is a continuous, proper injection.
Injectivity is clear, since an isometry $(\Gamma \backslash X, [f])\longrightarrow (\Gamma' \backslash X, [f])$  lifts to a $(\Gamma,\Gamma')$-equivariant isometry $X\longrightarrow X$  fixing the base frame $f$,  which must then be the identity,  implying $\Gamma=\Gamma'$.    For properness, Lemma \ref{collapse} implies that on any compact subset $K \subset \mathcal {MF}^d$, there is some $\epsilon >0 $ such that
$$\text {InjRad}_{M} (p) \geq \epsilon \ \ \text {for all } (M,p)\in  K. $$
So, $\Phi^{-1}(K)$ is  a family of uniformly discrete, torsion free subgroups of $G$.   Lemma~\ref{uniformdiscretelemma}  implies that $\Phi^{-1}(K)$  is precompact in $\sub_G^{\text {dtf}}$, so it suffices to check that $\Phi$ is continuous, since then the preimage $\Phi^{-1}(K)$  will be closed, hence compact.

 Suppose that $\Gamma_i \to \Gamma_\infty$ in $\sub_G$, write $M_i=\Gamma_i \backslash X$ and let the projection of the frame $f\in X$ be $f_i \in T_{p_i}M_i$.  Fixing $R>0$, we define a relation $\sim$ between the balls  $B_{M_i}(p_i,R)$  and $B_{M_\infty}(p_\infty,R)$ as in \eqref{exper}, via
$$ \exp_{f_i}(v) \sim\exp_{f_\infty}(v) \text{ when } v\in B_{\BR^d}(0,R). $$

 Fixing $\epsilon >0$, we want to show that $\sim$ is an $(\epsilon, R) $-relation for large $i $. As conditions (1) -- (4)  are immediate, the point is to prove (5), i.e.\ that for  large $i $,
\begin {equation}\label {5eq}\big | \, d_{M_i}( \exp_{f_i}(v) , \exp_{f_i}(w) ) - d_{M_\infty}(\exp_{f_\infty}(v),\exp_{f_\infty}(w)) \, \big | \leq \epsilon\end {equation}
 for all $v,w\in B_{\BR ^ d}(0,R)$.
If not, then after passing to a subsequence, there are sequences $v_i,w_i$ that violate this inequality for all $i$.   Passing to another subsequence,  we can assume that $(v_i,w_i)\to (v_\infty,w_\infty) $ in $\overline { B_{\BR ^ d}(0,R) }\times \overline { B_{\BR ^ d}(0,R)}$.

Now $\exp_{f_i}(v_i)$ is the projection under $X \longrightarrow M_i$ of the point $\exp_{f}(v_i)\in X$, and similarly for $w_i$ and $i=\infty$. So for $i=1,2,\ldots,\infty$, there are $g_i \in \Gamma_i$  with
\begin {align*} d_{M_i}( \exp_{f_i}(v_i) , \exp_{f_i}(w_i) ) = d_X( g_i\exp_{f}(v_i) , \exp_{f}(w_i) ). \end {align*}
 That is, $g_i \exp_{f}(v_i) $ is the closest point in the $\Gamma_i$-orbit of $\exp_{f}(v_i) $ to $\exp_{f}(w_i) $.

 Passing to a subsequence, we may assume that $g_i \to g $ in $G$,  since all the $g_i$ translate a point inside $B_X(\pi(f),R)$ a distance at most, say, $10R$. Then by Chabauty convergence,  we have $g\in \Gamma_\infty $, so
\begin{align*}
	\lim_{i\to\infty} \, d_{M_i}( \exp_{f_i}(v_i) , \exp_{f_i}(w_i) ) & = \lim_{i \to\infty} \, d_X( g_i \exp_{f}(v_i) , \exp_{f}(w_i) ) \\ &= d_X( g \exp_{f}(v_\infty) , \exp_{f}(w_\infty) ) \\
 & \geq d_{M_\infty}( \exp_{f_\infty}(v_\infty) , \exp_{f_\infty}(w_\infty) ).
\end{align*}

 On the other hand, Chabauty convergence also provides a sequence $x_i \in \Gamma_i$ with $x_i \to g_\infty$. So,
\begin{align*}
	\lim_{i\to\infty} \, d_{M_i}( \exp_{f_i}(v_i) , \exp_{f_i}(w_i) ) & \leq \lim_{i \to\infty} \, d_X( x_i \exp_{f}(v_i) , \exp_{f}(w_i) ) \\ &= d_X( g_\infty \exp_{f}(v_\infty) , \exp_{f}(w_\infty) ) \\
 & =d_{M_\infty}( \exp_{f_\infty}(v_\infty) , \exp_{f_\infty}(w_\infty) ).
\end{align*}
 This contradicts the fact that the $v_i,w_i$ were chosen to violate \eqref{5eq}.
 \end {proof}

An IRS $\mu$ of $G $  is \it discrete \rm or \emph {torsion free} if $\mu$-a.e.\  closed subgroup of $G$ is discrete, or torsion free.   As a consequence of Proposition \ref {p:chabvsGH}, we have:

\begin{cor}\label{geometric}
 The following map is a homeomorphism  onto its image: $$\Phi_* : \{\text {discrete, torsion free IRSs of } G\} \longrightarrow \probability (\mathcal {M F}^d) $$ 
 \end{cor}
 
So,  weak convergence of (discrete, torsion free) IRSs can be viewed as weak convergence of measures on  the Gromov--Hausdorff space of framed $X $-manifolds.
 
 \smallskip
 
Corollary \ref {geometric} captures most of the interesting topology of the space of all IRSs of $G$.  When $G $ is a simple Lie group, for instance, Theorem \ref{discIRS} and Proposition \ref{prop:isolated} imply that every IRS of $G$  is supported on discrete subgroups, except for a possible atom on the isolated subgroup $\{G\}\in \sub_G$. So, the space of all IRSs is a cone on the space of discrete IRSs.  

While we have chosen to simplify the argument by considering manifolds instead of orbifolds, the corollary is still true if one replaces framed manifolds by framed orbifolds and drops the torsion free hypothesis.  For simple Lie groups $G $, a quick and dirty argument is as follows. One proves as in Proposition \ref {p:chabvsGH} that the map $\Phi$ in a continuous injection, and hence a Borel isomorphism onto its image\footnote {Any continuous injection between standard Borel spaces  is a Borel isomorphism \cite [Theorem 15.6]{Kechris}, and $\sub_G$ and $\mathcal {M F}^d $ are Polish spaces.}, so $\Phi_*$  is a continuous injection.  Since $G $ is simple, the argument in the previous paragraph shows that the space of discrete IRSs is compact, and any continuous injection from a compact space into a Hausdorff space is a homeomorphism onto its image.

\medskip
\noindent

\section{IRSs in higher rank} \label{sec:4}

As in the previous section, suppose that $G$ is a center free semi-simple Lie group without compact factors, and let $X=G/K$ be the associated Riemannian symmetric space. We say an IRS is {\it irreducible} if every simple factor acts ergodically. When $G$ has higher rank and Kazhdan's property $(T)$ we prove the following strong result using the Nevo--St\"{u}ck--Zimmer rigidity theorem (see below):

\begin{thm}\label{thm:SZ-rk1}
Let $G$ be a center-free semi-simple Lie group of real rank at least $2$ and with Kazhdan's property $(T)$. Let $\mu$ be a non-atomic irreducible IRS of $G$. Then $\mu=\mu_\gC$ for some irreducible lattice $\gC\le G$.
\end{thm}

Recall that a simple Lie group of $\BR$-rank at least $2$ has property $(T)$ by Kazhdan's theorem \cite[Section 1.6]{Tbook} and a rank-1 simple Lie group has property $(T)$ if and only if it is locally isomorphic to $\text{Sp}(n,1),~n\ge 2$ or $F_{4(-20)}$ by Kostant's result \cite[Section 3.3]{Tbook}. A semi-simple Lie group has property $(T)$ iff all its simple factors have $(T) $. By the arithmeticity theorems of Margulis and Corlette--Gromov--Schoen \cite{Mar-arithmeticity,Corlette,Gr-Sc}, if $G$ has property $(T)$ then all its lattices are arithmetic.

When all the simple factors of $G$ are of real rank at least $2$, one can furthermore classify all the ergodic invariant random subgroups of $G$ as follows:

\begin{thm}\label{weak}
Let $G$ be a connected semi-simple Lie group without center, and suppose that each simple factor of $G$ has $\mathbb{R}$-rank at least $2$.
Then every ergodic invariant random subgroup is either
\begin{enumerate}
  \item $\mu_N$ for a normal subgroup $N$ in $G$,
  \item $\mu_\Lambda$ for a lattice $\Lambda$ in a normal subgroup $M$ of $G$, or
  \item products of the previous two measures, where $N$ and $M$ commute.
\end{enumerate}
Explicitly, if $\mu$ is an ergodic invariant random subgroup, then there are commuting normal subgroups $N,M$ in $G$ and a lattice $\Lambda$ in $M$ such that $\mu = \mu_N \otimes \mu_\Lambda = \mu_{N \times \Lambda}$.
\end{thm}

We shall prove Theorems \ref{weak} and \ref{thm:SZ-rk1} by making use of the
following fundamental result of Nevo, St\"{u}ck and Zimmer, which is a particular case of \cite[Theorem 4.3]{stuckzimmer}.

\begin{thm}[Nevo--St\"{u}ck--Zimmer] \label{prop:SZ}
Let $G$ be a connected semi-simple Lie group without center, such that each simple factor of $G$ has $\mathbb{R}$-rank at least 2 or is isomorphic to $\text{Sp}(n,1),~n\ge 2$ or $F_{4(-20)}$. Suppose that $G$, as well as every rank one factor of $G$, acts ergodically and faithfully preserving a probability measure on a space $X$. Then there is a normal subgroup $N \lhd G$ and a lattice $\Gamma < N$ such that for almost every $x \in X$ the stabilizer of $x$ is conjugate to $\Gamma$.
\end{thm}

Let us mention that some new results in the spirit of Theorem \ref{prop:SZ} were established recently in \cite{CrPe} and \cite{Levit}.

 Before we start the proofs, we would like to mention that the following could be simplified a bit by appealing to Theorem \ref{pmpthm}.  In particular, one could avoid referencing the Margulis normal subgroup theorem.   However, we have chosen to give an independent proof, as it is not that much longer.
 
\begin {proof} [Proof of Theorems \ref {thm:SZ-rk1} and \ref{weak}]
Let's assume that $G $ has $\BR $-rank at least $2 $ and Kazhdan's property $(T) $.  At various points in the proof we will mention how the assumptions of Theorem \ref{weak} imply a stronger conclusion.  In the following, let $\mu $ be an ergodic invariant random subgroup of $G$.

Suppose first that the action of $G$ is faithful. By \ref{prop:SZ} we obtain a normal subgroup $N$ and a lattice $\Gamma < N$ such that the stabilizer, i.e. the normalizer, of a $\mu$-random subgroup is conjugated to $\Gamma$.
We claim that $N = G$. Indeed, the direct complement $M$ of $N$ in $G$ normalizes every conjugate of any subgroup of $\Gamma$. Hence $M$ fixes almost every point in $\sub_G$ and as the action is faithful $M$ is trivial.

Next we claim that if $\Lambda$ is a subgroup of $G$ whose normalizer is $\Gamma$ then $[\Gamma:\Lambda]<\infty$.
Recall the Margulis Normal Subgroup Theorem: {\it a normal subgroup of an irreducible lattice in a semi-simple Lie group with $\R$-rank $\geq 2$ is either central or is a lattice}. In our cases, the assumptions of Theorem \ref{thm:SZ-rk1} implies that $\gC$ is irreducible, but the assumptions of Theorem \ref{weak} does not.
However, by \cite[Theorem 5.22]{Raghunathan},
there is a decomposition of $G$ as a product of normal subgroups $\prod G_i$ such that $\Gamma_i := \Gamma \cap G_i$ is an irreducible lattice in $G_i$ and $\prod \Gamma_i$ has finite index in $\Gamma$. Note that by the assumptions of \ref{weak} $\R\text{-rank}(G_i) \geq 2$ for every $i$. Moreover, the projection of $\Lambda$ to each $G_i$ cannot be trivial since $\Gamma$ is the full normalizer of $\Lambda$. By considering the commutator $[\Gamma_i,\Lambda]$ one deduces that $\Lambda_i := \Lambda\cap G_i$ is nontrivial for every $i$.
By the normal subgroup theorem $\Lambda_i$ is of finite index in $\Gamma_i$ as the latter is center free. Therefore $\prod\Lambda_i$ and hence also $\Lambda$ is a lattice in $G=\prod G_i$.

We have shown that a $\mu$-random subgroup in $\sub_G$ is a lattice. It is proved in \cite{stuckzimmer} that the action of $G$ on the subset of lattices in $\sub_G$ is tame (i.e., the Borel structure on the orbit space is countably separable). In particular, an ergodic measure supported on this subset must be supported on a single orbit. Thus $\mu = \mu_\Lambda$ for some lattice $\Lambda$. In particular, this finishes the proof of Theorem \ref{thm:SZ-rk1}.

\vspace {1mm}
We now finish the remaining cases of Theorem \ref{weak} when the action is not faithful. Let $N$ be the kernel of this action. If $N=G$ then $\mu$ is supported on a normal subgroup of $G$, and we are done. Otherwise, take a direct complement $M$ of $N$ such that $G \simeq N \times M$.

We note that a subgroup normalized by $N$ has a certain decomposition as a direct product. To this end, suppose that a subgroup $H \in \sub_G$ is normalized by a simple factor $L$ of $N$. By simplicity, either $H$ contains $L$ or $L \cap H = 1$. In the latter case, $L$ and $H$ commute, and thus the projection of $H$ to $L$ is central, and hence trivial. It follows that if $H$ is normalized by $N$ then it decomposes as $H = H_N \times H_M$ where $H_N := H \cap N$ is a product of simple factors in $N$, and $H_M := H \cap M$.

As there are finitely many possibilities for $H_N$, this factor of the decomposition is independent of $H$, by ergodicity. That is, there exists a normal subgroup $L \leq N$ such that $H = L \times (H \cap M)$ for almost every $H \in \sub_G$. Thus, $\mu = \mu_L \otimes \mu'$ where $\mu'$ is an invariant ergodic measure supported on the image of $\sub_M$ in $\sub_G$. Since $M$ acts faithfully and ergodically on $(\sub_M,\mu')$, we deduce from the previous case that $\mu'=\mu_\Lambda$ for a lattice $\Lambda$ in $M$. Finally, it is easy to check that $\mu = \mu_L \otimes \mu_\Lambda = \mu_{L \times \Lambda}$. This completes the proof of Theorem \ref{weak}.
\end {proof}

\medskip

The proofs of the uniform approximation results in the higher rank case relies on the following:

\begin{thm}\label{weak2}
Let $G$ be a center-free semi-simple Lie group of $\mathbb{R}$-rank at least $2$ with Kazhdan's property $(T)$.
Then $\mu _{\mathrm{id}}$ is the only accumulation point of the set
\[
\left\{ \mu _{\Gamma }\mid \Gamma \text{ is an irreducible lattice in }G\right\}.
\]
\end{thm}

We will make use of the following result.

\begin{thm}[Glasner--Weiss \cite{GlasnerKazhdan}]\label{lem:ergodiclimit}
  Let $X$ be a compact topological space, and let $G$ be a topological group with property (T) acting continuously on $X$. Let $(\mu_n)$ be a sequence of $G$-invariant Borel probability measures on $X$ that weakly converges to $\mu_\infty$. If the measures $\mu_n$ are ergodic, then the limit measure $\mu_\infty$ is ergodic.
\end{thm}

\begin{proof}[Proof of Theorem \ref {weak2}]
Fix a sequence $\Gamma _{n}$ of distinct irreducible lattices in $G$ such that $\mu_n :=\mu_{\Gamma _{n}}$ weakly converges and let $\mu_{\infty}$ be the limit measure. An important point here is that $\mu_{\infty}$ is ergodic with respect to the action of every simple factor of $G$.
By our assumption, if $N$ is any simple factor of $G$ then it has property $(T)$. Therefore, by Theorem \ref{lem:ergodiclimit}, $N$ acts ergodically on $\mu_\infty$. Combining this with Theorem \ref{thm:SZ-rk1}, we deduce that
either $\mu_\infty = \mu_N$ for a normal subgroup $N \leq G$, or $\mu_\infty = \mu_\Lambda$ for an irreducible lattice $\Lambda < G$.

Let us start by ruling out the possibility that $\mu_{\infty} = \mu_N$ for any connected normal subgroup of positive dimension. Since $N$ is not nilpotent, by Proposition \ref{conv-to-nilpotent} there is a neighborhood $U$ of $N$ in $\sub_G$ that does not contain any lattice.
Thus, if $\mu_n$ weakly converges to $\mu_N$ we would have
$$0 = \liminf_{n\to\infty} \mu_n(U) \geq \mu_N(U) =1,$$
which is absurd.

Next, we exclude the case that $\mu_{\infty} = \mu _{\Gamma}$ for a lattice $\Gamma$ in $G$.  By our assumption $G$ has property $(T)$. Therefore by a theorem of Leuzinger \cite{Leuzinger} there is a uniform lower bound for $\lambda_1 (\Gamma_n \backslash X) $, the first nonzero eigenvalue of the Laplacian operator on $\Gamma_n \backslash X $. Furthermore, since $(\mu_n) $ is not eventually constant the co-volumes of $\Gamma_n $ must tend towards infinity by Wang's Finiteness Theorem \cite[8.1]{wang}.  Theorem \ref {weak2} then follows from the following lemma.

\begin{lem}
\label{discrete}Let $G$ be a semi-simple Lie group, let $X$ be its associated Riemannian symmetric space  and let
$\Gamma _{n}$ be a sequence of lattices in $G$ where the covolume of $\Gamma _{n}$ tends to
infinity and $\mathrm{inf}\, \lambda_1 (\Gamma_n \backslash X) > 0$. Then the set $\left\{ \mu _{\Gamma_{n}}\right\} $ is discrete.
\end{lem}
\begin{proof} For simplicity, we will first describe the proof when all the $\Gamma _{n}$ are torsion-free, and afterward indicate the modifications necessary to deal with torsion.

Assume that after passing to a subsequence, $\mu
_{\Gamma _{n}}$ weakly converges to $\mu _{\Gamma }$ for some lattice $%
\Gamma $ in $G$.   As these measures are supported on the conjugates of their defining lattices, after conjugations and passing to a further subsequence we can assume that $\Gamma_n $ converges to $\Gamma  $ in the Chabauty topology.  By Proposition \ref{p:chabvsGH} and Lemma \ref{collapse}, this implies that after a suitable choice of base points the manifolds $Y_n = \Gamma_n \backslash X$ converge to $Y =\Gamma \backslash X $ in the pointed smooth topology.

If $Y$ is compact then the sequence $(\Gamma_n)$ is eventually constant, contradicting the fact that the co-volumes tend to infinity.  Otherwise, for every $\delta >0 $ there is a codimension-zero submanifold $B_\delta \subset Y$ with 
$$\frac { \vol (Y)} {2}  \leq \vol (B_\delta) \leq \vol(Y) \  \ \text {  and   } \ \ \vol ^ {d-1}( \partial B_\delta) < \delta.$$

This implies that for large $n $, there is a subset $B_n \subset Y_n $ such that
$$\frac { \vol (Y)} {4}  \leq \vol (B_n) \leq 2\vol(Y) \  \ \text {  and   } \ \ \vol ^ {d - 1}(\partial B_n) < 2\delta,$$
where if $d=\dim X$ then  $\vol^{d-1}$ is $(d-1)$-dimensional volume. As $\vol (Y_n) \to \infty $, we have $\vol (Y_n \setminus B_n)\to \infty $ as well.   This implies that for large $n $ the \emph {Cheeger constant}
$$h(Y_n) := \inf_{B \subset Y_n}  \frac {\vol ^ {d-1}( \partial B)}{\min \{ \vol(B), \vol(Y_n \setminus B)} \leq \frac {8\delta} {\vol (Y)} ,$$  where the infimum is over codimension-zero submanifolds of $Y_n$. As $\delta $ was arbitrary, this implies that $h(Y_n)\to 0$. An inequality of Buser \cite{Busernote} then implies that $\lambda_1(Y_n)\to 0$,  contradicting the uniform spectral gap condition.

 Morally, the proof for orbifolds is the same, but we cannot rely on smooth convergence because Lemma \ref{collapse} applies only to manifolds. However, one can proceed as follows: choose a large metric ball $B \subset X$ whose projection to $Y$ is nearly full measure, but where $\partial B$ projects to have small $(d-1)$-dimensional volume.  Fixing $\epsilon>0$, we can verify that the projection of $\partial B$ has small volume by choosing a small number of $\epsilon$-balls whose $\Gamma$-translates cover an $\epsilon/2$-neighborhood of $\partial B$. For large $n$, the projection of  $B$ to $Y_n$ will still have volume bounded below. In $X$, a neighborhood of the boundary $\partial B$ will still be covered by the $\Gamma_n$-translates of our $\epsilon$-balls above, so its projection in $Y_n$ has small volume.  This is enough to force the first eigenvalue $\lambda_1(\Gamma_n \backslash X)\to 0$, compare with \cite[Proposition 2.1]{Biringerfiniteness}.
\end{proof}

In summary, we have shown that the only possible accumulation point of the set
\[
\left\{ \mu _{\Gamma }\mid \Gamma \text{ is a lattice in }G\right\}
\]
is $\mu _{\mathrm{id}}$. On the other hand $\mu _{\mathrm{id}}$ is clearly an accumulation point.  For instance, if $\Gamma < G$ is any lattice,  then by residually finiteness,  there is a chain of  finite index normal subgroups  $\Gamma_n< \Gamma$  with trivial intersection.  Then $\Gamma_n\backslash X$ BS-converges to $X$, and by  Corollary \ref{lem:loccv}, $\mu_{\Gamma_n}\to\mu_{\rm id}$.
Hence we have proved Theorem \ref{weak2}.
\end{proof}

\begin{cor} \label{cor:4.15}
Let $G$ be a  center-free semi-simple Lie group with $\mathbb{R}$-rank at least $2$ and Kazhdan's property $(T)$ and let $(\Gamma _n)_{n\geq 0}$ be a sequence of irreducible lattices in $G$ where the covolume of $\Gamma_n$ tends to infinity. Then $\Gamma _n \backslash X$ BS-converges to $X$.
\end{cor}

As a consequence we have:

\begin{cor}\label{cor:M_{<r}}
Let $G$ and $X$ be as above. Then for every $r>0$ and for every sequence of $X$-orbifolds $M_n$ with $\vol(M_n)\to\infty$ one has
$$
 \frac{\vol((M_n)_{<r})}{\vol(M_n)}\to 0.
$$
\end{cor}

In Section \ref{sec:7} we will see how to use Corollary \ref{cor:M_{<r}} to obtain estimates on the growth of Betti numbers.
In particular, if $(\Gamma_n)_{n\geq 0}$ is a uniformly discrete sequence of nonconjugate lattices in a higher rank, center-free simple Lie group, then the hypotheses of Theorem \ref{T1} and Corollaries \ref{T3} and \ref{C84} follow from Corollary \ref{cor:4.15}. In particular the convergence of volume-normalized Betti numbers (Corollary \ref{cor4}) follows.

\section{Benjamini--Schramm convergence for congruence lattices} \label{sec:5} \label{5}

Let $G$ be a semisimple real simple Lie group, $X=G/K$ its associated symmetric space and let $\Gamma_0 \subset G$ be a uniform irreducible arithmetic lattice. We will assume that $\Gamma_0$ is torsion free, so in particular $\Gamma_0$ intersects the center $Z(G)$ of $G$ at the identity. There exists a $k$-simple, simply connected algebraic group $\mathbf{G}$ defined over a totally real number field $k$ such that $\Gamma_0$ is commensurable with $\mathbf{G}(\mathcal{O}_S)$, the group of $S$-integral points of $\mathbf{G}$. The {\it principal
congruence subgroups} of $\Gamma_0$:
$$\Gamma_0 (N) = \{ \gamma \in \Gamma_0 \cap \mathbf{G}(\mathcal{O}_S) \; : \; \gamma \equiv \mathrm{id} \ \mathrm{mod} \ N \}$$
obviously form a BS-convergent sequence of lattices in $G$. One may even quantify this observation:

\begin{lem} \label{Lprinccong}
There are constants $a>0$ and $b$, depending on $\Gamma$, such that for all $N\geq 1$,
$$\mathrm{InjRad} (\Gamma_0 (N)) \geq a \log \mathrm{vol} (\Gamma_0 (N) \backslash X) + b.$$
\end{lem}
Here we denote by $\mathrm{InjRad} (\Gamma)$ the infimum over $x \in \Gamma \backslash X$ of the
local injectivity radii $\mathrm{InjRad}_{\Gamma \backslash X} (x)$.

The conclusion of Lemma \ref{Lprinccong} does not hold for general congruence lattices (i.e. lattices which contain a principal congruence subgroup), as shown in the
following example.

\medskip
\noindent
{\it Example.} Let $\mathbf{H}$ be a proper $k$-subgroup of $\mathbf G$ which contains a semi-simple element of $\mathbf{G}$. Consider the congruence subgroups of $\Gamma_0$:
$$\{  \gamma \in \Gamma_0 \cap \mathbf{G}(\mathcal{O}_S) \; : \; \gamma \in \mathbf{H} (\mathcal{O}_S) \ \mathrm{mod} \ N \}.$$
These form a sequence of cocompact lattices in $G$ whose volumes tend to infinity but
whose (minimal) injectivity radius remains bounded (in fact it becomes stationary).

\medskip

It nevertheless remains true that any sequence of distinct congruence subgroups of $\Gamma_0$ locally
converges toward the trivial group. 

The main result of this section is to prove the following quantified version of the above statement:

\begin{thm} \label{nik}
There exist positive constants $c$ and $\delta$ depending only on $\Gamma_0$ (and $G$), such that for any congruence subgroup $\Gamma \subset \Gamma_0$ and any $R>1$ we have:
$$\mathrm{vol} ((\Gamma \backslash X)_{<R}) \leq e^{cR} \mathrm{vol} (\Gamma \backslash X)^{1-\delta}.$$
\end{thm}

The proof of this theorem is given in the rest of the section. 
\medskip

We note that if $\Gamma$ is an arithmetic lattice that has the Strong Approximation Property below, then one can prove uniform BS-convergence for finite index subgroups of $\Gamma$ without having to use Theorem \ref{nik}. 

An arithmetic lattice $\Gamma$ in a Lie group $G$ has the \emph{Strong
Approximation Property (SAP)}, if for every Zariski dense subgroup $H$ of $\Gamma$, the closure of $H$ in the congruence completion of
$\Gamma$ is open. 
This is equivalent to the statement that for every Zariski dense subgroup $H$ there
exists $M>0$, such that for any congruence subgroup $K$ of $\Gamma$, we have
$\left\vert \Gamma:HK\right\vert \leq M$. For $\Gamma$ with SAP,  one can prove uniform BS-convergence for congruence subgroups of $\Gamma$ without having to use Theorem \ref{nik}:

\begin{thm}
Let $\Gamma$ be an arithmetic lattice with SAP and assume that $\Gamma$ has trivial center. Then for any sequence of
congruence subgroups $\Gamma_{n}$ of $\Gamma$ with $[ \Gamma
:\Gamma_{n}] \rightarrow\infty$ we have $\mu_{\Gamma_{n}}
\rightarrow\mu_{1}$.\label {easynik}
\end{thm}

SAP\ has been proved by Nori \cite{Nori} and Matthews, Vaserstein and Weisfeiler \cite{Weisfeiler} for arithmetic lattices in simple, connected, simply connected groups $G$.  So,

\begin{cor}
Let $\Gamma$ be an arithmetic lattice in a simple connected Lie group $G$ and assume that $\Gamma$ has trivial center. Then for any sequence of
congruence subgroups $\Gamma_{n}$ of $\Gamma$ with $[ \Gamma
:\Gamma_{n}] \rightarrow\infty$ we have $\mu_{\Gamma_{n}}
\rightarrow\mu_{1}$.
\end{cor}
\begin{proof}
	The group $\Gamma$ is commensurable with an arithmetic lattice $\Gamma_0$ in the universal cover of $G$. By replacing $\Gamma$ with $\Gamma_0$ and using \cite{Nori} and \cite{Weisfeiler} we can therefore reduce the corollary to the previous theorem.
\end{proof}

\begin{proof}[Proof of  Theorem \ref{easynik}]
Assume, by contradiction, that $\mu_{\Gamma_{n}}$
does not converge to $\mu_{1}$. By passing to a subsequence, we can assume
that $\mu_{\Gamma_{n}}\rightarrow\mu$ where $\mu\neq\mu_{1}$. Let $K$ be the
$\mu$-random subgroup of $\Gamma$.
Since $[\Gamma:\Gamma_{n}] \rightarrow\infty$, $K$ has
infinite index in $\Gamma$ a.s. We shall prove that $K=1$
a.s., to reach a contradiction. 

Let us say that a subgroup $H$ of $\Gamma$ is {hyperclosed}, if it can be obtained as an
ascending union $H=\bigcup_{k=1}^\infty J_{k}$ of congruence closed subgroups $J_k \subset \Gamma$. We claim that a Zariski
dense hyperclosed subgroup $H$ of $\Gamma$ has finite index in $\Gamma$.
Indeed, by a standard dimension argument, there exists some $k$ such that
$J_{k}$ is already Zariski dense. By SAP the congruence closure of $J_{k}$
(which equals $J_{k}$) has finite index in $\Gamma$ and so does $H$.

Now given a sequence of congruence subgroups $H_{n}$ of $\Gamma$, in the Chabauty
topology, we have
\[
\lim H_{n}=\bigcup\limits_{k=1}^{\infty}\bigcap\limits_{n=k}^{\infty}%
H_{n}\text{.}%
\]
Since the intersection of congruence closed subgroups is congruence closed, we
get that $\lim H_{n}$ is hyperclosed. The set of possible Chabauty limits of
congruence subgroups is compact in the Chabauty topology, so we obtain that
$\mu$ is supported on hyperclosed subgroups. Applying SAP on hyperclosed
subgroups and that $K$ has infinite index in $\Gamma$ a.s., gives us that $K$
is not Zariski dense in $\Gamma$ a.s. By Theorem \ref{discIRS}, this implies that $K=1$
a.s., a contradiction. 
\end{proof}

\subsection{Proof of Theorem \ref{nik}} \label{red} 

We first reduce the proof to the case where $\Gamma_0$ is a finite index subgroup of $\mathbf{G} (\mathcal{O}_S)$.
In fact we will show that if $\Gamma_0$ and $\Delta_0$ are two arithmetic commensurable  torsion free lattices, then the conclusion of Theorem \ref{nik} holds for the congruence subgroups of $\Gamma_0$ if and only if it holds for the congruence subgroup of $\Delta_0$ provided we change the constant $c$.

By considering $\Gamma_0 \cap \Delta_0$ inside $\Gamma_0$ and inside $\Delta_0$ we see that it is sufficient to prove the case when $\Delta_0$ is a finite index subgroup of $\Gamma_0$.

Let $\Gamma \subset \Gamma_0$ be a congruence subgroup and denote by $M$ the corresponding $X$-manifold $\Gamma \backslash X$.
Set $\Delta  = \Gamma \cap \Delta_0$ and $M' = \Delta \backslash X$.
Let $p:M' \rightarrow M$ be the covering map and $m:= [\Gamma_0 :\Delta_0]$. Then $[\Gamma : \Delta] \leq m$ and $p$ is of degree at most $m$ so that for any $x\in M$ and $x'\in M'$ with $p(x')=x$ we have:
\[ \frac{\mathrm{InjRad}_{M'}(x')}{m} \leq \mathrm{InjRad}_M (x) \leq \mathrm{InjRad}_{M'}(x') . \]
In particular:
\[ \mathrm{Vol}(M_{<R}) \leq \frac{\mathrm{Vol}(M'_{<mR} )}{[\Gamma : \Gamma ']} \leq \mathrm{Vol}(M_{<mR} ).\]

In turn $\mathrm{vol}(M')\leq \mathrm{vol}(M)m$. So if we have the inequality 

\[ \mathrm{vol} (M'_{<mR}) \leq e^{cmR} \mathrm{vol} (M')^{1-\delta},\]
for some $c$ and $\delta$ then by changing the constant $c$ (to $cm+(1-\delta)\log m$) we obtain the corresponding inequality 

\[ \mathrm{vol} (M_{<R}) \leq e^{cmR+(1-\delta)\log m} \mathrm{vol} (M)^{1-\delta}.\]

The other direction is even easier: starting with a congruence subgroup $\Delta$ containing $\Delta_0(N)$ for some integer $N$ we put $\Gamma:= \Delta \Gamma_0(N)$, a congruence subgroup in $\Gamma$, and observe that $\Delta= \Gamma \cap \Delta_0$. The inequality for $\mathrm{vol}((\Gamma \backslash X)_{<R})$ easily gives a corresponding inequality for $\mathrm{vol}((\Delta \backslash X)_{<R})$.

 So it is  sufficient to prove Theorem \ref{nik} in the case when $\Gamma_0$ is any given finite index subgroup of $\mathbf{G} (\mathcal{O}_S)$.

\medskip

We will first prove the following combinatorial version of Theorem \ref{nik}:

\begin{thm} \label{ag} Let $\mathbf{G}$ be a $k$-simple simply connected algebraic group defined over a number field $k$. For a finite set of valuations $S$ of $\mathcal O$, including all archimedian ones, let $\mathcal{O}_S$ be the ring of $S$-integers in $k$. There exists some finite index center-free subgroup $\Gamma \subset \mathbf{G} (\mathcal{O}_S)$  and some positive constants $\epsilon$ and $C$ (depending only on $\Gamma$ and some fixed word metric on it) with the following property: 

Let $g \in \Gamma - \{1\}$ and let $H$ be a congruence subgroup of index $N$ in $\Gamma$. Then $g$ fixes at most $e^{C  l(g)} N^{1-\epsilon}$
points in the action of $\Gamma$ on the right cosets $H \backslash \Gamma$ by multiplication. Here $l(g)$ is the length of $g$ with respect to the fixed word metric of $\Gamma$.
\end{thm}

 The proof of  this theorem that appeared in our original 2012 arXiv preprint had a mistake in Proposition \ref{pp} below,  a correct proof of which has since been given by Finis--Lapid \cite{FinLa},  who also  obtain explicit bounds in  Theorem \ref {ag}. Here we give a self-contained proof based on the theory of $p$-adic analytic groups, which avoids the algebraic geometry arguments in \cite{FinLa}. We only need the basic Lemma \ref{polycong} on the solutions of polynomial congruences and well-known results on the fixity of permutation actions of simple groups of Lie type. A careful examination of all the steps of the proof can lead to an explicit value of $\epsilon$, at least for Chevalley groups, on the order of magnitude $(\dim \mathbf G)^{- [k:\mathbb Q] \dim \mathbf G}$, which seems however very far from optimal.

We postpone the proof of Theorem \ref{ag} and first show how it implies Theorem~\ref{nik}.

\subsection{Proof of Theorem \ref{nik}}
According to \S \ref{red} we may assume that $\Gamma_0$ is the finite index subgroup of $\mathbf{G} (\mathcal{O}_S)$ given by Theorem \ref{ag}.
Let $\Gamma \subset \Gamma_0$ be any congruence subgroup.

Let $\Omega\subset X$ be a compact fundamental domain for the action of $\Gamma_0$ on $X$ and let $p: M =\Gamma \backslash X \rightarrow M_0=\Gamma_0 \backslash X$ be the covering map. We will identify the elements of $M$ (resp. $M_0$) with the orbits of $\Gamma$ (resp. $\Gamma_0$) in $X$.

Suppose that $y \in M$ has $\mathrm{InjRad}_M(y)<R$. Let $x$ be a lift of $y$ in $X$, i.e. $y=\Gamma x \in \Gamma \backslash X$. We have that $d(x, \gamma x) <2R$ for some $\gamma \in \Gamma$.

Now let $g$ be the unique element of $\Gamma_0$ such that $g^{-1}x=x_0 \in \Omega$. We have:
$$d(x,\gamma x)=d(x_0,g^{-1}\gamma x)=d(x_0, \gamma^{g} x_0)<2R$$
where $\gamma^g=g^{-1}\gamma g$. Since $\gamma^g$ moves the point $x_0$ of $\Omega$ to a point of distance at most $2R$ from it and since $\Omega$ is compact it follows that $l(\gamma^g)<C'R$ for some constant $C'$ depending only on the choice of $\Omega$ and a generating set (fixed by the choice of the word metric in Theorem \ref{ag}) of $\Gamma_0$.

Now, given the element $x_0 \in \Omega$ and a nontrivial $\gamma_0 \in \Gamma_0$ with $l(\gamma_0 )<C'R$ suppose that for some $x=gx_0 \in X$ ($g \in \Gamma_0$) there exists
$\gamma \in \Gamma$ with $d(x,\gamma x)<2R$ and $\gamma^g = \gamma_0$. Then
$g^{-1}\gamma g= \gamma_0$ so that $\Gamma g=\Gamma g \gamma_0$. The number of $\Gamma$-equivalence classes of points $x=gx_0$ in  $\Gamma x_0$ as above giving rise to the same
$\gamma_0$ is therefore equal to the cardinal of the set $\mathrm{fix}(\gamma_0, \Gamma \backslash \Gamma_0)$ of fixed points of $\gamma_0$ acting on $\Gamma \backslash \Gamma_0$.

Therefore
\[ \mathrm{Vol}(M_{<R}) \leq \mathrm{Vol} (\Omega) \sum_{0<l(\gamma_0) <C' R } |\mathrm{fix}(\gamma_0, \Gamma \backslash \Gamma_0)|. \]
In turn by Theorem \ref{ag}
$$|\mathrm{fix}(\gamma_0, \Gamma \backslash \Gamma_0)|\leq e^{Cl(\gamma_0)} [\Gamma_0:\Gamma]^{1- \epsilon}$$
and there are at most $e^{C''R}$ elements $\gamma_0$ with $0<l(\gamma_0)<C' R$ which combine to give the desired bound for large enough $c$.\qed

\subsection{} \label{mu} The proof of Theorem \ref{ag} takes up the rest of this section. We can consider  $\mathbf{G}(k)$ as the rational points of a restriction of scalars of an absolutely simple group defined over a larger field, moreover the respective groups of integral points and their congruence topologies are compatible. So by enlarging the field $k$ if necessary we may assume from the start that $\mathbf{G}$ is absolutely simple.

Take a prime ideal $\mathfrak P$ of $\mathcal O$ and let $p$ be the rational prime such that $\mathfrak P |p$. Let $r_\mathfrak P$ be the ramification index of $\mathfrak P$ and let $w_\mathfrak P$ be its residue degree, i.e. $|\mathcal O:\mathfrak P|=p^{w_\mathfrak P}$. We have $p=\prod_{\mathfrak P|p} \mathfrak P^{r_{\mathfrak P}}$ and $[k: \mathbb Q]=\sum_{\mathfrak P|p} r_{\mathfrak P}w_{\mathfrak P}$. 

From now on we will denote by $\mathfrak P$ a prime ideal of $\mathcal O_S$ dividing a rational prime $p$.
We will denote by $|-|_{\mathfrak P}$ the $\mathfrak P$-adic valuation on $k$ defined by $|x|_{\mathfrak P}=p^{-w_{\mathfrak P} n}$ for $x \in \mathfrak P^n \backslash \mathfrak P^{n+1}$. Denote by $k_{\mathfrak P}$ and $\mathcal O_{\mathfrak P}$ the completions of $k$ and $\mathcal O_S$ with respect to this valuation. We have $[k_{\mathfrak P}: \mathbb Q_p]=r_{\mathfrak P}w_{\mathfrak P}$. Let $G_{\mathfrak P} =\mathbf{G} (\mathcal O_{\mathfrak{P}})$ be the congruence completion of $\mathbf{G} (\mathcal{O}_S )$ with respect to the $\mathfrak P$-adic topology. For $m \geq 1$ let $G_{\mathfrak P} (m)$ be the principal congruence subgroup mod $\mathfrak P ^{r_{\mathfrak P}m}$, i.e. the matrices in $G_{\mathfrak P}$ which are congruent to the identity mod $\mathfrak P ^{r_{\mathfrak P}m}$. (The presence of $r_{\mathfrak P}$ in the exponent is to ensure that $(G_{\mathfrak P} (i))_{i=1}^ \infty$ is the Frattini series of the $p$-adic analytic group $G_{\mathfrak P}(1)$.) The dimension of $G_{\mathfrak P}$ as an analytic group over $\mathbb Q_p$ is $r_{\mathfrak P}w_{\mathfrak P} \dim \mathbf G= [k_{\mathfrak P}: \mathbb Q_p] \dim \mathbf G$. Note that all but finitely many of the prime ideals $\mathfrak P$ are unramified, i.e. $r_{\mathfrak P}=1$. For almost all unramified prime ideals $\mathfrak P$ the quotient $S_{\mathfrak P}:=G_{\mathfrak P}/G_{\mathfrak P} (1)$ is the reduction $\mathbf G (\mathcal O_S / \mathfrak P)$ of $\mathbf G$ mod $\mathfrak P$. Since $\mathbf G$ is absolutely simple and simply connected it follows that $S_{\mathfrak P}$ is generated by its unipotent elements and is therefore a finite quasi-simple group $S_{\mathfrak P}$ of Lie type over the field $\mathcal O_S / \mathfrak P$, see Proposition 7.5 of \cite{Plat}. (A perfect group is called quasi-simple if it is simple modulo its center).  Moreover the Frattini subgroup $\Phi(G_{\mathfrak P})$ contains $G_{\mathfrak P}(1)$,  see \cite{sg} Lemma 16.4.5. Let us call these prime ideals \emph{good} and all other prime ideals \emph{bad}. Let $Z$ be the finite center of $\mathbf G$.  

For a rational prime $p$ define $G_p:= \prod_{\mathfrak P |p, \mathfrak P \not \in S} G_{\mathfrak P}$ and for $m \in \mathbb N$ let $G_p(m):= \prod_{\mathfrak P |p, \mathfrak P \not \in S} G_{\mathfrak P}(m)$.
That is $G_p(m)$ consists of the elements of $G_p$ congruent to 1 mod $p^m$. The group $G_p$ is a semisimple $p$-adic analytic group of dimension at most $D:=[k:\mathbb Q] \dim \mathbb G$. (The dimension of $G_p$ is exactly $D$ when the set $S$ avoids any ideal divisors of $p$.) The level of an open subgroup $H \leq G_p$ is defined to be the integer $n$, such that $H$ contains $G_p(n)$ but not $G_p(n-1)$.

Let $P=\{ \mathfrak{P}_1 , \ldots , \mathfrak{P}_r\}$ be a finite set of prime ideals of $\mathcal O_S$ including all the bad primes such that  the principal congruence subgroup
$$\Gamma = \left\{ g \in \mathbf{G} (\mathcal O_S ) \; : \; g \equiv \mathrm{id} \ \mathrm{mod} \ \mathfrak{P}_1 \cdots \mathfrak{P}_r \right\}$$ intersects $Z$ trivially.

We will show that $\Gamma$
satisfies the conclusion of Theorem \ref{ag}. From now on we fix an element $g \in \Gamma \backslash \{1 \}$.

\subsection{}
The congruence subgroups of $\mathbf G(\mathcal O_S)$ correspond to open subgroups of its congruence completion
\[ \widetilde{G} = 
\prod_{p \textrm{ prime}} G_p = \prod_{\mathfrak P \not \in S} G_{\mathfrak P}.\]

The strategy is to reduce the fixity estimate of Theorem \ref{ag} to an analogous problem inside each local $p$-adic factor $G_p$. This reduction follows easily from the detailed knowledge of the subgroups and representations of the simple factors  $S_{\mathfrak P}$ for different prime ideals, and in particular the elementary fact that $S_{\mathfrak P}$ has no proper subgroups of index less than $p$. We then solve the local problem (Proposition \ref{pp}) using the natural coordinate system of uniform subgroups of $G_p$. To formulate the local estimate we need some further notation.

Let $L$ be the Lie algebra of $\mathbf G$, which is a simple Lie algebra defined over $k$. Let $L_{\mathfrak P}$ be the $\mathcal O_{ \mathfrak P}$-Lie ring which corresponds to the uniform pro-$p$ group $G_{\mathfrak P}(1)$, so $L_{\mathfrak P}$ is an $\mathcal O_{\mathfrak P}$-lattice of the $k_{\mathfrak P}$-Lie algebra $L \otimes_k k_{\mathfrak P}$. There is a free $\mathcal O_S$-lattice $L_0$ of $L$ such that for almost all prime ideals $\mathfrak P$ we have $L_{\mathfrak P}= L_0 \otimes_{\mathcal O_S} \mathfrak P O_{\mathfrak P}$ (The two lattices are open in  $L \otimes_k k_{\mathfrak P}$ and so commensurable for all $\mathfrak P$.) Let $\mathrm{Ad}_{\mathfrak P}: G_{\mathfrak P} \rightarrow \mathrm{End}(L_{\mathfrak P})$ be the adjoint representation on $L_{\mathfrak P}$. Define $n_{\mathfrak P}(g)$ to be the largest integer $m$ such that
$\mathrm{Ad}_{\mathfrak P}(g) \equiv 1$ mod $\mathfrak P ^m$. 

Note that since  $g \not \in Z$ by our choice of $\Gamma$, the integers $n_{\mathfrak P}(g)$ exist.

Define $n_p(g)= \max_{\mathfrak P | p} n_{\mathfrak P}(g)$ (where the maximum is over all prime ideals $\mathfrak P$ of $\mathcal O_S$ dividing $p$). For completeness we set $n_p(g)=0$ for those primes $p$ which are units in $\mathcal O_S$.
For example when $\mathbf G(\mathcal O_S)=\mathrm{SL}_n(\Z)$ then $L_p=\mathfrak{sl}_n(\Z_p)$ in which case $n_p(g)$ is the largest integer $m$ such that $g$ is congruent to a scalar matrix modulo $p^m$.

Since $g$ is not in the centre of $G$ we have $\mathrm{Ad}(g)-1 \not = 0$ as a linear endomorphism of $L$. Choose any nonzero matrix coefficient $\beta \in k$ of the matrix of $\mathrm{Ad}(g)-1$ with respect to a fixed $\mathcal O_S$-basis of $L_0$. For any prime $p$ which is not a unit of $\mathcal O_S$, by the definition of $n_p(g)$ there is a prime ideal $\mathfrak P$ of $\mathcal O$ outside $S$ which divides $p$ and such that $|\beta|_{\mathfrak P} \leq |\mathcal O: \mathfrak P|^{n_p(g)} \leq p^{-n_p(g)}$. (Here, we may need to change $\beta$ by a multiplicative constant to take account for the finitely many prime ideals such that $L_{\mathfrak P} \not = L_0 \otimes_{\mathcal O_S} \mathfrak P O_{\mathfrak P}$.)

 Hence $\prod_{\mathfrak P \not \in S} |\beta|_{\mathfrak P } \leq \prod_p p^{-n_p(g)}$. On the other hand, for any valuation $v$  the coefficient $|\beta| _v$ is bounded above exponentially by the word length $l(g)$. Putting it all together we conclude that   $\prod_{v \in S} |\beta|_v \leq e^{C_1 l(g)}$ for some constant $C_1$ depending only on $\Gamma$ (and the  basis of $L_0$ and the word metric on $\Gamma$). Note that \[ \prod_v |\beta|_v=1 \]
 where the product runs over all valuations of $\mathcal O$.
 Therefore all but finitely many of the $n_p(g)$'s are zero and (compare with Lemma \ref{Lprinccong})
\begin{equation} \label{l}
\prod_p p^{n_p(g)} \leq e^{C_1 l(g)}.
\end{equation}

Given a congruence subgroup $\Gamma '$ of $\Gamma$ we want to compute the {\it fixity ratio}
$$\alpha(g,\Gamma' \backslash \Gamma ) = \frac{|\mathrm{fix} (g , \Gamma' \backslash \Gamma)|}{[\Gamma : \Gamma ']}$$
of $g$ acting the right cosets of $\Gamma'$ in $\Gamma$ by multiplication. 

The congruence completion of $\mathbf G(\mathcal O_S)$ is 
$\tilde G:=\prod_p G_p$. Let $\bar \Gamma$ be the closure of $\Gamma$ in $\tilde G$ and let $H$ be the closure of $\Gamma'$.
We have now reduced our claim to showing that there exists
some positive constants $C$ and $\delta$ depending only on $\Gamma$ such that
\begin{equation} \label{red1}
\alpha (g,  H \backslash \bar \Gamma) \leq  e^{C l(g)} [\bar \Gamma : H ]^{-\delta}.
\end{equation}

\subsection{Reduction to the local case} We have $\bar \Gamma= \prod_p \Gamma_p$ is a product of its projections $\Gamma_p$ onto $G_p$. Moreover $\Gamma_p=G_p$ for almost all $p$ and always $\Gamma_p \geq G_p(1)$. Let $H_{p}$ be the projection of $H$ onto the factor $\Gamma_p$ of $\bar \Gamma$. We have that $H \leq \prod_p H_{p}$ and so
\begin{equation} \label{in}
\alpha(g,H \backslash \tilde \Gamma) \leq \alpha(g, \prod_p  H_{p} \backslash \tilde \Gamma ) = \prod_p \alpha(g,  H_p \backslash \Gamma_p).
\end{equation}

Since $H$ contains $G_p$ for almost all primes $p$ the product above is equal to a finite product.

Let $x_p= [\Gamma_p: H_{p}]$. Clearly $N=[\Gamma: \Gamma']=[\bar \Gamma:H] \geq \prod_p x_p$, where again we have that $x_p=1$ for almost all primes $p$.

\begin{lem} \label{b}
There is a positive constant $\delta_1$, such that
$$N^{\delta_1} \leq \prod_p x_p.$$
\end{lem}
\begin{proof} Let $\Delta= \prod_p G_p(1) \leq \tilde \Gamma$. Each $G_p(1)$ is a pro-$p$ group and hence $H \cap \Delta$ is the direct product of its projections onto the Sylow pro-$p$ factors $G_p(1)$ of $\Delta$. Since
$$N=[\tilde \Gamma:H\Delta][\Delta: H \cap \Delta],$$
it is enough to prove the special case when $H=H\Delta$ i.e. $H \geq \Delta$. Then $H/\Delta \leq \Gamma/\Delta= \prod_{\textrm{good }\mathfrak P} S_{\mathfrak P}$ where by the choice of $\Gamma$ the product runs over the good prime ideals $\mathfrak P$ of $\mathcal O_S$. Thus each $S_{\mathfrak P}$ is a finite quasisimple group of Lie type of bounded rank of characteristic $p$. Let $S_p=\prod_{\textrm{good }\mathfrak P|p} S_{\mathfrak P}$ and let us denote by $Lie(p)$ the set of simple groups of Lie type over fields of characteristic $p$. 
We can replace $H$ with its image in a finite product $\prod_{i=1}^s S_{p_i}$ such that $H$ does not contain any of the factors $S_{p_i}$. Then $H_{p_i}$ becomes the projection of $H$ in the factor $S_{p_i}$.
By the Larsen-Pink theorem \cite{Larsen}, there is a function $f: \mathbb N \rightarrow \mathbb N$ such that a subgroup of $GL(n,\mathbb F_{p^m})$ contains a subgroup of index at most $f(n)$ whose non-abelian composition factors are from $Lie(p)$. In particular if $q>f:=f(\dim \mathbb G)$ is a large prime then the direct factors of $S_q$ cannot occur as composition factors of any subgroup of any of the factors of $S_p$ for $p \not =q$.
It follows that if $q>f$ is a prime and $H_q=S_q$, i.e. $x_q=1$ then actually $H \geq S_q$.  

 By enlarging the set $P$ defining $\Gamma$ to include all prime ideals dividing $q$ for primes $q \leq f$ we can ensure that $H_{p_i}$ is a proper subgroup of $S_{p_i}$. The group $S_{p_i}$ is a product of finite quasisimple groups from $Lie(p_i)$ and in particular it is generated by elements of order $p_i$. Hence a proper subgroup of $S_{p_i}$ must have index at least $p_i$ and $x_{p_i}=[S_{p_i}:H_{p_1}] \geq p_i$. On the other hand $|S_{p_i}| < p_i^{[k: \mathbb Q]\dim \mathbf G }$ and $N \leq \prod_{i=1}^s|S_{p_i}|$ So Lemma \ref{b} follows with $\delta_1=([k:\mathbb Q] \dim \mathbf G)^{-1}=D^{-1}$.
\end{proof}

\subsection{The local case.} Lemma \ref{b} reduces the proof of \eqref{red1} to its local counterpart \eqref{eqGH} below at each prime $p$.
Here we explain how to conclude the proof of Theorem \ref{ag} assuming the following:
\medskip

\begin{prop}\label{pr} There exist constants $C_2$ and $\delta_2>0$ depending only on $\Gamma$ such that for all primes $p$
\begin{equation} \label{loc}  \alpha (g,  H_{p} \backslash \tilde \Gamma) \leq p^{C_2n_{p}(g)}x_p^{-\delta_2}.
\end{equation}
\end{prop}

Multiplying the inequalities (\ref{loc})  and using (\ref{in}), (\ref{l}) and Lemma \ref{b} we obtain

\[ \alpha(g, H \backslash  \tilde \Gamma) \leq \prod_p p^{C_2n_{p}(g)} (\prod_p x_p)^{-\delta_2} \leq  e^{C_1 C_2 l(g)} N^{-\delta_1\delta_2}, \]  Theorem \ref{ag} follows. \qed
\medskip

By considering the image of $H_{p}$ in the local factor $G_{p}$,
Proposition \ref{pr} is easily deduced from the following local bound.

\begin{prop} \label{pp} There are constants $a,b,c>0$ which depend on $\mathbf G$ but not on the prime $p$ such that if $H$ is an open subgroup of $G_p$ of level $n$ and $g \in G_p-Z(G_p)$, then the fixity proportion $\alpha(g, H \backslash G_p)$ of $g$ on $H \backslash G_p$ is at most $p^{-cn}$ provided $n> \max \{a, bn_p(g)\}$.
\end{prop}

Let us show how Proposition \ref{pp} implies the existence of $C_2$ and $\delta_2$ such that 
\begin{equation} \label{eqGH} \alpha (g, H \backslash G_p) \leq p^{C_2n_p(g)} [G_p:H]^{-\delta_2}\end{equation} provided $g \in \Gamma \backslash \{1\}$.
This inequality easily gives (\ref{loc}) by increasing $C_2$ to take into account the index  $[G(\mathcal O_S): \Gamma]$. 

Since $H$ contains $G_p(n)$ from the definition of $n$, we have $[G_p:H] \leq [G_p:G_p(n)] \leq p^{n [k:\mathbb Q] \dim \mathbf G}=p^{n D}$, 
(where $D=[k:\mathbb Q] \dim \mathbf G$). Therefore
\begin{equation} \label{Gp} \alpha (g, H \backslash G_p) \leq [G_p:H]^{-c/D}. \end{equation}
provided $n> \max \{a, bn_p(g)\}$.

It remains to prove (\ref{eqGH}) when $n \leq a$ or $n \leq bn_p(g)$.
First we consider the case that $n_{p}(g)=0$. Then by the choice of $\Gamma$ we have that $p$ is a product of good prime ideals.

Since $n_p(g)=0$ it follows that $g$ is not in the centre of any of the factors of the semi-quasisimple group $S_p=G_p/G_p(1)=\prod_{\mathfrak P|p} S_{\mathfrak P}$. We also have $HG_p(1) <G_p$ because $G_p(1)$ is in the Frattini subgroup of $G_p$. Choose a maximal subgroup $M$ of $S_p$ containing the image of  $H$. There are two possibilities: 

1. $M$ is the preimage of a maximal subgroup $M_0 \leq S_{\mathfrak P}$ of one of the factors of $S_p$. The main result of \cite{shalev} says that with finitely many exception any non-trivial element of a simple group $S$ of Lie type over a field of size $q$ has fixity at most $4/3q$ in any primitive action of $S$ unless $S=\mathrm{PSL}(2,q)$ when it is at most $2/(q+1)$ for $q \geq 5$. The proof relies on a case by case analysis of the Aschbacher classification of the  maximal subgroups of $S$. It follows that $\alpha (g, M \backslash S_p) \leq \alpha (g, M_0 \backslash S_{\mathfrak P}) \leq 2/q$.

2. $M$ is the preimage of a diagonal subgroup $T$ in a product $S \times S$ of two isomorphic factors of $S_p$. Now a direct computation shows that for $(y_1,y_2) \in S \times S$ the fixity ratio
$\alpha((y_1,y_2), T\backslash  S \times S)$ is nonzero only if $y_1$ and $y_2$ are conjugate and if so then is at most $|C_S(y_1)|/|S|$. Since we are assuming that $y_1$ is non-central in $S$ and $S$ is generated by elements of order $p$ we have as before that $|C_S(y_1)|/|S| \leq 1/p$.

All together, we can deduce that

\[  \alpha (g, H \backslash G_p) \leq  \alpha (g, M \backslash S_p) \leq p^{-\epsilon_2}\] for example with $\epsilon_2=1/2$.

We set $\delta_2= \min \{ \frac{c}{D}, \frac{\epsilon_2}{a D}  \}$
and note that when $n \leq a$ then 
$$
 p^{-\epsilon_2} \leq p^{- a \delta_2 D} \leq [G_p:H]^{-\delta_2}.
$$

To deal with the case $n_p(g) \geq 1$  we simply set
$C_2= (a+b) \delta_2 D$, which ensures that $p^{C_2n_p(g)}[G_p:H]^{-\delta_2} \geq 1$ when $n \leq \max \{a, bn_p(g) \}$.

\subsection{Proof of Proposition \ref{pp}}

From now we fix a prime number $p$ and denote $n_p(g)$ by $n_p$.

The congruence completion $G_p$ is a $p$-adic analytic group of dimension $d=d(p)=\dim \mathbb G( \sum_{\mathfrak P|p, \mathfrak P \not \in S} [k_{\mathfrak P}: \mathbb Q_p])$ over $\mathbb Q_p$ and we shall refer to \cite{SSM} for standard results about these.
Note that $d \leq D= [k:\mathbb Q] \dim \mathbf G$. Recall that $G_p(i)$ is the kernel of $G_p \rightarrow G(\mathcal O_S / p^i)$. There exists a constant $c_0 \geq 1$ such that the congruence subgroup $G_p(c_0)$ is a uniform pro-$p$ group for all primes $p$. In fact \cite{sg}, Corollary 16.4.6 gives that $c_0=1$ for all except finitely many primes $p$. We set $c_1=c_0$ unless $p=2$, when we set $c_1=c_0+1$. Define $U=G_p(c_1)$. So $U$ is a uniform group of dimension $d$ and in addition when $p=2$, $U$ is the Frattini subgroup of a uniform group (we need this in order to apply Proposition 8.21 from \cite{SSM} at a later point). The series $(G_p(i))_{i \geq c_1}$ coincides with the Frattini series of $U$ defined by $U_0=U$ and $U_{i+1}=\Phi (U_i)$, i.e. $U_i=G_p(i+c_1)=\prod_{\mathfrak P|p}G_{\mathfrak P}(i+c_1)$ for all $i \geq 0$.
For a prime ideal $\mathfrak P$ of $\mathcal O_S$ dividing $p$ we define $U_{\mathfrak P,i}= G_{\mathfrak P}(i+c_1)$, so that $U_i=\prod_{\mathfrak P|p} U_{\mathfrak P,i}$.

Let $n$ be the level of $H$. From now on we denote $n'=n-c_1$, so $G_p(n)=U_{n'} \leq H$ but $U_{n'-1} \not \leq H$.

First observe that at the cost of increasing $a$ and $b$ and decreasing $c$, it is sufficient to prove Proposition \ref{pp} with $n'$ in place of $n$. \medskip

Next we introduce coordinates of the second kind for $U$, which are more suitable for parametrizing its open subgroups, in the spirit of \cite{duS}.
The properties of a uniform pro-$p$ group $U$ we state below can be found in \cite{SSM} Chapters 4 and 8.

 Suppose elements $e_1, \ldots, e_d$ in $U$ are given such that their images form a basis of the vector space $U/U_1=U/\Phi(U)$. 
Then $e_1^{p^i}, \ldots, e_d^{p^i}$ is a basis for $U_{i}/U_{i+1}$ for each $i \in \mathbb N$.

Moreover every element of $U$ can be expressed uniquely in the form $\prod_{i=1}^d c_i$ where $c_i \in \overline{\langle e_i \rangle} \simeq \mathbb Z_p$.
The following properties hold: \medskip

1. The map $\mu : \mathbb (\mathbb Z_p)^d \rightarrow U$ defined by $\mu(y_1, \ldots, y_d) = e_1^{y_1} \cdots e_d^{y_d}$ is a homeomorphism. 
We say that the elements $e_1, \ldots, e_d$ provide a \emph{basis for coordinates $\mu$ of the second kind for $U$.}

2. $U_i$ is the image of $(p^{i}\mathbb Z_p)^d$ under $\mu$.

3. If we identify $U$ with $(\mathbb Z_p)^d$ via $\mu$, the group operations in $U$ (including exponentiation $\exp_a(z)=a^z, \  z \in \mathbb Z_p$ for fixed $a \in U$) are given by a converging power series in $\mathbb Z_p[[x_1, \ldots, x_d]]$  with $(x_1, \ldots, x_d) \in (\mathbb Z_p)^d$. This follows from \cite{SSM} Proposition 8.21: the inequality (5) there, together with the formula for the power series $g_i$ on page 192 implies that the coefficients of $g_i$ are $p$-adic integers.  

\medskip

Suppose $H$ is an open subgroup of $G_p$ of level $n$. Put $H'=H \cap U$. Then there exists a basis 
$e_1, \ldots, e_d$ for coordinates $\mu$ of the second kind for $U$ with the following additional properties: \medskip

4. For some integers $0 \leq s_1 \leq s_2 \cdots \leq s_d$ we have $\mu(x_1, \ldots, x_d) \in H'$ if and only if $x_i \in p^{s_i}\mathbb Z_p$. Moreover $H' \geq U_{s_d}$, $H' \not  \geq U_{s_d-1}$ and so $n'=s_d$. \medskip

5. For each $m \in \mathbb N$ the subgroup $H'U_m$ has the parametrization
$H'U_m=\mu(p^{k_1}\mathbb Z_p , \ldots ,p^{k_d} \mathbb Z_p)$ where $k_i= \min \{s_i, m\}$. \medskip

In the language of \cite{duS} the elements $e_1^{s_1}, \ldots, e_d^{s_d}$ are a good basis for $H'$ in $U$. 

Let us indicate how to find a good basis for an open subgroup $H'$ such that $U \geq H' \geq U_{l}$. We find subsets $B_1, \ldots B_{l}$ of $U$ inductively as follows:

$B_0$ is any subset of $U$ which is a basis for the vector subspace $H'U_1/U_1$ in $U/U_1$.
Having found $B_0, \ldots, B_{i-1}$ we choose $B_{i}$ such that the set

\[ \bigcup_{r=0}^{i} \{ a^{p^{i-r}} \ | \ a \in B_r\} \] is a basis for the subspace $(H'U_{i+1} \cap U_i) /U_{i+1}$ in $U_i/U_{i+1}$.

Then $\{e_1, \ldots e_d \}= B_0 \cup \cdots \cup B_{l}$. In addition each $|B_i|$ is the number of the integers from $s_1, \ldots, s_d$ which are equal to $i$.
\medskip

We will choose an integer $m \leq n'$ with $n'/m$ bounded, and for any element $w \in G_p$ we will find an upper bound for
the proportion

 \[ \alpha_m(g,H,w)=\frac{|\{ U_{n'}x \ | \  x \in U_m w,\  xgx^{-1} \in H\}|}{[U_{m}:U_{n'}]}. \]

The reason for focusing on fixed points of $g$ on the cosets of each $U_{n'}\backslash U_mw$ separately as opposed to the
whole space $U_{n'}\backslash G_p$ is that we will be able to express $\alpha_m(g,H,w)$ as the probability of solving polynomial congruences in $\mathbb Z_p$ of \emph{bounded degree} in the coordinates $x_1, \ldots x_d$.

If we prove that there are some constants $a,b,c$ depending only on $G$, such that $\alpha_m(g,H,w) < p^{-n'c}$ for
any $w \in G_p$ and $n'> \max \{ a,bn_p \}$, it will follow that $\alpha (g, H \backslash G_p) \leq p^{-n'c}$ and we will be done.

We may assume that the numerator of $ \alpha_m(g,H,w)$ is not zero (otherwise the proportion is zero and we are done). So we may assume that $wgw^{-1} \in H$ and by replacing $g$ with $wgw^{-1}$ we are reduced to proving that for some constant $c$ (to be specified later)
\[  \beta_m(g,H)= \frac{|\{ x \in U_{n'} \backslash U_m \ | \  [g,x] \in H\}|}{[U_{m}:U_{n'}]} \leq p^{-n'c} \]
for all sufficiently large integers $n'$.

The main idea of the proof is to reduce the membership condition $[g,x] \in H$ to a power series congruence defined by a good basis for $H$ and then the choice of $m$ reduces this to a polynomial congruence of bounded degree whose solutions we then estimate with Lemma \ref{polycong} below.

If $\mathbf x=\mu(x_1, \ldots x_d)= \prod_{i=1}^d e_i^{x_i}$ then $\mu^{-1}([g,\mathbf x])= (f_1, \ldots, f_d)$ where $f_i \in \mathbb Z_p[[x_1, \ldots, x_d]]$ are converging power series in $x_1, \ldots, x_d \in \mathbb Z_p$. Indeed if $e_i^{g}=t_i$ then $[g,\mathbf x]=\mathbf x^{-g} \mathbf x=
(\prod_{i=1}^d t_i^{x_i})^{-1} \prod_{i=1}^d e_i^{x_i}$ and we are composing the power series defining multiplication and exponentiation in the coordinate system $\mu$.  

Recall the definition of $n_p=n_p(g)$ as the largest integer $l$ such that $Ad(g)_{\mathfrak P} \equiv 1$ mod $\mathfrak P^l$ for some prime ideal $\mathfrak P$ of $\mathcal O_S$ dividing $p$.

The following lemma is well known in the case $m=1$ and is a consequence of the simplicity of the Lie algebras of the $p$-adic analytic groups $G_{\mathfrak P}$ and the fact that they are defined uniformly over $k$.

\begin{lem} \label{lie} There exists a constant $A$ independent of $p$, such that for any $m \geq 1$ $\langle g^{U_m} \rangle$ contains $U_{Am+n_p}$.
\end{lem}

\begin{proof}
Since $U_m$ is equal to the direct product $\oplus_{\mathfrak P|p} U_{\mathfrak P,m}$  it is sufficient to prove the lemma with $U_{\mathfrak P,m}$ in place of $U_m$
Note that 
\[ \langle g^{U_{\mathfrak P, m}} \rangle \supseteq [g,U_{\mathfrak P, m},U_{\mathfrak P, m}, \ldots, U_{\mathfrak P, m}] \]
so it is sufficient to show that $ [g,U_{\mathfrak P, m},U_{\mathfrak P, m}, \ldots, U_{\mathfrak P, m}]$ generates a group containing $U_{\mathfrak P, Am+n_p}$ for some constant $A$.

The logarithm map establishes a bijection between the uniform group $U_{\mathfrak P}$ and its $\mathcal O_{\mathfrak P}$-Lie ring $L_{\mathfrak P}$ such that $p^{i}L_{\mathfrak P}=\mathfrak P^{r_{\mathfrak P}i}L_{\mathfrak P}$ corresponds to $U_{\mathfrak P,i}$. Moreover for any $j \in \mathbb N$ the graded Lie algebras
$\oplus_{i \geq j} U_{\mathfrak P,i}/U_{\mathfrak P,i+j}$ and $\oplus_{i \geq j} p^iL_{\mathfrak P}/p^{i+j}L_{\mathfrak P}$ are isomorphic.
  In particular $g$ acts nontrivially by conjugation on $U_{\mathfrak P,m}/U_{\mathfrak P, m+n_p+1}$. Hence $[g,U_{\mathfrak P,m}] \not \subset U_{\mathfrak P,m+n_p+1}$ and we can choose $y \in [g,U_{\mathfrak P,m}]$ with $y \not \in U_{\mathfrak P,m+n_p+1}$. Now for almost all prime ideals $\mathfrak P$ we have that $L_{\mathfrak P}$ is isomorphic to $L_0 \otimes _{\mathcal O_S} \mathfrak P \mathcal O_{\mathfrak P}$ where $L_0$ is a fixed integral lattice of the $k$-Lie algebra of $\mathbf G$. For all $\mathfrak P$ the absolute simplicity of $\mathbf G$ gives that $L_1=L_0 \otimes_{\mathcal O_S} k_{\mathfrak P}=L_{\mathfrak P} \otimes_{\mathcal O_{\mathfrak P}}k_{\mathfrak P}$ is a simple $k_{\mathfrak P}$-Lie algebra of dimension $\dim \mathbf G$. In particular $[L_1,L_1]=L_1$ and the Jacobi identity gives $[T,L_1] \subseteq [[T,L_1],L_1]$ for any subset $T$ of $L_1$. Here and below, for subsets $X,Y$ of a given Lie algebra we denote by $[X,Y]$ the vector space spanned by all Lie brackets $[x,y]$ with $x \in X$ and $y \in Y$, and for an integer $n \in \mathbb N$ we denote $[X,_n Y]:= [ \cdots [X,Y],Y], \ldots ,Y]$ ($n$ times). It follows that for any $w \in L_1 -\{0\}$ the ascending sequence of subspaces $[w,L_1] \subset [w,L_1,L_1] \subset \cdots $ stabilizes at $L_1$ in at most $\dim L_1=\dim \mathbf G$ steps and hence $[w,_{\dim \mathbf G} L_1]=L_1$.

We claim that there is a constant $A_\mathfrak P$ such that for any $w \in L_{\mathfrak P}$ outside $pL_{\mathfrak P}$, 
\begin{equation} \label{repeat} p^{A_\mathfrak P}L_{\mathfrak P} \subset [w,_{\dim \mathbf G} L_{\mathfrak P}].
\end{equation}

Indeed if this is not true there is a sequence $(w_k)_{k=1}^\infty$ with $w_k \in  L_{\mathfrak P}-pL_{\mathfrak P}$ such that $p^k L_{\mathfrak P} \not \subset [w_k,_{\dim \mathbf G} L_{\mathfrak P}]$ for any $k \in \mathbb N$. By passing to a subsequence we may assume that $w_k$ converge to some $w_0$ in the $\mathfrak P$-adic topology of $L_{\mathfrak P}$. Moreover $w_0 \not \in pL_{\mathfrak P}$ and in particular $w_0 \not =0$. Now $[w_0,_{\dim \mathbf G }L_{1}]=L_1$ and so $p^m L_{\mathfrak P} \subset [w_0,_{\dim \mathbf G }L_{\mathfrak P}]$ for some $m \in \mathbb N$. Since $w_k \rightarrow w_0$ we have $w_k \equiv w_0$ mod $p^m L_{\mathfrak P}$ for almost all $k$ but this contradicts the choice of $w_k$ when $k >m$. The claim follows.

\medskip

In fact for almost all prime ideals the reduction $L':=L_0/ \mathfrak PL_0$ of the lattice $L_0$ mod $\mathfrak P$ is a simple Lie algebra over $\mathcal O_S/\mathfrak P$ and so $L'=[w,_{\dim \mathbf G} L']$ for any nonzero $w \in L'$. Since $L_\mathfrak P \simeq L_0 \otimes _{\mathcal O_S} \mathfrak P \mathcal O_{\mathfrak P}$ for almost all $\mathfrak P$ this gives that $A_\mathfrak P$ can be taken to be $\dim \mathbf G$ for those $\mathfrak P$. By further increasing $A_\mathfrak P$ at the remaining finitely many prime ideals we conclude that (\ref{repeat}) holds with a constant $A_1$ in place of $A_{\mathfrak P}$ which does not depend on $p$.

Applying (\ref{repeat}) to the graded Lie algebra \[ \bigoplus_{i>A_1} U_{\mathfrak P,i}/U_{\mathfrak P, i+A_1+1} \simeq \bigoplus_{i \geq A_1} p^iL_{\mathfrak P}/p^{i+A_1+1}L_{\mathfrak P}\] we see that $ \langle [y,_{\dim \mathbf G} U_{\mathfrak P,m}] \rangle $ contains $U_{\mathfrak P,A_1+m \dim \mathbf G+m+n_p}$. The Lemma follows with $A=\dim \mathbf G+A_1+1$.
\end{proof}
\medskip

To illustrate the main steps of the proof we first consider the special case when the integers $s_i$ associated to the good basis describing $H'=H \cap U$ are $s_1=\cdots = s_{d-1}=0$ and $s_d=n'$.

For $n'>6A+6n_p$ choose an integer $m$ such that $\frac{n'}{3A} <m \leq \frac{n'-2n_p}{2A}$. \medskip

The membership of $x = \mu(x_1, \ldots, x_d) \in U_m$ is described by $x_i \in p^{m} \mathbb Z_p$, and so we can write $x_i=p^{m}y_i$ for some $y_i \in \mathbb Z_p$. We now claim that $f_d(p^{m}\mathbb Z_p, \ldots, p^{m} \mathbb Z_p)  \not \subset p^{Am+n_p+1}\mathbb Z_p$ and in particular not all coefficients of $f_d(p^m y_1, \ldots, p^m y_d)$ are not divisible by  $p^{Am+n_p+1}$.

Suppose for the sake of contradiction that $f_d(p^{m}\mathbb Z_p, \ldots, p^{m} \mathbb Z_p) \subset p^{Am+n_p+1}\mathbb Z_p$. Then $[g,U_m] \subset H' U_{Am+n_p+1}$. Since $g \in H$ it follows that $\langle g^{U_m} \rangle \leq HU_{Am+n_p+1}$ and so Lemma \ref{lie} gives $HU_{Am+n_p+1} \geq U_{Am+n_p}$. But since $U_{Am+n_p+1}= \Phi (U_{Am+n_p})$ we obtain $H \geq U_{Am+n_p}$. Since $m<(n'-n_p)/A$, it follows that $H \geq U_{n'-1}$ which contradicts the condition that $H$ has level $n$. The claim follows.

We now estimate the number of elements $\mathbf x=\mu (p^my_1, \ldots p^m y_d) \in U_m/U_{n'}$ such that $[g, \mathbf x ] \in H'$. This is equivalent to the congruence

\[ f_d(p^{m}y_1,  \ldots, p^{m}y_d) \equiv 0 \  \textrm{mod} \  p^{n'} \quad  y_i \in \mathbb Z_p/ p^{n'-m}\mathbb Z_p.\]

On the other hand since $m>n'/3A$ the above congruence is equivalent to a polynomial congruence of the form  \[ p^lF(y_1, \ldots, y_d) \equiv 0 \ \textrm{mod} \ p^{n'} \] where the degree of $F$ is at most $3A$ and the integer $l$ is chosen such that some coefficient of $F$ is coprime to $p$.  

We showed that $f_d(p^{m}\mathbb Z_p, \ldots, p^{m} \mathbb Z_p)  \not \subset p^{Am+n_p+1}\mathbb Z_p$ and therefore $l  \leq Am +n_p \leq n'/2$, where the last inequality following from $m<(n'-2n_p)/2A$.  So $y_1, \ldots, y_d \in \mathbb Z_p$ are solutions to the congruence \[ F(y_1, \ldots, y_d) \equiv 0 \  \mathrm{mod} \ p^{t},\] where $F$ is a polynomial in $d$ variables of degree at most $3A$ and $t$ is an integer with $n' > t \geq n'/2$. By Lemma \ref{polycong}  the number of solutions has proportion
at most $(3A)^d(t+1)^{d-1}p^{-t/3A}$ in $(\mathbb Z_p/p^{t} \mathbb Z_p)^d$. Since $A$ does not depend on $p$ and  we can find a constant $a_1$ such that if 
$n'>a_1$ we get that $ (3A)^d(n'+1)^{d-1}< 2^{n'/12A} \leq p^{n'/12A}$. This gives $\beta_m(g,H) \leq p^{-n'/12A}$ whenever
$n'> \max \{ a_1,12A,12n_p\}$.

This concludes the proof in the special case when $H$ can be described with a good basis with integers $s_1= \ldots s_{d-1}=0$ and $s_d=n'$. 
\medskip

We now prove the general case when the integers $s_i$ associated to a good basis of $H'=H \cap U$ are $0 \leq s_1 \leq \cdots \leq s_d=n'$. Set $s_0=0$ and let $\epsilon= (3A)^{-D} \leq (3A)^{-d}$.

Since $s_d=n'>0$ and $s_0=0$ there exists an integer $i \geq 1$ such that $s_i> 3As_{i-1}$. Let $1 \leq i_0 \leq n$ be the largest such integer.
Since $s_j/s_{j-1}  \leq 3A$ for all $ d \geq j > i_0$ and $s_d=n'$ we have $s_{i_0}/s_j \geq \epsilon $ for all $j \geq i_0$ and in particilar $s_{i_0} \geq n'\epsilon$.
Assuming $s_{i_0}>6n_p+6A$ (which is the case provided $n'>6(n_p+A) \epsilon^{-1}$), choose an integer $m$ such that $ s_{i_0}/3A <m \leq (s_{i_0}-2n_p)/2A$. Consider the power series $f_, \ldots, f_d$ in $\mathbb Z_p[[x_1, \ldots, x_d]]$ defined by $\mu(f_1, \ldots, f_d)=[g, \mathbf x]$, $\mathbf x=\mu(x_1, \ldots, x_d) \in U$. 

The condition $[g,\mathbf x] \in H'$, is equivalent to \[ f_j(x_1, \ldots, x_d) \equiv 0 \ \mathrm{mod}\  p^{s_j}. \  \forall j =1, \ldots, d.\] 

Now take $\mathbf x \in U_m$ i.e. $x_j \in p^{m}\mathbb Z_p$. By setting $x_k=p^{m}y_k$ , $k=1, \ldots,d$ define $p^{l_j}$ to be the largest power of $p$ dividing all the coefficients of the power series $z_j(y_1, \ldots y_d):=f_j(p^{m}y_1, \ldots ,p^{m}y_d)$.

Now we claim that $l_j < Am+n_p+1$ for some $j \geq i_0$. Suppose for the sake of contradiction that $l_j \geq Am+n_p+1$ for all $j \geq i_0$.
This means that $p^{Am+n_p+1}|f_j(x_1, \ldots, x_d)$. for all $j \geq i_0$. At the same time  since $\mathbf x \in U_m$ $[g, \mathbf x] \in U_m$ and so $p^m | f_j$ for all $j$. Since $m>s_{i_0}/3A$ and $s_{i_0-1}< s_{i_0}/3A$ it follows that $m>s_{i_0-1}$ and so $m > s_j$ for all $j  \leq i_0-1$. 
Altogether we have $p^{\min \{s_j, Am+n_p+1\}}|f_j$ for all $j=1, \ldots, d$ and so $[g,U_m] \subseteq H'U_{Am+n_p+1}$. Lemma \ref{lie} now gives  $U_{Am+n_p} \leq H U_{Am+n_p+1}$ and hence $H \geq U_{Am+n_p}$. This is a contradiction since $Am+n_p \leq s_{i_0}/2<n'$. The claim follows.

Therefore $l_j \leq Am+n_p \leq s_{i_0}/2$ for some $j \geq i_0$.

We want to estimate the number of cosets $\mathbf x U_{n'}$ in $U_m/U_{n'}$ with $[g, \mathbf x] \in H'$. Let $\mathbf x= \mu(p^{m}y_1, \ldots, p^{m}y_d) \in U_m$. We will estimate the number of solutions in
$(y_1, \ldots, y_d)  \in (\mathbb Z_p/p^{n'-m}\mathbb Z_p)^d$ to the congruence \[ f_j(p^{m}y_1, \ldots, p^{m}y_d) \equiv \ \mathrm{mod} \ p^{s_j}. \] Since $s_j \geq s_{i_0}$ we must have that $p^{s_{i_0}}|f_j(p^{m}y_1, \ldots, p^{m}y_d)$. Since $m>s_{i_0}/3A$  the last congruence implies a polynomial congruence

\[ p^{l_j}F_j(y_1, \ldots, y_d) \equiv 0 \ \mathrm{mod} \ p^{s_{i_0}} \] where $\deg F_j \leq s_{i_0}/m < 3A $ and the polynomial $F_j$ is not divisible by $p$. We proved $l_j \leq s_{i_0}/2$ and in particular \[  F_j(y_1, \ldots, y_d) \equiv 0 \ \mathrm{ mod }  \ p^{t}, \] where the integer $t$ satisfies $ n'>t \geq s_{i_0}/2$. Recall also that $s_{i_0} \geq n' \epsilon$. By Lemma \ref{polycong} below the proportion of solutions in $(\mathbb Z_p/p^{n'-m} \mathbb Z_p)^d$ to the last congruence is at most \[ (3A)^d(t+1)^{d-1}p^{-t/ \deg F_j} < (3An')^d p^{-n'\epsilon /6A}. \] So if $n'>6(A+n_p) \epsilon^{-1}$ is in addition sufficiently large in terms of $A$ and $D$ so that $(3An')^D< p^{n' \epsilon /12A}$, we see that $\beta_m (g, H)< p^{-n'\epsilon /12A}$. Proposition \ref{pp} follows with $c=\epsilon/12A= (3A)^{-D-1}/4$.

The following is proved in \cite[Lemma A.9]{FinLa}  with a slightly stronger bound, but we include a proof here for completeness.

  \begin{lem} \label{polycong} Let $f \in \mathbb Z_p[x_1, \ldots, x_d]$ be a polynomial of degree $r$, with at least one coefficient which is not divisible by $p$.
  For any $n \in \mathbb N$ the proportion of solutions to $f \equiv 0$ mod $p^n$ in $(\mathbb Z_p/p^n \mathbb Z_p)^k$ is at most $r^d (n+1)^{d-1} p^{-n/r}$.
  \end{lem}
\begin{proof} The case $d=1$ can be found in \cite{stewart}, Corollary 2, which proves a stronger bound involving the discriminant of $f$. In particular the bound $rp^{n/r}$ we require is the inequality (44) there without any condition on the discriminant.
  
  To prove the Lemma in general we argue by induction on $d$ and assume it holds for $d-1$. Write $f= g_0 x_d^m+ g_1x_d^{m-1} + \cdots +g_m$ where $g_i \in \mathbb Z_p[x_1, \ldots, x_{d-1}]$. At least one of the $g_i$, say $g_j$ is not divisible by $p$ and $\deg g_j \leq r$.
  
  For $0 \leq s \leq n$ let $X_s$ be the set of tuples $(x_1, \ldots, x_{d-1}) \in (\mathbb Z_p/p^{n}\mathbb Z_p)^{d-1}$ such that $p^s$ is the greatest power of $p$ dividing all of $g_0, \ldots, g_m$ evaluated at $(x_1, \ldots, x_{d-1})$. By the induction hypothesis applied to $g_j$ we may assume \[ |X_s| \leq r^{d-1}(s+1)^{d-2}p^{n(d-1)-s/r}.\] For a given $(d-1)$-tuple in $X_s$, the number of choices for $x_d \in \mathbb Z_p/p^n \mathbb Z_p$ such that $p^n | f$ is at most \[ mp^{n-\frac{n-s}{m}} \leq rp^{n-\frac{n-s}{r}}\] by case $d=1$ of the Lemma applied to
  $\frac{g_0}{p^s} x_d^m + \cdots + \frac{g_m}{p^s}$ and $n-s$ in place of $n$.
  
  Putting everything together the number of solutions to $p^n|f$ is at most
  
  \[ \sum_{s=0}^n |X_s|rp^{n-\frac{n-s}{r}} \leq r^d \sum_{s=0}^n (s+1)^{d-2} p^{nd-\frac{s}{r} -\frac{n-s}{r}}  \leq  r^d (n+1)^{d-1}p^{nd-\frac{n}{r}}.\]
  
  The last inequality following from the crude estimate $1+2^{d-2}+ \cdots + (n+1)^{d-2}  \leq (n+1)^{d-1}$. The Lemma follows.
  \end{proof}

\section{Spectral approximation for locally convergent sequences of lattices} \label{sec:7}

Let $G$ be a connected center free semi-simple Lie group. We let $\widehat{G}$ be the unitary dual of $G$, i.e.
the set of equivalence classes of irreducible unitary representations of $G$, endowed with the Fell topology, see e.g. \cite[\S 2.2]{BC}. We fix once and for all a Haar measure on $G$. 

Let $\phi \in C_c^{\infty} (G)$. If $\pi \in \widehat{G}$ then
$$\pi (\phi) : =\int_G \phi(g) \pi (g) dg : \mathcal{H}_{\pi} \rightarrow \mathcal{H}_{\pi}$$
is a bounded operator of trace class.  We denote by
$$\widehat{\phi} : \pi \mapsto \mathrm{trace} \ \pi (\phi)$$
the (scalar) Fourier transform on $\widehat{G}$.

\subsection{Topology of $\widehat{G}$}
As a topological space,  $\widehat{G}$ is not separated. It is somewhat easier to work
with the set $\Theta (G)$ of {\it infinitesimal characters} of $G$,
that is the set of characters of the center $Z (\g )$ of the universal enveloping algebra of $G$.

Fix $MAN$ a minimal parabolic subgroup of $G$ and
a corresponding real vector space
$$\mathfrak{h}_0 = i \mathfrak{b}_0 \oplus \mathfrak{a}_0$$
where $\mathfrak{b}_0$ is a Cartan subalgebra of the compact Lie group $K \cap M$.
The space $\mathfrak{h}_0$ can be identified with a split Cartan subalgebra of a split inner form of $G$.
In particular the complex Weyl group $W$ of $G$ acts on $\mathfrak{h}_0$. We fix a positive definite,
$W$-invariant inner product $(\cdot,\cdot)$ on $\mathfrak{h}_0$.

The infinitesimal character of an irreducible representation $\pi \in \widehat{G}$ is represented by a $W$-orbit $\theta_{\pi}$ in the complex dual space $\mathfrak{h}^*$ of $\mathfrak{h}_0$. It satisfies
$$\pi (z f) = \langle h(z) , \theta_{\pi} \rangle \pi (f) , \quad (z \in Z(\g) , \ f \in C_c^{\infty} (G)),$$
where $h : Z(\g) \rightarrow S(\mathfrak{h} )^W$ is the isomorphism of Harish-Chandra, from $Z(\g)$ onto the algebra of $W$-invariant polynomial on $\mathfrak{h}^*$.

The map
\begin{equation} \label{theta}
p: \widehat{G} \rightarrow \Theta (G)
\end{equation}
which maps $\pi \in \widehat{G}$ onto its infinitesimal character $\theta_{\pi}$ is continuous with respect to the Fell topology. See \cite[Lem. 3.4]{Sauvageot} for a more precise description of the topological space
$\widehat{G}$  with respect to this map.

The Plancherel measure $\nu^G$ is a positive Borel measure on $\widehat{G}$. Note that $\nu^G$ depends on a choice of a Haar measure on $G$: if the Haar measure is multiplied by a scalar $c$ then $\nu^G$ is multiplied by $c^{-1}$. Denote by $\mathcal{B}_c (\widehat{G})$ the space of bounded $\nu^G$-measurable functions $f$ on $\widehat{G}$ such that the support of $f$ has compact image in the space of infinitesimal character via the map $p$ defined in \eqref{theta}.

\begin{defn}
Let $\widetilde{\mathcal{F}} (\widehat{G})$ be the space of functions $f \in \mathcal{B}_c (\widehat{G})$ such that for every Levi subgroup $L$ of $G$ and every discrete series 
$\sigma$ of $L$, the function 
$$
 \chi \mapsto f( \mathrm{ind}_L^G (\sigma \otimes \chi))
$$ 
on ``unramified'' unitary characters of $L$ (see \cite[\S 3]{Sauvageot}) has the property that its discontinuous points are contained in a measure zero set. Here by definition
$ f( \mathrm{ind}_L^G (\sigma \otimes \chi))$ is the sum of $f(\sigma ')$ as $\sigma'$ runs over the irreducible subquotients of the (normalized) induced representation $\mathrm{ind}_L^G (\sigma \otimes \chi)$ with multiplicity (any such subquotient $\sigma'$ is unitary). 
\end{defn}

For any $\phi \in C_c^{\infty} (G)$, the function $\widehat{\phi}$ belongs to $\widetilde{\mathcal{F}} (\widehat{G})$. 

\begin{defn}
Let $\mathcal{F} (\widehat{G})$ be the subspace 
$$\{ \widehat{\phi} \; : \; \phi \in C_c^{\infty} (G) \} \subset \widetilde{\mathcal{F}} (\widehat{G}).$$ 
\end{defn}

\begin{rem}
There are many functions in $\widetilde{\mathcal{F}} (\widehat{G})$ that do not belong to $\mathcal{F} (\widehat{G})$: any characteristic 
function of a $\nu^G$-regular open subset $S \subset \widehat{G}$ or $S \subset \widehat{G}_{\rm temp}$ belongs to $\widetilde{\mathcal{F}} (\widehat{G})$, see \cite[Lemma 7.2]{Sauvageot}.
\end{rem}

It is much easier to work with continuous linear forms on $\mathcal{F} (\widehat{G})$ than with Borel measure on $\widehat{G}$. This is possible thanks to the following fundamental density principle due to Sauvageot \cite[Thm. 7.3(b)]{Sauvageot} (see also \cite[Appendix A]{Shin} for some corrections). 

\begin{prop} \label{DP}
Let $f \in \widetilde{\mathcal{F}} (\widehat{G})$. For every positive $\varepsilon$,
there exist $\phi , \psi \in C_c^{\infty} (G)$ such that for every $\pi \in \widehat{G}$, we have:
$$|f (\pi ) - \widehat{\phi} (\pi ) | \leq \widehat{\psi} (\pi ) \mbox{ and } \nu^G (\widehat{\psi} ) \leq \varepsilon .$$ 
\end{prop}

In other words the Plancherel measure $\nu^G$ is completely determined by the continuous linear $I^G$ form that it defines on $\mathcal{F} (\widehat{G})$. Granted this proposition we shall work with continuous linear forms on  $\mathcal{F} (\widehat{G})$.

\subsection{The measure associated to a uniform lattice}
Let $\Gamma$ be a uniform lattice in $G$. We denote by $\rho_{\Gamma}$ the quasi-regular representation of $G$ in the space $L^2 (\Gamma \backslash G)$. Then $\rho_{\Gamma}$ is a direct sum of
representations $\pi \in\widehat{G}$ occuring with finite multiplicities $\mathrm{m} (\pi , \Gamma)$. The
measure
$$\nu_{\Gamma} = \frac{1}{\mathrm{vol} (\Gamma \backslash G)} \sum_{\pi \in \widehat{G}} \mathrm{m}(\pi , \Gamma ) \delta_{\pi}$$
is, up to the factor $\mathrm{vol} (\Gamma \backslash G)^{-1}$, the Plancherel measure of $L^{2} (\Gamma \backslash G)$. This measure defines a continuous linear form $I_\Gamma$ on $\mathcal{F} (\widehat{G})$. Here again, as the spectrum of $\rho_{\Gamma}$ is discrete, the measure $\nu_\Gamma$ is determined by $I_\Gamma$  and, if $\phi \in C_c^{\infty} (G)$, we have:
\begin{equation*}
\begin{split}
\mathrm{trace} \ \rho_{\Gamma} (\phi ) & = \sum_{\pi \in \widehat{G}} \mathrm{m} (\pi , \Gamma) \mathrm{trace} \ \pi (\phi) \\
& = \mathrm{vol} (\Gamma \backslash G) I_{\Gamma} (\widehat{\phi}) .
\end{split}
\end{equation*}
On the other hand, given $f \in L^2 (\Gamma \backslash G)$, we have:\footnote{Here and below we shall often abusively identify functions on $\Gamma \backslash G$ and $\Gamma$-invariant functions on $G$. Similarly we often use the same notation ($x$ or $y$) for an element in $\Gamma \backslash G$ and for a choice of a representative of this element in $G$.}
\begin{equation*}
\begin{split}
(\rho_{\Gamma} (\phi) f)(x) & = \int_{G}\phi (y) f(xy) dy \\
& = \int_G \phi(x^{-1} y) f(y) dy \\
& = \int_{\Gamma \backslash G} \left( \sum_{\gamma \in \Gamma} \phi (x^{-1} \gamma y) \right) f(y) dy.
\end{split}
\end{equation*}
It follows that the kernel of $\rho_{\Gamma} (\phi)$ is
\begin{equation} \label{Kphi}
K_{\Gamma}^{\phi} (x,y) = \sum_{\gamma \in 
\Gamma} \phi (x^{-1} \gamma y), \quad (x,y \in \Gamma \backslash G).
\end{equation}
The sum over $\Gamma$ is finite for any $x$ and $y$, since it may be taken over the intersection of the discrete group $\Gamma$ with the compact subset $x \mathrm{supp}(\phi) y^{-1} \subset G$.
We conclude that:
\begin{equation} \label{Kphi2}
\begin{split}
\nu_{\Gamma} (\widehat{\phi}) & = \frac{1}{\mathrm{vol} (\Gamma \backslash G)} \int_{\Gamma \backslash G} K_{\Gamma}^{\phi} (x,x) dx \\
& = \int_{\Lambda \in \mathrm{Sub}_G} K_{\Lambda}^{\phi} (\mathrm{id},\mathrm{id}) d\mu_{\Gamma} (\Lambda).
\end{split}
\end{equation}
Here for any discrete subgroup $\Lambda \in\mathrm{Sub}_G$ and any $(x,y) \in G$ we denote by $K^\phi_\Lambda (x,y)$ the sum
$$K^\phi_\Lambda (x,y) = \sum_{\lambda \in \Lambda} \phi (x^{-1} \lambda y).$$
The latter equality of \eqref{Kphi2} then follows from the fact that $K^{\phi}_{g\Lambda g^{-1}} (x,y) = K_{\Lambda}^{\phi} (g^{-1} x , g^{-1} y)$. 
Note that $\Lambda \mapsto K_\Lambda^{\phi} (\mathrm{id},\mathrm{id})$ defines a continuous function on the support of $\mu_\Gamma$.

\begin{defn}
We say that a discrete IRS $\mu$ or a sequence $\mu_1 , \mu_2 , \ldots $ of discrete IRSs of $G$ is \emph{uniformly
discrete} if there exists some positive $\varepsilon$ such that:
$$\forall \Lambda \in \overline{\cup_{n=1}^{\infty} \mathrm{supp}(\mu_n)}, \quad \Lambda \cap \mathrm{B}_G (\mathrm{id}, \varepsilon ) = \{\mathrm{id}\}.$$
We shall sometimes specify $\varepsilon$, by saying that a sequence of IRS or a single IRS is $\varepsilon$-discrete.
\end{defn}

\medskip
\noindent
{\it Example.} Let $(\Gamma_n)_{n\geq 1}$ be a uniformly discrete sequence of uniform lattices in $G$. Then the sequence
$(\mu_{\Gamma_n})_{n\geq 1}$ is uniformly discrete.

\medskip

\begin{thm} \label{T1}
Let $(\Gamma_n)_{n\geq 1}$ be a uniformly discrete sequence of uniform lattices in $G$ such that $\Gamma_n \backslash X$ BS-converges to $X$.
Then for every relatively compact $\nu^G$-regular open subset $S \subset \widehat{G}$ or $S \subset \widehat{G}_{\rm temp}$, the sequence of measures $(\nu_{\Gamma_n})_{n\geq 1}$ is such that:
$$\nu_{\Gamma_n} (S) \rightarrow \nu^G (S).$$
\end{thm}
\begin{proof} Set $\nu_n = \nu_{\Gamma_n}$ and $I_n = I_{\Gamma_n}$. Let $\phi \in C_c^{\infty} (G)$ we shall first prove that:
\begin{equation} \label{limitphi}
\lim_{n \rightarrow +\infty} I_{n} (\widehat{\phi}) = I^G (\widehat{\phi}).
\end{equation}

We will make use of the following general lemma.

\begin{lem} \label{Lseparation}
Let $\mu_n$ be a uniformly discrete sequence of IRSs. Then there exists an open neighborhood of the identity $U \subset G$ and a compact subset $K\subset G$ such that, setting $U^*:= U \setminus \{ \mathrm{id} \}$, the open sets $\mathcal{O}_1 (K)$ and $\mathcal{O}_2 (U^*)$ in $\mathrm{Sub}_G$ are \emph{disjoint}, every non-discrete subgroup $H \in \mathrm{Sub}_G$ is contained in $\mathcal{O}_2 (U^*)$ and  
$$\overline{\bigcup_n \mathrm{supp} (\mu_n )} \subset \mathcal{O}_1 (K).$$ 
\end{lem}
\begin{proof} Let $d$ be a left invariant metric on $G$ and let $\delta$ be small enough so that the corresponding $\delta$-ball around $\mathrm{id}$ has no non-trivial subgroups. Since the sequence $(\Gamma_n)_{n\geq 1}$ is uniformly discrete, there exists some $\varepsilon < \delta$ such that
$$\Lambda \in \overline{\bigcup_n \mathrm{supp} (\mu_n )} \Rightarrow \overline{B}_G (\mathrm{id} , \varepsilon) \cap \Lambda = \{ \mathrm{id} \}.$$ 
Let $U$ be the open ball $B_G (\mathrm{id}, \varepsilon)$ and let $K$ be the compact set which is the closed $\epsilon$ ball minus the open $\epsilon/2$ ball around $\mathrm{id}$. Recall the following definitions:
$$\mathcal{O}_1 (K) = \{ H \in \mathrm{Sub}_G \; : \; H \cap K = \emptyset \}$$
and 
$$\mathcal{O}_2 (U^*)= \{ H \in \mathrm{Sub}_G \; : \; H \cap U^* \neq \emptyset \}.$$
These are open subsets of $\mathrm{Sub}_G$. Every non-discrete subgroup $H \in \mathrm{Sub}_G$ is obviously contained in $\mathcal{O}_2 (U^*)$ and 
$$\overline{\bigcup_n \mathrm{supp} (\mu_n )} \subset \mathcal{O}_1 (K).$$ 
Let us now prove that $\mathcal{O}_1 (K)$ and $\mathcal{O}_2 (U^*)$ are disjoint: suppose by way of contradiction that their intersection contains some subgroup $H<G$. Then the intersection of $H$ with the closed $\varepsilon /2$ ball around $\mathrm{id}$ contains a non-trivial element $h$. Since $\overline{B}_G (\mathrm{id}, \varepsilon/2)$ does not contain non-trivial subgroups, the cyclic group $\langle h \rangle$ is not entirely contained into $\overline{B}_G (\mathrm{id}, \varepsilon /2)$. Let $h^k$ be the first non-trivial power that does not belong to $\overline{B}_G (\mathrm{id}, \varepsilon /2)$. Since both $d(\mathrm{id} , h ) $ and $d(\mathrm{id} , h^{k-1} )$ are $\leq \varepsilon /2$ and since the metric $d$ is left invariant, we conclude that we have $d(\mathrm{id} , h^{k} ) \leq \varepsilon$. Therefore $h^k$ belongs to $K$, a contradiction.
\end{proof}

We shall apply Lemma \ref{Lseparation} to the sequence $\mu_n = \mu_{\Gamma_n}$. Let $V$ be an open symmetric neighborhood of the identity in $G$ such that $V^2 \subset U$. The $G$-translates of $V$ form an open covering of $G$ from which we may extract a finite collection $g_1V , \ldots , g_k V$ that covers the {\it compact} support of $\phi$. Every $\Lambda \notin \mathcal{O}_2 (U^*)$ intersects each $g_i V$ along at most one element. It follows that the function 
$$\Lambda \mapsto \sum_{\lambda \in \Lambda} \phi (\lambda )$$
is well defined, continuous and uniformly bounded (by $k ||\phi ||_\infty$) on $\mathrm{Sub}_G \setminus \mathcal{O}_2 (U^*)$. Tietze's Extension Theorem then allows to extend this function to a compactly supported continuous function $F_{\phi}$ on $\mathrm{Sub}_G$ such that
$$F_{\phi} (\Lambda) = \left\{
\begin{array}{ll}
\sum_{\lambda \in \Lambda} \phi (\lambda) & \mbox{ if } \Lambda \in \overline{\bigcup_n \mathrm{supp} (\mu_n )} \\
0 & \mbox { if } \Lambda \mbox{ is not discrete.}
\end{array} \right.$$

Since by hypothesis the sequence $\mu_n$ converges weakly toward $\mu_{\mathrm{id}}$ we get that:
$$I_{n} (\widehat{\phi}) = \int_{\mathrm{Sub}_G} F_{\phi} d\mu_n \rightarrow \int_{\mathrm{Sub}_G} F_{\phi} d\mu_{\mathrm{id}}.$$
The limit is equal to $\phi (\mathrm{id})$ which, according to the Plancherel formula proved by Harish-Chandra, is equal to $I^G (\widehat{\phi})$.
This proves \eqref{limitphi}.

To conclude the proof of Theorem \ref{T1} we recall that the linear form $I_n$ determines the Borel measure $\nu_n$ on $\widehat{G}$ and that it similarly follows from Proposition \ref{DP} (Sauvageot's density principle) that the linear form $I^G$ determines the Plancherel measure of $G$. The Theorem easily follows. Indeed: let $S \subset \widehat{G}$, or $S \subset \widehat{G}_{\rm temp}$, be a relatively compact open subset which is regular with respect to the Plancherel measure of $G$ (i.e. $\nu^G (S) = \nu^G (\overline{S})$).
Let $\varepsilon$ be a positive real number. By the density principle, there exist $\phi , \psi \in C_c^{\infty} (G)$ such that
$$|\mathrm{1}_S - \widehat{\phi} | \leq \widehat{\psi} \mbox{ and } \nu^G (\widehat{\psi} ) \leq \varepsilon .$$
We conclude that:
\begin{equation*}
\begin{split}
|\nu_{\Gamma_n} (S) - \nu^G (S) | & \leq \nu_{\Gamma_n} (\widehat{\psi}) + | \nu_{\Gamma_n} (\widehat{\phi}) - \nu^G (\widehat{\phi})| + \nu^G (\widehat{\psi}) \\
& \leq |I_n (\widehat{\psi}) - I^G  (\widehat{\psi}) | + 2  I^G (\widehat{\psi}) +  | I_n (\widehat{\phi}) - I^G  (\widehat{\phi})| \\
& \leq 4 \varepsilon,
\end{split}
\end{equation*}
for sufficiently large $n$.
\end{proof}

Theorem \ref{T1} implies the following:
\begin{cor}[Pointwise convergence] \label{T3}
Let $(\Gamma_n)_{n\geq 1}$ be a uniformly discrete sequence of uniform lattices in $G$ such that $\Gamma_n \backslash X$ BS-converges to $X$. Then:
$$\lim_{n \rightarrow +\infty} \nu_{\Gamma_n} (\{ \pi \}) = \nu^G (\{ \pi \} )$$
for every $\pi \in \widehat{G}$.
\end{cor}
Note that $d(\pi):=\nu^G (\{ \pi \})$ is $0$ unless $\pi$ is square integrable (i.e. is a discrete series) in which case it is the formal degree of $\pi$,  see \cite[Theorem 6.2]{DeGeorgeWallach}.

\subsection{An alternative proof of Corollary \ref{T3}} \label{alternative} Here we propose a proof of Corollary \ref{T3} in the spirit of DeGeorge--Wallach \cite{DeGeorgeWallach} and Savin \cite{Savin} that avoids the intricate analysis of \cite{Sauvageot}. We first prove that:
\begin{equation} \label{limsup}
\begin{split}
\limsup_{n \rightarrow \infty} \frac{\mathrm{m} (\pi , \Gamma_n )}{\mathrm{vol} (\Gamma_n \backslash G)} & = \limsup_{n \rightarrow \infty} \nu_n (\{ \pi \} ) \\
& \leq \nu^G (\{ \pi \}) = d(\pi).
\end{split}
\end{equation}

Let $\phi \in C_c^{\infty} (G)$. We first note that:
\begin{equation} \label{alter}
\begin{split}
\frac{\mathrm{m} (\pi , \Gamma)}{\mathrm{vol} (\Gamma \backslash G)} ||\pi (\phi ) ||_{\rm H-S}^2
& \leq \frac{|| \rho_{\Gamma} (\phi ) ||_{\rm H-S}^2}{\mathrm{vol} (\Gamma \backslash G)} \\
& \leq \frac{\mathrm{trace} \ \rho_{\Gamma} (\phi * \widetilde{\phi})}{\mathrm{vol} (\Gamma \backslash G)} \\
& \leq \nu_{\Gamma} (\widehat{\phi * \widetilde{\phi}}).
\end{split}
\end{equation}

\medskip
\noindent
{\it Remark.} We have:
$$\nu^G (\widehat{\phi * \widetilde{\phi}}) = (\phi * \widetilde{\phi}) (1) = ||\phi ||^2.$$

\medskip

Note that:
$$||\pi (\phi ) ||_{\rm H-S}^2 \geq |\langle \pi (\phi) v , v \rangle |^2 = \left| \int_G \phi (g) \langle \pi (g) v, v \rangle dg \right|^2$$
where $v$ is any unit vector in the Hilbert space associated with $\pi$. It is therefore tempting to apply \eqref{alter} with
$\phi (g) = (\phi_r (g) :=) \chi_r(g)  \langle \pi (g) v , v \rangle$ where $\chi_r$ is the characteristic
function of $G_r = KA_r^+ K$, $A_r^+ = \{ a \in A^+ \; : \; a = \exp (H), \ ||H|| \leq r\}$ for some
metric $||\cdot ||$ on the Lie algebra of the Cartan subgroup $A$.  The function $\phi_r$ is not smooth. However it is a limit in $L^2$ of smooth functions with support in $G_r$ and \eqref{alter} still holds. Similarly, under the hypotheses of Corollary \ref{T3}, Equation \eqref{limitphi} applies to $\phi_r * \widetilde{\phi}_r$. We therefore conclude from the remark above that we have:
\begin{equation} \label{limsup2}
\limsup_{n \rightarrow +\infty} \frac{\mathrm{m} (\pi , \Gamma_n)}{\mathrm{vol} (\Gamma_n \backslash G)} \leq \frac{1}{||\phi_r||^2}.
\end{equation}
As $r$ tends to infinity $1/||\phi_r||^2$ tends to $0$ if $\pi$ is not square integrable and tends to
$d(\pi)$ if $\pi$ is a discrete series. Inequality \eqref{limsup2} therefore implies \eqref{limsup}.

\subsection{} Now fix $\pi$ a discrete series representation of $G$. The set
$$\widehat{G} (\pi) = \{ \omega \in \widehat{G} \; : \; \theta_{\omega} = \theta_{\pi} \}$$
is finite. Computing the $G(\pi)$-part of the Euler characteristic, DeGeorge and Wallach \cite[Corollary 5.3]{DeGeorgeWallach} proved the following:
\begin{prop} \label{PRS}
Given a discrete series representation $\pi$ of $G$, there are constants $c(\omega)$, $\omega \in
\widehat{G} (\pi) $ with $c(\omega)=1$ whenever $\omega$ is a discrete series representation, such that
$$\sum_{\omega \in \widehat{G} (\pi)} c(\omega) \frac{\mathrm{m} (\omega , \Gamma)}{\mathrm{vol} (\Gamma \backslash G)} = \sum_{\omega \in \widehat{G} (\pi )} d(\omega ).$$
\end{prop}

Note that when $\omega$ is not a discrete series then the limit multiplicity is $0$ by \eqref{limsup}. Proposition \ref{PRS} and \eqref{limsup} theorefore imply Corollary \ref{T3}. \qed

\subsection{Sequences of congruence lattices}
Now we fix a uniform irreducible arithmetic lattice $\Gamma \subset G$ as $\Gamma_0$ in Theorem \ref{nik}. We also fix $\pi \in \widehat{G}$ a {\it non tempered} representation, i.e. $\pi$ is not weakly contained in $L^2 (G)$. In  this setting we prove the following:

\begin{thm} \label{Tcong}
Let $(\Gamma_n)_{n\geq 1}$ be {\rm any} infinite sequence of distinct congruence subgroups of $\Gamma$. Then there exists $\alpha = \alpha ( G, \Gamma , \pi) >0$ such that
$$\mathrm{m} (\pi , \Gamma_n ) \ll \mathrm{vol} (\Gamma_n \backslash G)^{1- \alpha}.$$
\end{thm}
\begin{proof}  This follows the same lines as \S \ref{alternative}: Let $(\tau , V_{\tau})$ be the lowest $K$-type of $\pi$, as defined by Vogan \cite{Vogan}, and let $v \in V_{\tau}$ be a highest weight vector. As in \cite{Savin} we introduce:
$$W_n = \mathrm{span} \left\{ T v \; : \; T \in \mathrm{Hom}_G (V_\tau  , L^2 (\Gamma_n \backslash G)  )\right\} \subset L^2 (\Gamma_n \backslash G)$$
and
$$B_n (x) = \sup_{f \in W_n} \frac{|f(x)|^2}{||f||^2} \quad (x \in \Gamma_n \backslash G).$$
As in \S \ref{alternative} we let
$$\phi_r (g) = \chi_r (g) \langle \pi (g) v , v \rangle \quad ( g \in G , \ r >0 ).$$

We will use the following two lemmas. The first goes back at least to Kazhdan's proof \cite{Kazhdan} of the so-called Kazhdan's inequality according to which along a residual tower the limsup of the normalized Betti numbers are bounded above by the corresponding $L^2$-Betti numbers; we include a proof of this first lemma for the reader's convenience. The second --- due to Savin \cite[Proposition 3]{Savin} --- is a reformulation of the basic identity of DeGeorge and Wallach.

\begin{lem} \label{L1}
We have:
$$\int_{\Gamma_n \backslash G}  B_n (x) dx = \mathrm{m} (\pi , \Gamma_n ).$$
\end{lem}
\begin{proof} Let $f_1 , \ldots , f_m$ ($m=\mathrm{m} (\pi , \Gamma_n )$) be an orthonormal basis of $W_n$. The Cauchy-Schwarz inequality implies that 
$B_n (x) \leq \sum_{i=1}^m |f_i (x) |^2$. Now if we fix $x$, the function $F:y \mapsto \sum_i \overline{f_i (x)} f_i (y)$ belongs in $W_n$ and we have:
$$||F||^2 = F(x) = \sum_i |f_i (x) |^2.$$
It follows that for all $x$ we have:
$$B_n (x) = \sum_{i=1}^m |f_i (x) |^2.$$
Integration over $\Gamma_n \backslash G$ gives the lemma.
\end{proof}

\begin{lem} \label{L2}
We have:
$$\pi (\phi_r ) v = ||\phi_r ||^2 v.$$
\end{lem}

Now let $f \in W_n$. It follows from Lemma \ref{L2} that:
$$||\phi_r ||^2 f(x) = \int_G \phi_r (g) f(xg) dg = \int_{\Gamma_n \backslash G} \sum_{\gamma \in \Gamma_n} \phi_r (x^{-1} \gamma g) f(g) dg .$$
By the Cauchy-Schwarz inequality, we have:
\begin{equation} \label{savin}
||\phi_r ||^2 | f(x) | \leq ||f|| \left( \int_{\Gamma_n \backslash G} \left| \sum_{\gamma \in \Gamma_n} \phi_r (x^{-1} \gamma g) \right|^2 dg \right)^{1/2}.
\end{equation}

Given $x \in G$, we set
$$N_n (x; r) = \# \{ \gamma \in \Gamma_n \; : \; \chi_r (x^{-1} \gamma x ) \neq 0 \}.$$
Theorem \ref{nik} implies the following:

\begin{prop} \label{nik stuff}
There exist positive constants $\beta$, $c$ such that for all $n$
$$\mathrm{vol} ( (\Gamma_n \backslash G)_{< c \log \mathrm{vol} (\Gamma_n \backslash G)} ) \leq \mathrm{vol} (\Gamma_n \backslash G)^{1-\beta}.$$
\end{prop}

We now recall the following:

\begin{lem} \label{gauss}
There exist constants $c_1,c_2 > 0$, depending only on $G$, such that for any $x,y \in X$,
$$N(x;R):=\left|\{ \gamma \in \Gamma \; : \; d(x,\gamma x) \leq R \} \right| \leq c_1 \mathrm{InjRad}_{\Gamma_n \backslash G} (x)^{-d} e^{c_2 R},$$
where $d$ is the dimension of $X$.
\end{lem}
\begin{proof}
Clearly, it suffices to prove this for $R \geq \mathrm{InjRad}_{M_n} (x)$. By definition, $$B(x , \mathrm{InjRad}_{\Gamma_n \backslash G} (x)) \cap B(\gamma x , \mathrm{InjRad}_{M_n} (x)) = \emptyset$$
for all $\gamma \in \Gamma - \{ \mathrm{id} \}$. This implies
\begin{align*}
N(x;R) \cdot \vol \, B(x, \mathrm{InjRad}_{\Gamma_n \backslash G} (x)) & \leq \vol \, B (x , R + \mathrm{InjRad}_{\Gamma_n \backslash G} (x)) \\ & \leq \vol \, B (x , 2 R).
\end{align*}
Now, Knieper \cite{Knieper} shows that there exists a constant $c_2 = a(G)$ such that
$$\vol \, B(x , R ) \approx R^{\frac{\mathrm{rank}_\BR(G) -1}{2}} e^{c_2R}$$
asymptotically as $R \to \infty$.
This yields an upper bound for $\vol \, B(x, 2R)$.

On the other hand, since $X$ has non-positive curvature, the volume of a ball in $X$ is bounded below by the volume of a ball with the same radius in $d$-dimensional Euclidean space. Hence
$$\vol \, B(x,\mathrm{InjRad}_{\Gamma_n \backslash G} (x)) \geq b \cdot \mathrm{InjRad}_{\Gamma_n \backslash G} (x)^d,$$
with a constant $b=b(d)$.
The lemma follows.
\end{proof}

\begin{rem}
When $\mathrm{InjRad}_{\Gamma_n \backslash G} (x)$ and $R$ are both sufficiently small, it is possible to attain better bounds in \ref{gauss} by using the Margulis Lemma, see the analysis in Section \ref{sec:thin}.
\end{rem}

\subsection{} Replacing the constant $c$ by some smaller positive constant we may assume:
\begin{equation} \label{inegconst}
cc_2 \leq \beta.
\end{equation}
Here $c_2$ is the constant of Lemma \ref{gauss}. From this we conclude:

\begin{lem} \label{LN}
There exists a positive constant $C$ such that for all $n$
$$\int_{\Gamma_n \backslash G} N_n (x ; 2c \log \mathrm{vol} (\Gamma_n \backslash G )) dx \leq C \mathrm{vol} (\Gamma_n \backslash G ).$$
\end{lem}
\begin{proof} We split the integral into two parts:
$$I_1 = \int_{\{x  \in \Gamma_n \backslash G \; : \; \mathrm{InjRad}_{\Gamma_n \backslash G} (x)
\leq c \log \mathrm{vol} (\Gamma_n \backslash G) \}}  N_n (x ; 2c \log \mathrm{vol} (\Gamma_n \backslash G )) dx$$
and $I_2$.
Since in $I_2$ the integrand is everywhere equal to $1$ we have $I_2 \leq \mathrm{vol} (\Gamma_n \backslash G )$. As for $I_1$ we use Lemma \ref{gauss} to get the bound:
$$N_n (x ; 2c \log \mathrm{vol} (\Gamma_n \backslash G )) \leq c_1 \mathrm{InjRad}_{\Gamma_n \backslash G} (x)^{-d} \mathrm{vol} (\Gamma_n \backslash G )^{c_2 c}.$$
Since each lattice $\Gamma_n$ is a subgroup of $\Gamma$, there exists a uniform (in $n$)
lower bound on $\mathrm{InjRad}_{\Gamma_n \backslash G} (x)$. We therefore conclude from Proposition \ref{nik stuff} and \eqref{inegconst} that:
$$I_1 \leq (\mathrm{const}) \mathrm{vol} (\Gamma_n \backslash G )^{c_2 c + 1- \beta } \leq (\mathrm{const}) \mathrm{vol} (\Gamma_n \backslash G ).$$
And the lemma follows.
\end{proof}

\subsection{} Now taking $r = c \log \mathrm{vol} ( \Gamma_n \backslash G)$ we note that for every $x\in \Gamma_n \backslash G$ and $g \in G$ the sum $\sum_{\gamma \in \Gamma_n} \phi_r (x^{-1} \gamma g)$ has at most $N_n (x ; 2r)$ nonzero term. Therefore
$$\left| \sum_{\gamma \in \Gamma_n} \phi_r (x^{-1} \gamma g) \right|^2 \leq  N_n (x ;2r) \sum_{\gamma \in \Gamma_n} \left| \phi_r (x^{-1} \gamma g) \right|^2.$$
Moreover, since
$$\int_{\Gamma_n \backslash G} \sum_{\gamma \in \Gamma_n} \left| \phi_r (x^{-1} \gamma g) \right|^2 dg = || \phi_r ||^2 ,$$
it follows from \eqref{savin} that for every $x\in \Gamma_n \backslash G$
\begin{equation} \label{truc}
\frac{|f(x)|^2}{||f||^2} \leq \frac{N_n (x, 2r)}{||\phi_r||^2}.
\end{equation}
Integrating \eqref{truc} over $\Gamma_n \backslash G$ we conclude from Lemma \ref{LN} that:
$$\mathrm{m} (\pi , \Gamma_n ) \leq \frac{1}{||\phi_r||^2} \int_{\Gamma_n \backslash G} N_n (x ; c \log \mathrm{vol} (\Gamma_n \backslash G )) dx \leq C \frac{\mathrm{vol} (\Gamma_n \backslash G )}{||\phi_r||^2}.$$
We finally note that
$$||\phi_r ||^2 = \int_{G_r} |\langle \pi (g) v ,v \rangle |^2 dg$$
and
$$\mathrm{vol} (G_r) \geq \exp (\nu r) = \mathrm{vol} (\Gamma_n \backslash G)^{c\nu}$$
for some positive constant $\nu$.
Combining this last inequality with the asymptotics of the matrix coefficient $g \mapsto \langle \pi (g) v ,v \rangle$, see e.g. \cite{Knapp} or \cite[Corollary 3.18 and Lemma 4.4]{VandenBan}, we conclude from the fact that $\pi$ is non-tempered that there exists some positive constant $\alpha$ such that
$$\frac{1}{||\phi_r||^2} \ll \mathrm{vol} (\Gamma_n \backslash G)^{-\alpha}.$$
The theorem follows.
\end{proof}

\subsection{Nonuniform lattices} In the nonuniform case things get more complicated: there is continuous spectrum in $L^2 (\Gamma \backslash G)$ and the integral
$$\int_{\Gamma \backslash G} K_{\Gamma}^{\phi} (x,x) dx$$
is divergent. We may nevertheless hope that, maybe under suitable conditions, Theorem \ref{T1} holds when replacing $\nu_{\Gamma}$ by the measure associated to the {\it discrete}, or to the {\it cuspidal}, spectrum of $L^2 (\Gamma \backslash G)$. There is not yet such complete results even in the case of towers of coverings. We may however refer to the already mentioned work of Shin \cite{Shin} and to the recent work of Finis, Lapid and Mueller \cite{FLM,FL2015}
which, in particular, completely solves the problem for the case of principal congruence subgroups of $\GL(n)$.

\subsection{From representations to differential forms} We conclude this section by relating the above results with the study of the spectrum of the Laplace
operator.

Given a unitary representation $\tau$ of $K$ we consider the following subset of $\widehat{G}$:
$$\widehat{G}_{\tau} = \{ \pi \in\widehat{G} \; : \;  \mathrm{Hom}_K (\tau , \mathcal{H}_{\pi}) \neq \{ 0 \} \}.$$
Let $\tau_k$ ($k=0 , 1 , \ldots $) be the adjoint representation of $K$ into $\wedge^k \p$. Representations in $\widehat{G}_{\tau_k}$ are exactly the ones which correspond to $k$-differential forms on $X=G/K$. Our choice of Haar measure on $G$ corresponds to a choice of a left invariant Riemannian metric on $G$. We denote by $\mathrm{vol} (K)$ the corresponding volume of $K$. 

Let $\Gamma$ be a lattice in $G$. First note that we have: 
\begin{equation} \label{vol}
\mathrm{vol} (\Gamma \backslash G) = \mathrm{vol} (\Gamma \backslash X ) \mathrm{vol} (K).
\end{equation}

Now let $C\in Z (\g)$ be the Casimir element. Set $\lambda_{\pi} =-\theta_{\pi} (C)$. Let $\pi \in \widehat{G}_{\tau_k}$ and $v \in \mathcal{H}_{\pi}$ be a nonzero vector in the $K$-type $\tau_k$.
Any element in
$$E_{\pi}^k (\Gamma \backslash G) := \mathrm{span} \left\{ T v \; : \; T \in \mathrm{Hom}_G ( \mathcal{H}_{\pi},L^2 (\Gamma \backslash G)  \right\} \subset L^2 (\Gamma   \backslash G)$$
defines a square integrable $k$-differential form on $\Gamma  \backslash G/K$ whose eigenvalue is $\lambda_{\pi}$. Conversely it follows from Matsushima's formula (see e.g. \cite[Thm. 1.0.2]{BC}) that
$$E_{\lambda}^k (\Gamma \backslash G) = \bigoplus_{\substack{ \pi \in  \widehat{G}_{\tau_k} \\ \lambda_{\pi} = \lambda }} E_{\pi}^k (\Gamma \backslash G)$$
where $E_{\lambda}^k (\Gamma \backslash G)$ denotes the $\lambda$-eigenspace of the laplace operator on square integrable $k$-differential form on $\Gamma \backslash X$.

We let $\Theta_k (G)$ be the image of $\widehat{G}_{\tau_k}$ by the map $p$ in \eqref{theta}.
Evaluation on the Casimir element therefore gives a map
\begin{equation} \label{map}
\Theta_k (G) \rightarrow \R^+.
\end{equation}
A Borel measure $\nu$ on $\widehat{G}$ induces a measure $p^* \nu$ on $\Theta (G)$ that we may restrict to a measure on $\Theta_k (G)$; we denote by $\nu^k$ the push-forward of the latter by the map \eqref{map} so that $\nu^k$ is a measure on $\R^+$.

\subsection{} \label{8.2} Suppose that $\Gamma$ is uniform. We have:
$$\nu^k_{\Gamma} (\{ \lambda \} )  = \frac{1}{\mathrm{vol} (\Gamma \backslash G)} \dim E_{\lambda}^k (\Gamma \backslash G).$$
In particular
$$\nu^k_{\Gamma} (\{ 0 \} ) = \frac{b_k (\Gamma )}{\mathrm{vol} (\Gamma \backslash G)}$$
where $b_k  (\Gamma )$ is the $k$-th Betti number of $\Gamma$. Note that $\Gamma$ being virtually torsion-free, $b_k (\Gamma )$ makes sense. If $\Gamma$ is torsion-free we have $b_k (\Gamma) = b_k (\Gamma  \backslash X)$.
Similarly $\nu^{G,k}$ is the spectral measure of the Laplace operator on square integrable differential $K$-forms on $X$ and we define the $k$-th $L^2$-Betti number of the symmetric space $X=G/K$ as
$$\beta_k^{(2)} (X) = \nu^{G,k} (\{ 0 \})  \mathrm{vol} (K).$$
Note that it follows from \eqref{vol} that $\mathrm{vol} (\Gamma \backslash X) \beta_k^{(2)} (X)$ is the usual $k$-th $L^2$-Betti number
of $\Gamma$.

Theorem~\ref{T1} implies the following two corollaries:

\begin{cor} \label{C84}
Let $(\Gamma_n)_{n\geq 1}$ be a uniformly discrete sequence of uniform lattices in $G$ such that $\Gamma_n \backslash X$ BS-converges to $X$.
Then for each $k$ the sequence of spectral measures $\nu^k_{\Gamma_n}$  converges weakly toward $\nu^{G , k}$.
\end{cor}

\begin{cor} \label{main}
Let $(\Gamma_n)_{n\geq 1}$ be a uniformly discrete sequence of uniform lattices in $G$ such that $\Gamma_n \backslash X$ BS-converges to $X$. Then:
$$\lim_{n\rightarrow \infty }\frac{b_{k}(\Gamma_{n})}{\mathrm{vol}(\Gamma_n \backslash X)}=\beta
_{k}^{(2)}(X)
$$
for $0\leq k\leq \dim (X)$.
\end{cor}

\subsection{} We finally recall from \cite[p. 98]{BorelWallach} that if $\pi \in \widehat{G}$ is such that $\pi \in \widehat{G}_{\tau_k}$ and $\lambda_{\pi} = 0$, we have:
$$\pi \mbox{ is tempered } \Leftrightarrow k \in \left[ \frac12 \dim X -e , \frac12 \dim X +e \right]$$
where $e = \frac12 (\mathrm{rank}_{\mathbb{C}} G - \mathrm{rank}_{\mathbb{C}} (K))$. Theorem
\ref{Tcong} therefore implies the following:

\begin{cor}
Let $(\Gamma_n )_{n\geq 1}$ be a sequence of congruence lattices in a fixed rational form $\mathbf{G} (\Q)$. Suppose that
$\mathrm{vol} (\Gamma_n \backslash X ) \rightarrow \infty$. Then there exists $\alpha = \alpha ( \mathbf{G}) >0$ such that for every $k \notin \left[ \frac12 \dim X -e , \frac12 \dim X +e \right]$,
$$b_k (\Gamma_n )  \ll \mathrm{vol} (\Gamma_n \backslash X )^{1- \alpha}.$$
\end{cor}

\section{Heat kernel estimates and hyperbolic manifolds} \label{sec:thin}

As explained in the announcement \cite{cras}, our original proof of Corollary \ref{main} used the heat kernel following the original path of DeGeorge--Wallach and especially Donnelly \cite{Donnellytowers}. Introducing the notion of BS convergence allowed us to deal with more general sequences that our predecessors did. However, as in these classical works, this approach relies on heat kernel estimates which require a lower bound on the injectivity radius (our `uniformly discrete' assumption). One novel aspect of the current section is a fine study of heat kernel estimates in the thin parts of hyperbolic manifolds in dimension $d\ge 4$. This will allow us to get rid of the `uniform discreteness' assumption.

As in the preceding sections we let $X = G/K $ be the symmetric space associated to a connected center free semi-simple Lie group $G$.  

\subsection{The heat kernel on forms} 
We denote by $e^{-t \Delta_k^{(2)}} (x,y)$ the heat kernel on square integrable $k$-forms on $X$. The corresponding bounded integral operator in $\mathrm{End}(\Omega^k_{(2)}(X))$ defined by
$$
 (e^{-t\Delta_k^{(2)}} f) (x) = \int_{X} e^{-t\Delta_k^{(2)}} (x,y) f(y) \, dy, \ \ \forall f \in  \Omega^k_{(2)}(X)
$$
is the fundamental solution of the heat equation (cf. \cite{BarbaschMoscovici}).

A standard result from local index theory (see e.g. \cite[Lemma 3.8]{BV}) implies:

\begin{lem}\label{L:heat}
Let $m >0$. There exists a positive constant $c=c(G,m)$ such that
$$|| e^{-t \Delta_k^{(2)}} (x,y) || \leq c t^{-d/2} e^{- d (x, y)^2 / 5t } ,  \quad  0 < t \leq m.$$
\end{lem}

 Much of the content of the statement above is when $t\to 0$.  Here, we are mostly interested in the case of fixed $t$, in which case Lemma \ref{L:heat} gives constants $c_1,c_2$ depending only on $G,t$ such that
\begin{equation} || e^{-t \Delta_k^{(2)}} (x,y) || \leq c_1 e^{- d (x, y)^2 / c_2} .
\label{specialized}
\end{equation}

Now let $M=\Gamma\backslash X$ be a compact  $X$-manifold. Let $\Delta_k$ be the Laplacian on differentiable $k$-forms on
$M$. It is a symmetric, positive definite, elliptic operator with pure point spectrum.
Write $e^{-t\Delta_k} (x,y)$ ($x,y \in M$) for the heat kernel on $k$-forms on $M$, then for each positive $t$ we have:
\begin{equation} \label{eq:heatkernels}
e^{-t\Delta_k} (x,y) = \sum_{\gamma \in \Gamma} (\gamma_{\tilde y})^* e^{-t\Delta_k^{(2)}} (\widetilde{x} , \gamma \widetilde{y}),
\end{equation}
where $\widetilde{x}, \widetilde{y}$ are lifts of $x,y$ to $X$ and by $(\gamma_y)^*$,
we mean pullback by the map $(x,y) \mapsto (\tilde x, \gamma \tilde y)$.
The sum converges absolutely and uniformly for $\tilde{x}, \tilde{y}$ in compacta; this follows from Lemma \ref{gauss} and \ref{L:heat}. Given $x\in M$ and a lift $\ti x\in X$ we set:
\begin{equation} \label{eq:f_t}
 f_t(x)=\| e^{-t\gD_k}(x,x)-e^{-t\gD^{(2)}_k}(\ti x,\ti x)\|=\|\sum_{\gc\in\gC\setminus\{1\}}e^{-t\gD^{(2)}_k}(\ti x,\gc\cdot\ti x)\|.
\end{equation}
Here, the middle part of the equation can be made well defined by identifying the tangent spaces of $T_xM$ and $T_{\ti x}X$.
Let $\ti f_t (\ti x)=f_t(x)$ and note that $\ti f_t$ is $\gC$-invariant.
Recall that we denote by $\mathrm{InjRad}_{M} (x)$ the injectivity radius of $M$ at $x$.

\subsection{($L^2$-)Betti numbers} 
\label {bettinumbers}
The trace of the heat kernel $e^{-t \Delta_k^{(2)}}(x,x)$ on the diagonal
is independent of $x \in X$, being $G$-invariant. We denote it by
$$\mathrm{Tr} \,e^{-t \Delta_k^{(2)}}:=\mathrm {tr} \, e^{-t \Delta_k^{(2)}}(x,x).$$
It follows from \S \ref{8.2} that
$$\beta_k^{(2)} (X) =  \lim_{t \rightarrow \infty} \mathrm{Tr} e^{-t \Delta_k^{(2)}} .$$
It is equal to zero unless $\delta (G)=0$ and $k = \frac12 \dim X$, in  which case
$$
 \beta_k^{(2)} (X) = \frac{\chi (X^d)}{\mathrm{vol} (X^d)},
 $$ 
where $X^d$ is the compact dual, see \cite{Olbrich}.

Recall also that the usual Betti numbers of $M$ are given by
$$b_k (M) = \lim_{t \rightarrow \infty} \mathrm{Tr} e^{-t\Delta_k} = \lim_{t \rightarrow \infty} \int_M \mathrm{tr} \ e^{-t \Delta_k}(x,x)dx,$$
and that since $\mathrm{Tr} \, e^{-t\Delta_k} = \sum_i e^{-t\lambda_i}$, where $\lambda_i$ are the eigenvalues of $\Delta_k$, the limit above is \emph {monotone decreasing} in $t$.

\begin{lem} \label{lem:hk}
Let $t >0$ be a real number. There exists a constant $C = C(t,G)$ such that for any $x \in M$, $$
f_t (x)  \leq C \cdot \mathrm{InjRad}_{M} (x)^{-d}.
$$
\end{lem}
\begin{proof}
Let $x \in M$ and $\tilde x$ be a lift of $x$ to $X$. 
Then by the definition \eqref{eq:f_t},
\begin {align*}
f_t(x)  & \leq \sum_{\gc\in\gC\setminus\{1\}} \| e^{-t\gD^{(2)}_k}(\ti x,\gc\cdot\ti x)\|. \\
& \leq \sum_{\gc\in\gC\setminus\{1\}} c e^{- d (x, \gamma \tilde x)^2 / 5t },  \ \text{ by } \eqref{specialized} \\
& \leq \sum_{n=0}^\infty c e^{- n^2 / 5t } N(x,n+1),
\end {align*}
where in an overestimate, $N(x;n+1)=\# \{ {\gc\in\gC\setminus\{1\}} :  d(\tilde x , \gamma \tilde x ) \leq n +1\}.$ But
$$N(x;n+1) \leq c_1 \mathrm{InjRad}_{M} (x)^{-d} e^{c_2 (n+1)}$$
by Lemma \ref{gauss}, for some $c_1,c_2$ depending only on $G $. So for some $C=C(t,G)$,
$$f_t(x)\leq  \sum_{n=0}^\infty cc_1 e^{- n^2 / 5t +c_2 (n+1)} \mathrm{InjRad}_{M} (x)^{-d}  \leq C \cdot \mathrm{InjRad}_{M} (x)^{-d}.\qedhere$$
\end{proof}

\subsection{Convergence of Betti numbers} We now explain how to use the heat kernel estimates above to prove the following proposition, which implies Corollary \ref{main}.
\begin{prop} \label{Pcvhk}
Suppose that $(M_n)$ is a sequence of compact $X $-manifolds that BS-converges to $X$. Then we have $(1) \implies (2)\implies (3) \implies (4)$, where

\begin{minipage}{.8\textwidth}
\begin {align}
\tag{1} & (M_n) \text{ is uniformly discrete,}\\
\tag{2}  & \frac{1}{\mathrm{vol} (M_n)} \int_{M_n} \mathrm{InjRad}_{M_n} (x)^{- d} dx \rightarrow 0, \\
\tag{3} &\lim_{n \to +\infty} \frac{b_k(M_n )}{\vol(M_n )}\le\gb^{(2)}_k(X), \text{ for } k=0,\ldots,\dim(X)\\
\tag{4}  &\lim_{n \to +\infty} \frac{b_k(M_n )}{\vol(M_n )}=\gb^{(2)}_k(X), \text{ for } k=0,\ldots,\dim(X).
\end {align}
\end{minipage}
\end{prop}
\begin {proof}  $(1) \implies (2)$. Since $(M_n)$  is uniformly discrete, there is some $\epsilon>0$  such that the injectivity radius $\mathrm{InjRad}_{M_n} (x)\geq\epsilon$  for all $x,n$.  Fixing $R>0$,
 \begin{align*}
	\lim_{n\to\infty}\frac{1}{\mathrm{vol} (M_n)} \int_{M_n} \mathrm{InjRad}_{M_n} (x)^{- d} dx & \leq \lim_{n\to\infty} \left (R^{-d} +
	\frac{\mathrm{vol}(M_n)_{\leq R}}{\mathrm{vol} (M_n)} \epsilon^{-d} \right ) = R^{-d},
 \end{align*}
by integrating separately over $(M_n)_{\leq R}$ and its complement, and then using BS-convergence. Sending $R\to \infty$ proves (2).

\vspace{2mm}

\noindent $(2)\implies (3)$. 
Since $\gb^{(2)}_k (X)=\lim_{t\to\infty}\text{Tr}(e^{-t\gD_k^{(2)}}(\ti x,\ti x))$, we may fix an arbitrarily small $\nu>0$ and $t$ large enough so that 
$$
 \gb^{(2)}_k(X)<\text{Tr}(e^{-t\gD_k^{(2)}}(\ti x,\ti x))+\nu.
$$
Then since $b_k(M)\leq  \int_M \mathrm{tr} \ e^{-t \Delta_k}(x,x)dx$  for each fixed $t$, we have

\begin{equation*}
\begin{split}
\frac{b_k(M_n )}{\vol(M_n )}-\gb^{(2)}_k(X) & \le \frac{1}{\vol(M_n )}\int_{M_n } \mathrm{tr} \, e^{-t\gD_k}(x,x)-\gb^{(2)}_k(X) dx \\
& \le\frac{1}{\vol(M_n )}\int_{M_n} f_t^n(x)dx+\nu\\
\end{split}
\end{equation*}
Now it follows from the hypothesis of the proposition and Lemma \ref{lem:hk} that for $n$ large enough, the right hand side is less than $2\nu$, so $(3)$ follows.

\vspace{2mm}

\noindent $(3)\implies (4)$. Unless $\dim(X)$ is even, and $k=\dim(X)/2$, the equality in $(4)$ is automatic since $\gb^{(2)}_k(X)=0$.  The equality when $k=\dim(X)/2$ follows since the Euler characteristic of $M_n$  is the same as its $L^2$-analogue.\end {proof}

\subsection{Heat kernel estimates in rank one}\label{subsection:notations}
We now establish some preliminary estimates on the difference $f_t(x)$ between the heat kernel on a rank one  locally symmetric space $M$ and the $L^2$ heat kernel in the universal cover.  While we will apply these estimates only to real hyperbolic manifolds, we write them up more generally here, since we anticipate they will be useful in the future and the proof is not any simpler for $X=\BH^n$.

Let $G=\mathbb{G}(\BR)$ be a connected adjoint simple real algebraic group of real rank one. 
We fix a Cartan decomposition $\mathfrak{g}=\mathfrak{k}\oplus\mathfrak{p}$ of $\mathfrak{g}=\text{Lie}(G)$ and let $K\le G$ be the maximal compact subgroup of $G$ corresponding to $\mathfrak{k}$. 

Let $\ti x_0\in X=G/K$ be the point corresponding to $K$.
Recall that $\mathfrak{p}$ is identified with the tangent space $T_{\ti x_0}X$ and the Killing form on $G$ induces an inner product on $\mathfrak{p}$ which determines the Riemannian structure on $X$.
Fix an $\text{Ad}(K)$-invariant inner product on $\mathfrak{k}$ and extend it to an $\text{Ad}(K)$-invariant inner product on $\mathfrak{g}$ so that $\mathfrak{k}$ and $\mathfrak{p}$ are orthogonal. 
 Finally, let $d=\dim(X)$ and let $s=\text{rank}_\BC(\mathbb{G})$ be the complex rank of $\mathbb{G}$, e.g. if $G=\SO(d,1)$ then $s=[\frac{d+1}{2}]$.

\medskip

We wish to establish estimates on $f_t$ within the `thin parts' of an $X$-manifold, i.e.\ parts where the injectivity radius is small. The geometry of thin parts is controlled by the classical Margulis lemma:

\begin{thm}\label{lem:margulis}(\cite[Section 4.1]{Th})
There is a constant $\epsilon=\epsilon(X)>0$ such that if $\Lambda$ is a discrete torsion free subgroup of $G$ consisting of semi-simple elements and is generated by
$\{ \gamma\in\Lambda : d(\gamma\cdot \ti x,\ti x)<\epsilon\}$ for some $\ti x\in X$, then $\Lambda$ is cyclic. Moreover there is a unique geodesic, the axis of $\Lambda$, on which it acts by translations.
\end{thm}

An important consequence of the Margulis lemma is the thick-thin decomposition which, in our case, says that the thin part
$$
 M_{<\epsilon}=\{x\in M:~\text{InjRad}_M(x)<\frac{\epsilon}{2}\}
$$
consists of finitely many connected components, each of which is a tubular neighbourhood of a short closed geodesic.

For $\ti x\in X$, we shall denote by $\gS_{\ti x,\gep}$ the set of elements in $\gC$ that move $\ti x$ by less than $\gep$, and by $\gC_{\ti x,\gep}=\langle\gS_{\ti x,\gep}\rangle$ the cyclic group it generates.

\medskip

The following Proposition gives an estimate on $f_t$ in terms of the number of $\Gamma$-orbit points in a ball. It is easily deduced from the proof of Lemma \ref{lem:hk}, more precisely from both Lemma \ref{L:heat} and the fact that $X$ has exponential growth.

\begin{prop}\label{prop:D}
Given $r>0$ there is $D=D(r,t)$ such that for any $x\in X$,
$$
 f_t(x)\le D \cdot \text{card}(\gC\cdot x\cap B(x,r)).
$$
\end{prop}

In view of Proposition \ref{prop:D}, our goal is to estimate the number of orbit points in a given ball and deduce bounds on $f_t$. We will split this into two estimates: one which is better close to the geodesic core of the thin part, and one which is better at points far from the core.

\begin{lem}[Near the core]\label{lem:0.2}
Let $x\in M$ be a point in an ${\gep}$-thin tubular neighbourhood of a short geodesic, and suppose that the length of that short geodesic is $\gt$. Then $f_t(x)\le C_1\gt^{-1}$, for some constant $C_1=C_1(X,t)$.
\end{lem}

\begin{proof}
Let $x$ be such a point in $M$ and let $\ti x$ be a lift of $x$ to $X$. In view of Proposition \ref{prop:D} we should obtain an upper bound of the form $\text{const}\cdot\tau^{-1}$ on the cardinality of the set
$$
 \mathcal{E}=\Gamma\cdot\ti x \cap B(\ti x,\epsilon)=\Gamma_{\ti x,\gep}\cdot\ti x \cap B(\ti x,\epsilon).
$$
Let $c$ be the axis of $\Gamma_{\ti x,\gep}$ and  $\pi_c:X\to c$ be the nearest point projection. Since $c$ is convex and $X$ is non-positively curved, $\pi_c$ is $1$-Lipschitz. Since $\Gamma_{\ti x,\gep}$ is torsion free and stabilizes $c$, it follows that the restriction of $\pi_c$ to a $\Gamma_{\ti x,\gep}$-orbit is one to one and its image is again a $\Gamma_{\ti x,\gep}$-orbit. Moreover, since $\mathcal{E}$ has diameter $\le 2\epsilon$ we deduce that $\pi_c(\mathcal{E})$ is contained in an interval of length $2\epsilon$ in $c$. Thus $\text{Card}(\mathcal{E})\le\frac{2\epsilon}{\tau}$.
\end{proof}

Lemma \ref{lem:0.2} gives a sufficiently good bound on $f_t(x)$ when $x$ is close to a short geodesic. However when $x$ is far from the geodesic, the injectivity radius
$\mathrm{InjRad}_{M} (x)$ might be of several magnitude larger than the minimal displacement $\tau$, and the result of \ref{lem:0.2} will not be enough for our purpose, so we should obtain a better estimate in terms of $\mathrm{InjRad}_M (x)$.
At first glance one may expect that the number of orbit points in a ball is controlled by $\mathrm{InjRad}_M (x)^{-1}$ (or by $\mathrm{InjRad}_M (x)^{-r}$ in general when $r=\text{rank}_\BR(X)$). However the rotational parts of the isometries may make the orbit denser at certain distances from the submanifold of local minimal displacement. The true exponent is the absolute rank $s$:

\begin{lem}[Far from the core]\label{lem:power{-s}}
If $x$ lies in the $\epsilon$-thin part of $M$, then 
$$
 f_t(x)\le C_2\mathrm{InjRad}_M (x)^{-s}
$$ 
for some constant $C_2=C_2(X,t)$.
\end{lem}

\begin{proof}
Let $\gd>0$ be sufficiently small so that for
$$
 U_{\ti x_0}=\exp(\{ X\in \mathfrak{g}:\| X\|\le\delta\}
$$
we have that $U_{\ti x_0}^2$ forms a Zassenhaus neighbourhood in $G$ (see \cite[Chapter XI]{Raghunathan} and \cite[Section 4.1]{Th}). Here, $\ti x_0\in X$ is the point corresponding to $K\in G/K=X$. 

We shall call $U_{\ti x_0}$ the Zassenhaus neighbourhood associated to $\ti x_0$. Since $G$ acts transitively, for any $\ti x\in X$ we have some $g\in G$ such that $g\cdot \ti x_0=\ti x$. Set $U_{\ti x}=gU_{\ti x_0}g^{-1}$ and $U_{\ti x}^2$ as the Zassenhaus neighbourhood associated to $\ti x$. Since $U_{\ti x_0}$ is invariant under conjugation by $K$, $U_{\ti x}$ is well defined.

The orbit map $X\mapsto \exp(X)\cdot \ti x_0$ restricted to $\{ X\in \mathfrak{g}:\| X\|\le\delta\}$ is $\alpha$-bi-Lipschitz for some constant $\alpha$ and covers an open ball $B_X(\ti x_0,\beta)$ for some $1\ge\beta>0$. It follows that if $V_1,\ldots,V_t\in \mathfrak{g}$ are of norm at most $\delta$ and $\{\exp(V_1)\cdot \ti x_0,\ldots,\exp(V_t)\cdot \ti x_0\}$ forms a $\rho$-discrete subset of $X$ then $\{V_1,\ldots,V_t\}$ is $\frac{\rho}{\alpha}$ discrete in $\mathfrak{g}$.

Let now $x\in M_{\le\epsilon}$ be the point in question and let $\ti x\in X$ be a lift of $x$. We may suppose that $\mathrm{InjRad}_M (x)<\beta$. Let
$$
 m=\frac{\mu(U_{\ti x}\cdot\{g\in G:d(g\cdot \ti x,\ti x)\le 1\})}{\mu (U_{\ti x})}+1.
$$
Note that $m$ is independent of $\ti x$.
In the proof of the Margulis Lemma given in \cite[Section 4.1]{Th} it is shown that the Margulis constant $\gep$ can be chosen to be $1/m$ or smaller. Since we have defined $m$ and $\gb$ independently of $\gep$, we may assume that $\gep\le \frac{\gb}{2m}$. In that case, as follows easily from the argument in \cite[Section 4.1]{Th}, $N=\langle U_{\ti x}^2\cap\gC_{\ti x,\gep}\rangle$ is a subgroup of index $\le m$ in $\gC_{\ti x,\gep}$ and one can choose coset representatives within $\gS_{\ti x,\gep}^m$.
In particular it follows that
$$
 \text{card}\left(\gC\cdot \ti x\cap B\left(\ti x,\frac{\beta}{2}\right)\right)\le m\cdot \text{card}(N\cdot \ti x\cap B(\ti x,\beta)).
 $$
Moreover by the Zassenhaus--Kazhdan-Margulis theorem (see \cite[Chapter XI]{Raghunathan}) $\log N$ spans a connected nilpotent Lie sub-algebra $\mathfrak{n}$ of the Lie algebra of the stabilizer $\text{Stab}_G(c)$ where $c$ is the axis of $\gC_{\ti x,\gep}$. Note that $\text{Stab}_G(c)$ is isomorphic to a compact group times $\BR^*$ and hence admits no unipotent elements. It follows that $\mathfrak{n}$ is abelian and semi-simple and its exponentiation $\exp(\mathfrak{n})$ is a torus in $G$. In particular $\dim \mathfrak{n}\le s$.
Finally since $N\cdot \ti x\cap B(\ti x,\beta)$ is $\mathrm{InjRad}_M (x)/2$ discrete we get that $\log (N)$ is $\mathrm{InjRad}_M (x)/(2\alpha)$ discrete in $\mathfrak{n}$. Thus
$$
 \text{card}(N\cdot \ti x\cap B(\ti x,\beta))\le\text{card}(\log (N)\cap B_\mathfrak{g}(0,\delta))\le C' \left(\frac{\mathrm{InjRad}_M (x)}{\alpha}\right)^s,
$$
and the result follows from Proposition \ref{prop:D}.
\end{proof}

\subsection{Real hyperbolic manifolds} 

Given a symmetric space $X$ of non-compact type, Margulis has conjectured that the set of arithmetic compact $X$-manifolds is uniformly discrete. If $\text{rank}(X)\ge 2$ or if $X$ is the symmetric space corresponding to $\Sp(d,1)$ or $F_4^{-20}$, then all irreducible $X$-manifolds are arithmetic, by Margulis's Arithmeticity Theorem \cite{Mar-arithmeticity} and the Corlette--Gromov--Schoen Theorem~\cite{Corlette,Gr-Sc}, respectively. For $\SU(d,1)$ there are a few known examples of non-arithmetic manifolds for $d=2,3$, and it is likely that most manifolds are arithmetic. It is thus natural to conjecture that if $X$ is not isometric to $\BH^d$, then the family of all irreducible compact $X$-manifolds is uniformly discrete. 

For $X=\BH^d$, it is known that for all $d$, there are closed hyperbolic $d$-manifolds with arbitrarily short systoles, \cite{Agol,BHW,BT}. Our aim here is to prove a strong generalization of Corollary \ref{main} for $\BH^d$, not assuming uniform discreteness.

\begin{thm}\label{thm:rank-one}
Let $M_n=\Gamma_n\backslash \BH^d$ be a sequence of compact hyperbolic $d$-manifolds that BS-converges to $\BH^d$. Then for every $k=0,\ldots,d$,
$$
 \lim_{n \to +\infty} \frac{b_k(M_n)}{\vol(M_n)}=\beta^{(2)}_k(\BH^d).
$$
\end{thm}

\begin{rem}
The analog of Theorem \ref{thm:rank-one} holds in the greater generality where $\BH^d$ is replaced by a general rank one symmetric space. The proof of that however uses different techniques and is much longer. This result will appear in \cite{ABBG} where we will also treat higher rank symmetric spaces. 
\end{rem}

Note that for $X=\BH^2$, the hyperbolic plane, Theorem \ref{thm:rank-one} is a consequence of the Gauss--Bonnet theorem, even under the weak assumption that only $\vol(M_n)\to\infty$, without requiring BS-convergence.   The cases $d=3$ and $d\geq 4$ will be handled separately. When $d\geq 4$, Theorem \ref{thm:rank-one} will follow from a fine analysis of the heat kernel in the thin parts of the $M_n$, which is of independent interest. In dimension $3$, the analogous statements about the heat kernel are not true, as we will explain, but we can use a trick to reduce the calculation of Betti numbers to  estimates  on the heat kernel only over the thick part of $M_n$.

We start with the case $d\geq 4.$ Fix $k\in\{1,\ldots,d\}$, and let $f_t$ be the function defined in \ref{eq:f_t}. In view of Proposition \ref{Pcvhk}, we need to prove: 
\begin{equation}
\label{prop:0.4} \lim_n \frac{1}{\vol(M_n )}\int_{M_n } f_t(x) \, dx\to 0.
\end{equation}

The crucial technical tool is the following.

\begin{thm}\label{thm:f-bounded-by-vol}
Given $d\geq 4$ and $t>0$, there is a constant $C=C(d,t)$ such that
$$
 \int_{M_{\le\gep/2}}f_t(x)\, dx\le C_0\cdot \vol(M_{\le\gep}),
$$
for every compact hyperbolic $d$-manifold.
\end{thm}

Assuming Theorem \ref{thm:f-bounded-by-vol}, let us prove \eqref{prop:0.4}. First, note that
\begin{align}
	 \lim_n \frac{\int_{M_n } f_t(x) \, dx}{\vol(M_n ) }&= \lim_n \frac{\int_{(M_n)_{\geq \epsilon} } f_t(x) \, dx}{\vol(M_n )} + \frac{\int_{(M_n)_{< \epsilon} } f_t(x) \, dx}{\vol({(M_n)_{< \epsilon} })} \cdot \frac{\vol({(M_n)_{< \epsilon} })}{\vol(M_n )}\nonumber\\
	 &\leq  \lim_n \frac{\int_{(M_n)_{\geq \epsilon} } c(t) \mathrm{InjRad}(x)^{-d} \, dx}{\vol(M_n )} + C_0 \cdot \lim_n\frac{\vol({(M_n)_{< \epsilon} })}{\vol(M_n )}\label{wow}
\end{align}
Here, $c(t) \mathrm{InjRad}(x)^{-d}$ comes from Lemma \ref{lem:hk} and the $C_0$ is from Theorem~\ref{thm:f-bounded-by-vol}.  On the far right, BS-convergence $M_n\to X$  implies that the limit is zero. So, for any fixed $r\geq \epsilon$,  splitting up the first term in \eqref{wow} gives the upper bound
\begin{align*}&\leq  \lim_n \frac{\int_{(M_n)_{\geq \epsilon} \cap (M_n)_{<r}} c(t) \epsilon^{-d} \, dx}{\vol(M_n )} +  \frac{\int_{(M_n)_{\geq r} } c(t) r^{-d} \, dx}{\vol(M_n )}\\
&\leq c(t) \epsilon^{-d}  \lim_n\frac{\vol({(M_n)_{<r} })}{\vol(M_n )} + c(t)r^{-d}\\
&=c(t)r^{-d},
\end{align*}
again by BS-convergence. Letting $r\to \infty$, this proves \eqref{prop:0.4}.

\subsection{The proof of Theorem \ref{thm:f-bounded-by-vol}}
We shall work in radial horospherical coordinates of the upper half space model of $\BH^d$
$$
 \{(x_1,\ldots,x_d\in\BR^d): x_d>0\},~ ds^2={\sum dx_i^2\over x_d^2}.
$$
Consider the vertical geodesic $c=(0,\infty)$ and the horizontal (intrinsically Euclidean) horosphere $\BE^{d-1}$ passing through $c$ at $p=(0,\ldots,0,1)$. We will consider the coordinates $(r,\theta)$ for points on $\BE^{d-1}$ where $r$ is the horospherical radial distance to $p$ and $\theta$ is the direction. (Note that the hyperbolic distance of the point $(r,\theta)$ to $p$ is roughly $\log r$.) We can extend these coordinates to the upper half space, letting $(r,\theta,a)$ denote the point $a\cdot x$ where $x$ is the point on $\BE^{d-1}$ of coordinate $(r,\theta)$ and $a$ is the isometric homothety corresponding to a multiplication by $a>0$ in $\BR^n$.
Let $G_c$ be the stabilizer of $c$ in $G$, $G_c\cong\SO(d-1)\times\BR^{>0}$.

\begin{lem}\label{lem:1.01}
There are $R<\infty$ and $\ga>1$ such that if $r_1,r_2>R$ then for any two points points $x_1=(r_1,\theta),~x_2=(r_2,\theta)$ at the same direction $\theta$ and any $g\in G_c$ for which $d_g(x_1),d_g(x_2)\le\gep$ we have:
$$
\ga^{-1}\frac{r_1}{r_2}<\frac{d_g(x_1)}{d_g(x_2)}<\ga\frac{r_1}{r_2}.
$$
\end{lem}

\begin{proof}
Since the points $x_1,x_2$ are far from the invariant geodesic $c=(0,\infty)$ and have small $g$-displacement, the distances $d(g\cdot x_i,x_i)$ are approximated, up to a bounded multiplicative error, by the intrinsic Euclidean distance between the Euclidean projections of $g\cdot x_i$ and $x_i$ to the horosphere $\BE^{d-1}$. For the projections (considered with the intrinsic distance) however the ratio in question is equal to $\frac{r_1}{r_2}$ by similarity of Euclidean triangles.
\end{proof}

Let now $M_{\le\gep}^\circ$ be a thin component which is a tubular neighbourhood of a short geodesic, and let $\ti M_{\le\gep}^\circ$ be a connected component of its pre-image in the upper half space. We may suppose that the short geodesic lifts to $c=(0,\infty)$. Suppose that the length of the short geodesic is $\gt$. Note that $G_c=N_G(G_c)$ and hence it follows from the Margulis Lemma that $\gC_{p,\gep}$ is contained in $G_c$. Choose a fundamental domain for $\gC_{p,\gep}$ in $\ti M_{\le\gep}^\circ$ of the following form:
$$
\mathcal{F}=\{ (r,\theta,a): r\le\psi(\theta),~1\le a<e^\gt\}
$$
where $\psi(\theta)$ is defined to be the radial horospherical distance for which at direction $\theta$ the minimal displacement is exactly $\gep$, i.e.
$$
 \min\{d_\gc(x):\gc\in\gC_{p,\gep}\setminus\{1\}\}=\min\{d_\gc(x):\gc\in\gC\setminus\{1\}\}=\gep
$$
for $x\in \BE^{d-1}$ of coordinates $(\psi(\theta),\theta)$.

\begin{lem}\label{lem:tao_0}
Given $R>0$, there is some $\ti \tau(R)>0$ such that if $\gt\le\ti \tau$ then $$\psi(\theta)>R \ \ \forall \theta.$$
\end{lem}

\begin{proof}
Let $\alpha>0$ be sufficiently small so that any two horocylic rays $r_1(t)=(t,\theta_1,a)$ and $r_2(t)=(t,\theta_2,a)$ starting at an angle $\le \alpha$ stay at distance $\le \gep/2$ from each other when $t\le R$.
Since $SO(d-1)$ is compact there is some $l\in\BN$ such that for any $o\in SO(d-1)$ there is $j=j(o)\le l$ such that $\angle (o^j(\hat v),\hat v)<{\alpha}$ for every $\hat v\in \BR^{d-1}$. Let $\lambda>0$ be small enough so that any two horocyclic rays orthogonal to $c$ that start parallel to each other at distance $\le \lambda$, stay at distance $\le \gep/2$ for $t\le T$.

Take $\gt_T={\lambda \over l}$. If $g\in G_c$ is any isometry with displacement $\gt\le\ti \gt$ and rotational part $o$, it is easy to see that $g^{j(o)}$ has translational part $\le \lambda$ on $c$ and rotational part $\le \alpha$. Thus its displacement is $\le \gep$ everywhere on the $R$ neighbourhood of $c$.
\end{proof}

We may fix $R>0$ and assume $\gt\le \ti\gt(R)$. To estimate the integral of $f_t(x)$ over $\mathcal{F}$, we divide the domain into two parts, $\mathcal{F}_1=\{ 0\le r\le R\}$ and $\mathcal{F}_2=\{ R< r<\psi(\theta)\}$:

$$
 \int_{\mathcal{F}} f_t(x)dx=\int_{\mathcal{F}_1} f_t(x)dx + \int_{\mathcal{F}_2} f_t(x)dx.
$$
The first integral can be bounded using Lemma \ref{lem:0.2}:
$$
 \int_{\mathcal{F}_1} f_t(x)dx\le \vol(\mathcal{F}_1)\cdot C_1\gt^{-1}\le\gt\cdot \vol(B^{d-1}(R))\cdot C_1\gt^{-1} = \vol(B^{d-1}(R))\cdot C_1
$$
where $B^{d-1}(R)$ is an Euclidian $(d-1)$-ball of radius $R$. So the first integral is bounded by a constant. Recall that the volume of each thin component is bounded below by a constant since one can inject an $\gep\over 2$ ball tangent to the boundary of the component.

Let us estimate the second integral. Note that by Lemma \ref{lem:1.01} the $\gC_{p,\gep}$ minimal displacement at $(r,\theta)$ for $r>R$ is at least $\ga^{-1}\frac{\gep r}{\psi(\theta)}$. Therefore using Lemma \ref{lem:power{-s}} we deduce
\begin{equation*}
\begin{split}
\int_{\mathcal{F}_2} f_t(x)dx & \le C_2\ga^{s}\int_{\theta\in \mathbb{S}^{d-2}}  \int_R^{\psi(\theta)}\big(\frac{\gep\cdot r}{\psi(\theta)}\big)^{-s}\cdot \gt\cdot r^{d-2}dr  d\theta \\
& \le\text{Const}\cdot\gt\int_{S^{d-2}} \left( \psi(\theta)^{s} \int_0^{\psi(\theta)}r^{d-s-2}dr\right) d\theta.
\end{split}
\end{equation*}
Here $s = \left[ \frac{d+1}{2} \right]$ and since $d\geq 4$ we have $d-s-2\geq 0$. It follows that 
$$\int_{\mathcal{F}_2} f_t(x)dx  \le \text{Const}\int_{\mathbb{S}^{d-2}}\gt\cdot\psi(\theta)\cdot\psi(\theta)^{d-2}d\theta.
$$
The point is that the last term is, up to a constant, the volume of the thin component.
This concludes the proof of Theorem \ref{thm:f-bounded-by-vol}.
\qed

\subsection{The case $d=3$}
Equation \eqref{prop:0.4} is false when $d=3$, essentially since $f_t$ has infinite integral when $M$ has a cusp, and cusped manifolds can be approximated by closed manifolds using hyperbolic Dehn surgery.

 To discuss this in more detail, suppose that $M$ is a finite volume hyperbolic $d$-manifold.  When $M$ is noncompact, the heat operators on $k$-forms $e^{-t\Delta_k}$ are not of trace class. In fact, following \cite[Equation (3.3)]{MP}, we may endow $M$ with a height function in the cusps. For $Y$ big enough the truncation $M(Y)$ of $M$ at height $Y$ is diffeomorphic to the so-called Borel--Serre compactification of $M$. Fixing $t$ it is a consequence of the Selberg trace formula (see e.g. \cite{Friedman} for the case of functions and \cite[Equation (5.5)]{MP} for the general case) that 
$$\int_{M(Y)} \mathrm{tr} \, e^{-t \Delta_k } (x,x) dx \, \sim \, k_0 \log Y + c .$$
Here, the notation $A(Y) \sim B(Y)$ means that $A(Y)-B(Y) \to 0$ as $Y \to +\infty$, and $k_0$ and $c$ are positive constants that depends on $t$. (In the case of $0$-forms $k_0 = \frac{h}{2\pi} \int_{0}^{+\infty} e^{-t (1+s^2)} ds$, where $h$ is the number of cusps.) 

In particular, for $Y$ big enough we have:
$$
\int_{M(Y )} \mathrm{tr} \, e^{-t \Delta_k } (x,x) dx \geq 2\vol (M).
$$
When $M$ has dimension $d=3$, hyperbolic Dehn surgery constructs from $M$ a closed hyperbolic manifold $M'$ so that $M_{\leq Y}$ is almost isometrically embedded inside $M'$ and $\vol (M')$ is close to $\vol (M)$. In particular, we may construct $M'$ so that
$$
\int_{M'} \mathrm{tr} \, e^{-t \Delta_k } (x,x) dx \geq \vol (M').
$$
Now take $k=1$. Starting from a sequence of finite volume, noncompact, hyperbolic manifolds that BS-converges toward $\BH^d$, the construction above yields a sequence of closed hyperbolic manifolds $(M_n)$ such that 
$$\int_{M_n } \mathrm{tr} \, e^{-t \Delta_1 } (x,x) dx \geq \vol (M_n ).$$
On the other hand the integral $\frac{1}{\vol (M_n )} \int_{M_n} \mathrm{tr} e^{-t \Delta_1^{(2)} } (x,x) dx$ is finite, bounded uniformly in $n$ and approaches $\beta_1^{(2)} (\BH^3 )= 0$ as $t$ tends to infinity. In particular for $t$ small enough we may assume that 
$$\int_{M_n } \mathrm{tr} \, e^{-t \Delta_1^{(2)} } (x,x) dx \leq \frac12 \vol (M_n ).$$
And it follows that \eqref{prop:0.4} cannot hold (when $d=3$).

\medskip

We now prove Theorem \ref{thm:rank-one} when $d=3$.  Suppose that $(M_n)$ is a sequence of finite volume\footnote{This argument even works for nonuniform lattices, while the estimates in the $d\geq 4$ case are just for uniform lattices.} hyperbolic $3$-manifolds that BS-converges to $\BH^3$.    In light of Proposition~\ref{Pcvhk}, we need to show that \footnote{Note that the right hand side below is $0$, but we will not make use of that.}
\begin{equation} \label{newlim}
\lim_{n \to +\infty} \frac{b_1(M_n )}{\vol(M_n )}\le\gb^{(2)}_1(\BH^3).
\end{equation}

Fix $\epsilon$ less than the Margulis constant.  When $M $ is a finite volume hyperbolic $3$-manifold, we let $M_T$ be the union of the $\epsilon$-thick part of $M$ and any components of the $\epsilon$-thin part on which the injectivity radius is bounded below by $\epsilon/2$.

\begin{lem}\label{bddg}
$M_T$ is a closed, $3$-dimensional submanifold whose boundary consists of tori or Klein bottles smoothly embedded in $M$, and the components of $M\setminus M_T$ are either solid tori or solid Klein bottles (i.e.\ disk bundles over a circle) or are products $T^2 \times (0,\infty)$ or $K^2 \times (0,\infty)$. Furthermore, $M_T$ has `bounded geometry', in the sense of  \cite[Definition 2.24]{Luckschick}.
\end{lem}

 As we will see below, `bounded geometry' requires that the boundary of $M_T$ is not too distorted in $M$, which is why we take $M_T$ instead of just the $\epsilon$-thick part.

\begin{proof}
In dimension $3$, the Margulis lemma implies that each component of the thin part $M\setminus M_T$ is the quotient of either a metric neighborhood of a geodesic in $\BH^3$ or of a horoball; this implies that the boundary is smooth, and gives the topological information above.\footnote{Isometries of $\BH^3$ that translate along an axis $c$ are compositions of pure translations and $2$-dimensional rotations in the orthogonal direction. So for a given $r>0$, a loxodromic isometry of $\BH^3$ with geodesic axis $c$ acts with the same translation distance on every point of the boundary $\partial N_r(c)$ of the $r$-neighborhood around $c$. This is not true in higher dimensions, since the rotational part of an isometry can be more complicated, and in fact the components of $M\setminus M_T$ that are (non-metric) neighborhoods of closed geodesics may not have smooth boundary.} See also \cite{Benedetti-Petronio} for details.

In \cite[Definition 2.24]{Luckschick}, `bounded geometry' means the following. First, the injectivity radius of $M_T$ should be bounded below, which is true by definition. Second, the geodesic flow starting from the inward normal vector field on $\partial M_T$ should give a collar neighborhood of the boundary with radius bounded below:  this follows since the injectivity radius of $M_T$  is bounded below and since the  components of the preimage of $M\setminus M_T$ in $\BH^3$ are convex (this again is a $3$ dimensional phenomenon and is false in higher dimension). Finally, the derivatives of the metric tensor and its inverse should be bounded, both in exponential coordinates and the `boundary normal coordinates' on the collar of $\partial M_T$ above.  In exponential coordinates, the bounds come from differentiating the metric tensor on $\BH^3$, while in boundary normal coordinates, one uses that the second fundamental form of $\partial M_T \subset M$ has bounded derivatives, as it is the quotient of a horosphere or of a metric neighborhood of a geodesic with radius bounded below by $\epsilon/4$. \end{proof}

For all four topological types of components of $M\setminus M_T$, the first cohomology of the boundary surjects, so using Mayer--Vietoris sequence we see that 
$$b_1(M)\leq b_1(M_T).$$
Therefore, to prove \eqref{newlim} it suffices to estimate the Betti numbers of $(M_n)_T$.

Let $\Delta$ be the Laplacian operator on differential $1$-forms on $M$, and $e^{-t\Delta}(x,x)$ the corresponding heat kernel. We also let $\Delta_1^T$ be the Laplacian operator on differential $1$-forms on $M_T$  with absolute boundary conditions, and denote by $e^{-t\Delta_1^T}(x,x)$ its integral kernel. It follows from \cite[Theorem 6.1]{Donnellylower} that 
\begin{equation}\label{disdat}b_1(M_T)=\lim_{t\to\infty} \mathrm{Tr} \, e^{-t\Delta_1^T} = \lim_{t\to\infty} \int_{M_{\geq \eps}} tr \, e^{-t\Delta_1^T}(x,x) \, dx;\end{equation}
note that since $\mathrm{Tr} \, e^{-t\Delta_1} = \sum_i e^{-t\lambda_i}$, where $\lambda_i$ are the eigenvalues of $\Delta_1^T$, the expression above is monotone decreasing in $t$, so the limit exists. 

Recall that $\gb^{(2)}_1(\BH^3)=\lim_{t\to\infty}  \text{tr} \, e^{-t\gD_1^{(2)}}( \ti x, \ti x),$ where $e^{-t\gD_1^{(2)}}( \ti x, \ti x)$  is the $L^2$-heat kernel of $\BH^3$. In light of \eqref{disdat}, it suffices to fix $t>0$ and show that
\begin{equation} \label{newlim2}
\limsup_{n\to \infty} \frac 1{\vol{(M_n)}} \int_{(M_n)_{T}} \mathrm{tr} \, e^{-t\Delta_1^T}(x,x) \, dx \leq \text{tr} \, e^{-t\gD_1^{(2)}}( \ti x, \ti x),
\end{equation} 
 for some (arbitrary) $\ti x \in \BH^3$. Since then, taking $t\to \infty$ proves \eqref{newlim}. 

Fix some large $R\gg 1>\epsilon$, and consider the subset $(M_n)_{\geq R} \subset (M_n)_{T}$. Then the boundary of $(M_n)_{T}$ is uniformly far from $(M_n)_{\geq R} $, and by a theorem of L\"uck and Schick \cite[Theorem 2.26]{Luckschick} we have that for all $x\in (M_n)_{\geq R} $ 
$$|| e^{-t\Delta_1^T}(x,x) -e^{-t\Delta_1}(x,x) || \leq C(t,R),$$
where $C(t,R)\to 0$ as $R \to \infty.$ (Note: although their statement assumes that $M_n$ has bounded geometry,  which in this case means the global injectivity radius $\mathrm{InjRad} M_n >0$, it suffices in their proof to assume a lower injectivity radius bound on $ (M_n)_{\geq R} $, which is automatic.) So, by Lemma \ref{lem:hk}, for all $x\in (M_n)_{\geq R} $,
$$||e^{-t\Delta_1^T}(x,x) - e^{-t\Delta_1^{(2)}}( \ti x, \ti x) || \leq C(t,R) + C(t) R^{-3}=C'(t,R),$$
where again $C'(t,R)\to 0$ as $R\to \infty$.
Hence, for all $n$, the average value
\begin{equation}\frac 1{\vol{(M_n)_{\geq R}}} \int_{(M_n)_{\geq R}} \mathrm{tr} \, e^{-t\Delta_1^T}(x,x) \, dx \leq 
\text{tr} \, e^{-t\gD_1^{(2)}}( \ti x, \ti x) + C'(t,R). \label{largeR}\end{equation}
Next, if $x\in D_n=(M_n)_{T} \setminus (M_n)_{\geq R}$, we have by \cite[Theorem 2.35]{Luckschick} that
$$||e^{-t\Delta_1^T}(x,x) || \leq C(t),$$
since by Lemma \ref{bddg} the manifold with boundary $(M_n)_{T}$ has bounded geometry, in the sense of \cite[Definition 2.24]{Luckschick}. So, we also have the average value
\begin{equation}\frac 1{\vol{D_n}} \int_{D_n} \mathrm{tr} \, e^{-t\Delta_1^T}(x,x) \, dx \leq 
C(t). \label{smallR}\end{equation}

Combining \eqref{largeR} and \eqref{smallR}, we obtain that for all $n$:
\begin{align*} \label{newlim2}
& \ \  \limsup_{n\to \infty} \frac 1{\vol{(M_n)_{T}}} \int_{(M_n)_{T}} \mathrm{tr} \, e^{-t\Delta_1^T}(x,x) \, dx \\  & \leq  \frac{\vol \, (M_n)_{\geq R}}{\vol \, M_n}  \left (
\text{tr} \, e^{-t\gD_1^{(2)}}( \ti x, \ti x) + C'(t,R) \right )+ \frac{\vol \, D_n}{\vol M_n} C(t) \\  & \leq 
\text{tr} \, e^{-t\gD_1^{(2)}}( \ti x, \ti x) + C'(t,R) + \frac{\vol \, D_n}{\vol M_n} C(t).
\end{align*} 
For a fixed $R$, by letting first $n\to\infty$ we deduce from BS convergence that the last term vanishes in the limit. Finally, by sending $R \to \infty$, the term $C'(t,R)$ disappears.

\qed
\section{Growth of torsion} \label{sec:torsion}

In this last section we consider only those $X = G/K$ for which all $\beta_k^{(2)} (X)$ vanish. It is then natural to consider the secondary invariant given by the $L^2$-torsion. We first review its definition and then consider the corresponding approximation problems. We continue with the notations of the preceeding sections. In particular we let $\Gamma$ be a cocompact torsion-free subgroup of $G$ and let $M = \Gamma \backslash X$.

\subsection{$L^2$- and analytic torsion} 
We will work in the setting of \cite{BV}: we will be here as brief as possible concerning definitions, etc. and refer to that paper for all details. 

Given a finite-dimensional representation $\rho$ of $G_\C$ on a vector space $E$ one can construct a canonical $G$-equivariant Hermitian bundle $E_\rho$ on $X$ with fiber $E$. The space of square-integrable $k$-forms with coefficients in $E_\rho$ is then endowed with a Laplacian $\Delta_k^{(2)}$ and associated heat kernels $e^{-t\Delta_k^{(2)}(\rho)}$ which are bounded operators given by convolution with a $G$-equivariant kernel $e^{-t\Delta_k^{(2)}(\rho)}(x,y)$ (a section of a bundle over $X\times X$). The trace $\Tr e^{-t\Delta_k^{(2)}(\rho)}(x,x)$ does not depend on $x\in X$. Let $\Gamma(s)$ denote the Euler Gamma-function ; the determinant $\det\Delta_k^{(2)}$ is then defined by : 
$$
\log\det\Delta_k^{(2)}  = \frac{d}{ds} \Big|_{s=0} \left( \frac{1}{\Gamma (s)} \int_0^1 t^{s-1}  \Tr e^{-t \Delta_k^{(2)}}(x,x) dt \right)  + \int_1^{+\infty} t^{-1} \Tr e^{-t \Delta_k^{(2)}}(x,x) dt.
$$
(see \cite[Definition 3.128]{LuckBook} for a justification) and the $L^2$-torsion $t_X^{(2)}(\rho)$ by : 
\begin{equation} \label{L2analyticTorsion}
t_X^{(2)} (\rho) = \frac 1 2 \sum_{k \geq 0} (-1)^k k \log\det\Delta_k^{(2)}. 
\end{equation}

The bundle $E_\rho$ descends to a bundle $V$ on $M$, with Laplacians $\Delta_k$ and heat kernels $e^{-t\Delta_k}$ ; similar to the $L^2$-case one can define determinants of the $\Delta_k$ and analytic torsion $T_M(\rho)$. We raise the following question/conjecture.

\begin{conj} \label{conjtors}
Let $(M_n )$ be a \emph{uniformly discrete} sequence of compact $X$-manifolds which BS-converges to $X$.
Then we have:
$$\frac{\log T_{M_n}(\rho)}{\vol (M_n)} \rightarrow t_X^{(2)}(\rho).$$
\end{conj}

We note that $t_X^{(2)}(\rho)$ is non-zero if and only if $\delta (G) =1$, i.e. if $G$ is one of the groups $\SL_2 (\C )$, $\SL_3(\R),  \SO_{n,m}, \ \mbox{$nm$ odd}$. In principle one can compute an explicit value of $t_X^{(2)}$ for all $G$ and $\rho$, see \cite[Section 5]{BV}. When $G=\SO_{2p+1 , 1}$ the space $X$ is the real hyperbolic space $\H^{2p+1}$. We have for trivial $\rho$ \cite[Theorem 3.152]{LuckBook}:
$$t_{\mathbb{H}^{3}}^{(2)} =  -\frac{1}{6\pi} , \ t_{\mathbb{H}^{5}}^{(2)} =  \frac{31}{45\pi^2}, \ \ldots$$

\subsection{Strongly acyclic coefficients}
The representation $\rho$ is said to be {\it strongly acyclic} if there is a constant $\eta$ such that for every cocompact $\Gamma\subset G$ and for every $k$, the spectrum of the Laplace operator $\Delta_k$ on $\Gamma \backslash X$ is contained in $[\eta,+\infty[$ (in particular this implies that $H^*(M;V) = 0$). When $\rho$ is strongly acyclic Conjecture \ref{conjtors} was proven for normal chains in \cite[Theorem 4.5]{BV}. The proof of loc. cit. adapts immediately to the setting of Benjamini--Schramm convergence, simply by replacing the main lemma there by Lemma \ref{lem:hk}. Thus we obtain : 

\begin{thm} \label{approxthm}
Assume that $\rho : G  \rightarrow \GL (E)$ is strongly acyclic. Let $(M_n )$ be a uniformly discrete sequence
of compact $X$-manifolds which BS-converges toward $X$.
Then we have:
$$\frac{\log (T_{M_n} (\rho))}{\vol(M_n)} \rightarrow t_X^{(2)} (\rho).$$
\end{thm}

\subsection{Example} Given any orientable compact hyperbolic $3$-manifold $M= \Gamma \backslash \H^3$ we can consider the discrete faithful $\SL_2 (\C)$-representation
$\alpha_{\rm can} : \Gamma \hookrightarrow \SL_2 (\C)$. It is strongly acyclic (see Example (3) of \cite[\S 5.9.3]{BV} with $(p,q)= (1,0)$). In particular: the corresponding twisted chain complex
$$C_* (\widetilde{M}) \otimes_{\Z [\Gamma ]} \C^2 $$
is acyclic and it follows that the corresponding Reidemeister torsion $\tau (M, \alpha_{\rm can}) \in \R^*$ is defined. According to the Cheeger-Mueller theorem extended to unimodular
representation by Mueller \cite{Mueller} we have $T_{M} (\rho) = |\tau (M, \alpha_{\rm can})|$ and Theorem \ref{thm7} follows from Theorem \ref{approxthm}.

\subsection{Torsion homology} 
In this (largely speculative) section we suppose that $\rho$ is trivial. According to the Cheeger-Mueller theorem \cite{Cheeger,Mueller} the analytic torsion $T_M$ decomposes as a product of
$$\prod_{k=0}^{\dim X} (-1)^{k+1} | H_k (M , \Z)_{\rm tors}|$$
by a so-called {\it regulator}, see \cite[eq. (2.2.4)]{BV}. This relates Conjecture \ref{conjtors} to the following question: let $M_n$ be a sequence of compact $X$-manifolds which BS-converges to $X$. Do we have for every $k \leq \dim (X)$:
$$\frac{\log |H_k (M_n , \mathbb{Z})_{\rm tors}|}{\vol (M_n)} \rightarrow \left\{
\begin{array}{ll}
|t_X^{(2)}| & \mbox{ if } k = \frac{\dim X -1}{2} \\
0 & \mbox{ otherwise } 
\end{array} \right.$$

To avoid discussing the growth of $H_k(M_n)$ for $k\not=(\dim(X)-1)/2$ here we will restrict to the case $X = \H^3$ so that $H_k(M_n)$ are torsion-free if $k \not=\frac{\dim X -1}2 = 1$ and $t_X^{(2)} = -(6\pi)^{-1}$. In this setting there are extensive computations by \c Sengun \cite{Sengun} for covers of a fixed manifold which suggest that the answer to the question above is negative, indicating that the contribution of the regulator to the limit in Conjecture \ref{conjtors} should be nonzero in general. However the same computations suggest that this is not the case when considering only congruence covers of an arithmetic manifold. See \cite{BSV} for a detailed discussion on regulators and the differences between congruence and non-congruence covers.

\medskip

The following result of Brock--Dunfield \cite{BrockDunfield} finally shows that Conjecture \ref{conjtors} cannot hold for general (non uniformly discrete) sequences. 

\begin{thm} \label{BDT}
There exists a sequence of hyperbolic {\rm integer} homology $3$-spheres which BS-converges toward the hyperbolic $3$-space.
\end{thm}

According to the Cheeger-M\"uller theorem, if $M$ is a homology sphere then $T_M = 1$. Thus the theorem above provides us with a sequence $M_n$ that converges to $\H^3$ in the Benjamini--Schramm sense but such that the conclusion of the conjecture is violated in an extreme way. 

\subsection{Knot exteriors}
Given a hyperbolic knot Dunfield, Friedl and Jackson \cite{DFJ} have introduced an invariant $\mathcal{T}_K (t) \in \C[t^{\pm1}]$ which is defined as the normalized twisted Alexander polynomial of $K$ corresponding to the discrete and faithful $\SL_2 (\C)$-representation of the knot group. It follows from \cite[Theorem 4]{Le} that the following holds: let $M_n$ be the $n$-th cyclic ramified cover of $\mathbb{S}^3$ along $K$, then for $n$ large enough $M_n$ is hyperbolic and
\begin{equation} \label{eq:ex2}
\lim_{n \rightarrow +\infty} \frac{1}{n} \log | \tau (M_n , \alpha_{\rm can}) | = - \log m ( \mathcal{T}_K ) ,
\end{equation}
where $m$ is the exponential Mahler measure. On the other hand Friedl and Jackson \cite{FriedlJackson} produce computations that suggest that $\log m ( \mathcal{T}_K )$
correlates strongly with $\mathrm{vol} (K)$: as $\mathrm{vol} (K)$ tends to infinity the ratio $\log m ( \mathcal{T}_K ) / \mathrm{vol} (K)$ seems to tend to a constant $\approx 0.29$.

Let $\overline{M}_n$ be the hyperbolic orbifold with underlying space $\mathbb{S}^3$ and $n$-th cyclic singularity along $K$. Then $M_n$ is a regular $n$-sheeted cover of
$\overline{M}_n$. Now recalling that $\overline{M}_n$ BS-converges toward $\mathbb{S}^3 - K$ (and in particular that  $\mathrm{vol} (\overline{M}_n ) \rightarrow \mathrm{vol} (K)$) as $n$ tends to infinity and that $11/12 \pi \approx 0.29$, in view of Theorem \ref{thm7} and equation \eqref{eq:ex2} it is natural to ask the following question (compare \cite{PurcellSouto}):

\medskip
\noindent
{\bf Question.} Let $(K_n)$ be a sequence of hyperbolic knots in ${\Bbb S}^3$ such that
$\vol(K_n) \rightarrow +\infty$. Can it happen that the sequence of finite volume hyperbolic manifolds
${\Bbb S}^3 - K_n$ BS-converge toward $\mathbb{H}^3$ ?

\bibliography{newbib}

\bibliographystyle{plain}

\end{document}